\documentclass[EJP, preprint]{ejpecp} 

\SHORTTITLE{Regularisation by fractional noise}

\TITLE{Regularisation by fractional noise for one-dimensional differential equations with distributional drift}

\AUTHORS{Lukas Anzeletti\footnote{Universit\'e Paris-Saclay, CentraleSup\'elec, MICS and CNRS FR-3487;~ \texttt{lukas.anzeletti@centralesupelec.fr}.} \and 
Alexandre Richard\footnote{Universit\'e Paris-Saclay, CentraleSup\'elec, MICS and CNRS FR-3487;~   \texttt{alexandre.richard@centralesupelec.fr}.} 
\and
Etienne Tanr\'e \footnote{Universit\'e C\^ote d'Azur, Inria, France; 
\texttt{Etienne.Tanre@inria.fr}}
} 


\AMSSUBJ{60H50; 60H10; 60G22; 34A06} %

\SUBMITTED{October 10, 2022} 

\VOLUME{0}
\YEAR{2020}
\PAPERNUM{0}
\DOI{10.1214/YY-TN}

\ABSTRACT{We study existence and uniqueness of solutions to the equation $dX_t=b(X_t)dt + dB_t$, where $b$ is a distribution in some Besov space and $B$ is a fractional Brownian motion with Hurst parameter $H\leqslant 1/2$. First, the equation is understood as a nonlinear Young equation. This involves a nonlinear Young integral constructed in the space of functions with finite $p$-variation, which is well suited when $b$ is a measure. Depending on $H$, a condition on the Besov regularity of $b$ is given so that solutions to the equation exist. The construction is deterministic, and $B$ can be replaced by a deterministic path $w$ with a sufficiently smooth local time. Using this construction we prove the existence of weak solutions (in the probabilistic sense).

We also prove that solutions coincide with limits of strong solutions obtained by regularisation of $b$. This is used to establish pathwise uniqueness and existence of a strong solution.
 In particular when $b$ is a finite measure, weak solutions exist for $H<\sqrt{2}-1$, while pathwise uniqueness and strong existence hold when $H\leqslant 1/4$. The proofs involve fine properties of the local time of the fractional Brownian motion, as well as new regularising properties of this process which are established using the stochastic sewing Lemma.}

\usepackage{pgfplots}
\pgfplotsset{width=7cm,compat=1.8}
\usetikzlibrary{patterns}
\usepackage[latin1]{inputenc}
\usepackage{latexsym}
\usepackage{amsmath,amssymb,amsfonts,amsthm}
\usepackage{xcolor}
\usepackage{mathrsfs}
\usepackage{comment}
\usepackage{bbm}
\usepackage{paralist}
\usepackage{float}

\usepackage{tikz}
\usetikzlibrary{calc}
\usetikzlibrary{decorations.pathreplacing}
\usetikzlibrary{arrows}
\usepgfplotslibrary{fillbetween}
\usepackage{mathtools}

\usepackage[toc,page]{appendix}
\usepackage{enumitem}

\usepackage{todonotes}

\definecolor{db}{RGB}{0, 0, 130}

\usepackage{pdfsync}

\definecolor{rp}{rgb}{0.25, 0, 0.75}
\definecolor{dg}{rgb}{0, 0.5, 0}
\newcommand{\red}{\textcolor{black}}

\newcommand{\R}{\mathbb{R}}

\newcommand{\N}{\mathbb{N}}
\newcommand{\EE}{\mathbb{E}}

\newcommand{\PP}{\mathbb{P}}

\newcommand{\vep}{\varepsilon}

\newcommand{\bqn}{\begin{equation}}
\newcommand{\eqn}{\end{equation}}
\newcommand{\bqne}{\begin{equation*}}
\newcommand{\eqne}{\end{equation*}}

\newcommand{\VCeta}{\mathcal{C}^{p\text{-}\text{var}}_{[0,T]}(\mathcal{C}^\eta)}

\newcommand{\COO}[4]{\mathcal{C}^{#1}_{#4}(L^{#2,#3})}

\newcommand{\Vn}[2]{[#1]_{\mathcal{C}_{#2}^{1\text{-}\text{var}}}}

\def \limbashaut#1#2#3{\mathrel{\mathop{\kern 0pt#1}\limits_{#2}^{#3}}}

\DeclareMathOperator{\var}{Var} 
 
\DeclareMathOperator{\supp}{supp}
\DeclareMathOperator{\Max}{Max}
\newcommand{\VCneta}[3]{[#1]_{\mathcal{C}_{#2}^{{#3}\text{-}\text{var}} 
(\mathcal{C}^\eta)}}
\newcommand{\VCnReta}[3]{[#1]_{\mathcal{C}_{#2}^{{#3}\text{-}\text{var}} 
(\mathcal{C}^\eta_{D_R})}}

\numberwithin{equation}{section}

\begin{document}

\section{Introduction}

We are interested in the well-posedness of the one-dimensional equation
\begin{align} \label{eq:skew}
dX_t=b(X_t)dt + dB_t,
\end{align}
when $b$ is a distribution in some Besov space and $(B_t)_{t \geqslant 0}$ is a 
fractional Brownian motion of Hurst parameter $H$. We will explain in which sense this 
equation can be solved when $b$ is a genuine distribution. It is noteworthy that, even when 
$b$ is a function, this equation can be well-posed while the corresponding equation without 
noise is not. This effect is often called regularisation by noise. We refer to 
\cite{FlandoliStFlour} for a thorough presentation of this phenomenon, in particular on PDE 
models of fluid mechanics. This equation encompasses at least two classes of equations 
 which have frequently been studied in the literature. 
 
First, when $B$ is the standard Brownian motion, there is an extensive literature which we will 
not attempt to describe thoroughly. Let us mention the early work of 
Veretennikov~\cite{Veretennikov} for bounded measurable drifts, then the more general 
$L^p-L^q$ criterion 
of Krylov and R\"ockner~\cite{KrylovRockner} for which the authors proved strong existence 
and uniqueness (both works allowing for time inhomogeneous drifts in dimension 
$d\geqslant1$). Flandoli, Russo and Wolf~\cite{FlandoliRussoWolf} developed a weak 
well-posedness theory while Bass and 
Chen~\cite{BassChen} proved existence and uniqueness of strong solutions with the drift being 
the distributional derivative of a $\gamma$-H\"older function for 
$\gamma>1/2$. Then Davie~\cite{Davie} provided conditions for path-by-path existence and 
uniqueness, which is a stronger form of uniqueness, for time inhomogeneous bounded 
measurable drift. 
Using rough path methods, Delarue and Diel~\cite{DelarueDiel} proved weak existence and 
uniqueness in dimension $1$ when the drift is the distributional derivative 
of a $\gamma$-H\"older function for $\gamma>1/3$. In higher dimension, Flandoli, Issoglio and Russo~\cite{FIR} 
identified a class of SDEs with distributional drifts in Bessel spaces such that there exists a solution that is unique in law. In addition, when the drift is random, another well-posedness result is given 
by Duboscq and R\'eveillac~\cite{DuboscqReveillac}. We also point out the work 
\cite{HarShepp}, with extensions in \cite{LeGall} on an SDE involving 
the local time at $0$ of the solution, which formally corresponds to a drift 
$b=a\delta_{0}$, for some 
$a\in [-1,1]$ and $\delta_{0}$ being the Dirac distribution. This setting corresponds to the 
so-called skew Brownian motion, see  \cite{Lejay} for more 
details and various constructions.

This leads to a second class of interesting problems, namely solving Equation \eqref{eq:skew} 
when $b$ is a distribution and $B$ is a fractional Brownian motion with sufficiently small Hurst 
parameter $H$. A first attempt in this direction seems to be due to Nualart and Ouknine 
\cite{NualartOuknine}, who proved existence and uniqueness for some non-Lipschitz drifts.
When 
$b=a\delta_{0}$ with $a\in \R$, the well-posedness of this equation was established for $H<1/4$ 
by Catellier and Gubinelli \cite{CatellierGubinelli} (who also consider multidimensional drifts in negative H\"older 
spaces) and 
independently for  $H<1/6$ in \cite{AmineEtAl,Banos} with extensions to dimension 
$d\geqslant 1$ in the three papers \cite{AmineEtAl,Banos,CatellierGubinelli}. The solution is 
generally referred to as skew fractional Brownian motion. We observe a gap between the 
one-dimensional Brownian case ($H=1/2$), with well posedness for $|a|\leqslant 1$ proven in 
\cite{LeGall}, and the aforementioned result for fractional Brownian motion with $H$ smaller 
than $1/4$. The intent of this paper is to partially close this gap. Note also that the case 
$a=1$ corresponds to reflection above $0$ in the Brownian case. The well-posedness of reflected equations was established even for multiplicative rough noises in case $X$ is one-dimensional \cite{DGHT,RTT}, while uniqueness might fail as soon as the dimension is greater than $2$ (see \cite{Gassiat}).

Finally, let us mention that regularisation by noise was also investigated  for other types of noise, for instance $\alpha$-stable noises \cite{CdRMenozziPub}, regular noises \cite{Gerencser} and other classes of rough processes   
\cite{HarangLing,HarangPerkowski}. Recently, the regularisation phenomenon was studied for SDEs with multiplicative noise (whether fractional Brownian motion \cite{DareiotisGerencser}, or more general rough paths \cite{CatellierDuboscq}).

~

In this paper, the drift $b$ is in some Besov space $\mathcal{B}^\beta_{p,\infty}$ (denoted 
$\mathcal{B}^\beta_{p}$ hereafter). The solutions we consider are stochastic processes of 
the form 
\begin{equation}
\label{eq:decompositionXKB}
X_{t} = X_{0} + K_{t} + B_{t},
\end{equation} 
where $K_{t}$ is the limit in probability of $\int_{0}^t 
b^{n}(X_{s})\, ds$ for any sequence $(b^n)$ of smooth approximations of $b$ (in line with \cite{Atetal,BassChen}). Roughly, 
when such a solution exists and $X$ and $B$ are adapted to the filtration of the underlying 
probability space, we call it a weak solution. When $X$ is adapted to the natural filtration of 
$B$, it is called a strong solution. We refer to Definition \ref{def:solution} for more 
details.

Our first main result, Theorem \ref{prop:existence}, gives conditions on $\beta$, $p$ and $H$ 
that ensure the existence of a weak solution to \eqref{eq:skew} when $b$ is measure. In 
particular when $b = a\delta_{0}$, for $a\in \R$, we obtain the existence of weak solutions to 
\eqref{eq:skew} for any $H<\sqrt{2}-1$. In the standard Brownian case, Theorem 
\ref{prop:existence} provides weak solutions when the drift is in 
$\mathcal{B}^{\frac{1}{4}+}_{1}$. This space contains functions which, to the best of our 
knowledge,  are not covered by the existing literature (see Remark \ref{rk:weakExBrownian}).

To prove Theorem \ref{prop:existence}, we consider another approach to study Equation \eqref{eq:skew} which is via nonlinear Young 
integrals as introduced in \cite{CatellierGubinelli}, extending Young's theory of integration 
\cite{Young}. Consider the more general equation
\begin{equation} \label{eq:skewZ}
dX_{t} = b(X_{t}) \, dt + dZ_t,
\end{equation}
where $Z$ is a continuous stochastic process. The idea is to define path-by-path solutions to 
\eqref{eq:skewZ}, that is, to solve this equation  for a fixed realisation of the noise \((Z_t(\omega))_t\). In order to do this one rewrites the equation as a random ODE: $dY_{t} = 
b(Y_{t} + Z_{t}) \, dt$, with $Y_{t} = X_{t} - Z_{t}$ and studies the regularity of the averaging 
operator $T_{t}^{Z}b: y\mapsto \int_{0}^t b(y + Z_{s})\, ds$.
In some interesting cases, $T_{t}^{Z}b$ is more regular than $b$ itself, which permits to have solutions of the form $Y_{t} = Y_{0} + \int_{0}^t T_{ds}^{Z}b(Y_{s})$, where the integral is a so-called nonlinear Young integral. We refer to \cite{Galeati} for a 
review of nonlinear Young integrals in the H\"older setting and also to the recent work \cite{GaleatiGubinelli}. 

In Theorem~\ref{solutionsagree} we give conditions such that solutions w.r.t. the probabilistic 
approach via approximation of the drift and w.r.t. the approach via nonlinear Young integral 
theory are equivalent (i.e. weak solutions coincide). Therefore, in order to show existence of a 
weak solution to Equation \eqref{eq:skew}, it is sufficient to construct a solution to the 
corresponding nonlinear Young integral equation that is additionally adapted to a small enough 
filtration.

To do so, we construct nonlinear Young integrals in the $p$-variation setting 
(instead of H\"older). In our one-dimensional setting this allows us to exploit the nonnegativity of $b$ to get existence of 
solutions to \eqref{eq:skewZ} under some milder conditions on the regularity of $b$. To 
obtain the regularity of the averaging operator, we proceed similarly to 
Harang and Perkowski
\cite{HarangPerkowski} by rewriting it as a convolution between $b$ 
and the local time of $Z$. We are then able to deduce the H\"older regularity of the 
averaging operator from the Besov regularity of the local time of $Z$ (see 
Lemma~\ref{lem:regTL}). These two ingredients permit to construct path-by-path solutions 
by convergence of the Euler scheme associated to the equation, see Theorem 
\ref{generalsolution}. Alternatively, if the local time of $Z$ has some probabilistic properties, as in the case $Z=B$, we are able to show that the averaging operator has a certain tightness property, see Lemma~\ref{lem:LT2}. This permits to prove that the sequence of (random) Euler schemes which 
approximate the nonlinear Young solution is compact in the space of continuous adapted stochastic processes, from which adaptedness of $X$ is deduced.


Our second main result, Theorem~\ref{thm:uniqueness}, states pathwise uniqueness of weak solutions in a class of processes which have some H\"older regularity. 
The main condition is that $b$ is in $\mathcal{B}^\beta_{p}$ with $\beta$ and $p$ satisfying 
$\beta-\frac{1}{p} \geqslant 1-\frac{1}{2H}$. As with the Yamada-Watanabe Theorem, weak 
existence and pathwise uniqueness also lead here, under the same conditions on $\beta$, $p$ 
and $H$, to the existence of a strong solution. Moreover, when $b$ is a measure, any strong solution is proven to have sufficient H\"older regularity, thus 
ensuring pathwise uniqueness. For instance, we get strong existence and pathwise 
uniqueness to \eqref{eq:skew} when $b=a\delta_{0}$ for $a\in \R$ and any $H\leqslant 1/4$. 
This extends the previously known condition $H<1/4$ from \cite{CatellierGubinelli,GaleatiGubinelli} to 
$H=1/4$.

To prove this theorem, we follow the recent approach of Athreya et al.~\cite{Atetal}. In this 
paper, the 
authors proved the strong well-posedness of the one-dimensional stochastic heat equation 
with Besov drift by a tightness-stability argument. The main regularity estimates are obtained 
\emph{via} the recent and powerful stochastic sewing Lemma introduced by L\^{e}~\cite{Le}. 
To 
adapt this argument to our setting, we control the H\"older norm of the conditional 
expectation of $x\mapsto \int_{s}^t f(x+B_{r})\, dr$ in terms of the Besov norm of $f$ (see 
Lemma~\ref{lem:mainregularisation}). One difficulty that arises here is the non-Markovian 
nature of $B$, which we could compensate by using a slightly adapted version of the local 
nondeterminism property of the fractional Brownian motion.

In the development of the proof of Theorem~\ref{thm:uniqueness}, we obtain that weak solutions to \eqref{eq:skew} are limits of strong solutions to \eqref{eq:skew} with $b$ replaced by a smooth bounded drift $b^n$, where the sequence $(b^n)$ converges to $b$ in Besov norm. This result is detailed in Theorem~\ref{th:approx} and can be of independent interest in view of numerical applications.

~

\paragraph{Structure of the paper.}
In Section~\ref{sec:main}, the main definitions and results are stated. We also present 
the organisation of the proofs in the paper.
In Section~\ref{nonlinearyoung}, we develop the construction of nonlinear Young integrals in 
$p$-variation (Theorem~\ref{int}) and use it to find solutions to nonlinear Young integral 
equations with nonnegative (or nonpositive) drifts, see Theorem~\ref{thm:existence}. 
Then in Section~\ref{sec:existWeak}, we prove successively Theorem~\ref{solutionsagree} (relation between different notions of solution) and Theorem~\ref{generalsolution} (existence of 
path-by-path solutions).  Then we  conclude with the proof of 
Theorem~\ref{prop:existence} about the existence of weak solutions.
 The regularity of weak 
solutions is studied in Section~\ref{sec:RegWeak}. 
The uniqueness part of Theorem~\ref{thm:uniqueness} is proven in Section~\ref{uniqueness}.  The tightness-stability argument which leads to the existence of strong solutions is in Section~\ref{thmproof}.

Besides, we recall some useful results on Besov spaces in Appendix \ref{app:Besov}. The local nondeterminism property of the fBm is stated and proven in Appendix~\ref{app:LND}, jointly with the proof of the important regularity estimates of the conditional expectation of the fBm (Lemma~\ref{lem:Cs}). Finally, we recall the stochastic sewing Lemma in Appendix~\ref{app:sewing} and use it to derive several H\"older bounds on the integrals of fBm.

\subsection{Notations and definitions.}\label{sec:notations}

\paragraph{Various notations.}
Throughout the paper, we use the following notations and conventions:
\begin{compactitem}
	\item Constants $C$ might vary from line to line. 
\item For $p \in [1,\infty]$, $p^\prime \in [1,\infty]$ is such that $1/p+1/{p^\prime}=1$.
\item
For topological spaces $X,Y$ we denote the set of continuous functions  
from $X$ to $Y$ by $\mathcal{C}_X(Y)$. 
\item For a Banach space $E$, the ball of radius $R>0$ is denoted by $D_R:=\{x \in 
E: \|x\|\leqslant R\}$.
\item 
Let \(s < t\) be two real numbers and \(\Pi = (s = t_0 < t_1 < \cdots < t_n =t)\) be a 
partition of \([s,t]\), we denote $|\Pi| = \sup_{i=1,\cdots,n}(t_{i} - t_{i-1})$ the mesh 
of $\Pi$. 
\item For $s,t \in \mathbb{R}$ with $s\leqslant t$, we denote $\Delta_{[s,t]}:=\{(u,v):s\leqslant u \leqslant v \leqslant t\}$.
\item For any function $f$ defined on $[s,t]$, we denote $f_{u,v}:=f_v-f_u$ for $(u,v) \in \Delta_{[s,t]}$. 
\item For any function \(g\) defined on $\Delta_{[s,t]}$ and 
$s\leqslant r\leqslant u \leqslant v\leqslant t$, we denote $\delta 
g_{r,u,v}:=g_{r,v}-g_{r,u}-g_{u,v}$.
\item 
For a probability space $\Omega$ and $p \in [1,\infty]$, the norm on $L^p(\Omega)$ is denoted by $\|\cdot\|_{L^p}$.
\item 
We denote by $(B_t)_{t \geqslant 0}$ a fractional Brownian motion 
with  Hurst parameter $H\leqslant 1/2$. 
\item 
The filtration  \((\mathcal{F}_t)_{t \geqslant 
	0}\) is denoted by \(\mathbb{F}\). 
\item The filtration generated by a process \((Z_t)_{t \geqslant 
	0}\) is 
denoted by $\mathbb{F}^Z$. 
\item All filtrations are assumed to fulfill  the usual conditions. 
\item 
Let $\mathbb{F}$ be a filtration. We call $(W_t)_{t \geqslant 
	0}$ a $\mathbb{F}$-Brownian 
motion if $(W_t)_{t \geqslant 
	0}$ is a Brownian motion adapted to $\mathbb{F}$ and for $0 
\leqslant s \leqslant t$, $W_t-W_s$ is independent of $\mathcal{F}_s$. For such a filtration, the conditional expectation $\EE[\cdot \mid \mathcal{F}_s]$ is denoted by 
$\EE^s[\cdot]$.
\end{compactitem}

\paragraph{Gaussian semigroup.} 
For any $t>0$ and $x \in \mathbb{R}$, let $g_t(x):=\tfrac{1}{\sqrt{2\pi 
t}}e^{-\tfrac{x^2}{2t}}$. For a tempered distribution $\phi$ on 
\(\mathbb{R}\),
let
\begin{align} \label{Gaussiansemigroup}
    G_t\phi(x):=(g_t * \phi) (x).
\end{align}

\paragraph{The occupation time formula.} 
Let $T>0$, $\mathit{w}: [0,T] \rightarrow \mathbb{R}$ be a measurable function
and let $\lambda$ 
denote the Lebesgue measure on $\mathbb{R}$. 
For $A \in \mathcal{B}([0,T])$, let $\mu_A$ be the occupation measure defined by
$\mu_A(\Lambda):=\lambda(\{t \in A:\mathit{w}_t \in \Lambda\})$ for $\Lambda \in 
\mathcal{B}(\mathbb{R})$. If $\mu_{[0,T]} \ll \lambda$ 
then 
there exists a measurable map $\ell: \mathcal{B}([0,T])\times \mathbb{R} \to \R_{+}$ 
such that for $A \in \mathcal{B}([0,T])$, $\mu_A(dx) = \ell(A,x) \lambda(dx)$. For any 
bounded measurable function $f$, the occupation time formula reads (see 
\cite{Geman} for more details)
\begin{align} \label{eq:kernel}
    \int_A f(\mathit{w}_r) dr = \int_{\mathbb{R}}f(x)\, \ell(A,x)\, dx.
\end{align}
We define a local time $L:[0,T] \times \mathbb{R}\rightarrow \mathbb{R}$ by $L_t(x):=\ell([0,t],x)$. By \eqref{eq:kernel}, it comes that for any bounded measurable $f$ and $t \in [0,T]$,
\begin{align} \label{occupation}
    \int_0^t f(\mathit{w}_r) \, dr = \int_{\mathbb{R}} f(x)\, L_t(x)\, dx.
\end{align}
Note that if $\mathit{w}:[0,T]\rightarrow K \subset \mathbb{R}$ for some compact $K$, then $L_t(\cdot)$ vanishes on $K^{\mathsf{c}}$ for all $t \in [0,T]$.

\paragraph{Finite variations spaces.} 

Let $p \in [1,\infty)$ and \((F, \|~\|_F)\) be a Banach space. Define the $p$-variation seminorm of a 
function $x: [s,t] \rightarrow F$ as
\begin{align*}
   [x]_{\mathcal{C}_{[s,t]}^{p\text{-}\text{var}} (F)}:= \sup_{\{t_i\}}\left(\sum_{i=0}^{N-1} 
   \|x_{t_{i},t_{i+1}}\|^p_F\right)^{\frac{1}{p}},
\end{align*}
where the supremum runs over all partitions $\{t_i\}_{i=0}^N$, $N \in \mathbb{N}$, of 
$[s,t]$. We denote by $\mathcal{C}^{p\text{-}\text{var}}_{[s,t]}(F)$ the set of such continuous 
functions with 
finite $p$-variation.

 If $F$ is just a Fr\'echet space, we say that a function mapping from $[s,t]$ to $F$ 
 has finite $p$-variation if its $p$-variation is finite with respect to any continuous 
 seminorm and we also use the notation $\mathcal{C}^{p\text{-}\text{var}}_{[s,t]}(F)$. If $F=\mathbb{R}$, we use the alleviate notations  $\mathcal{C}^{p\text{-}\text{var}}_{[s,t]}$ 
and  $[x]_{\mathcal{C}_{[s,t]}^{p\text{-}\text{var}}}$. 

A continuous function $\varkappa:\Delta_{[s,t]}\rightarrow [0,\infty)$ is a control function if, for 
$s\leqslant r \leqslant u \leqslant v \leqslant t$,
\begin{align} \label{eq:superadditive}
    \varkappa(r,u)+\varkappa(u,v) \leqslant \varkappa(r,v),
\end{align}
and $\varkappa(r,r)=0$ for all $r \in [s,t]$. A typical example of a control function is 
$[x]^p_{\mathcal{C}_{[s,t]}^{p\text{-}\text{var}}(F)}$ (see \cite[Prop.~5.8]{FrizVictoir}).

~

\paragraph{Besov and H\"older spaces.}

\begin{definition}
For $s \in \mathbb{R}$ and $1\leqslant p,q \leqslant \infty$, we denote the nonhomogeneous Besov 
space with these parameters by $\mathcal{B}_{p,q}^s$. For a precise definition see Appendix~\ref{app:Besov}.

Besides, for a bounded open interval $\mathcal{I}\subset \mathbb{R}$, we denote by 
$\mathcal{B}^s_{p,q}(\mathcal{I})$ the space of all distributions $u$ on $\mathcal{I}$ for 
which there exists $v \in \mathcal{B}^s_{p,q}$ such that $\left.u=v\right|_\mathcal{I}$ (see 
\cite[Def. 1, p.192]{Triebel}).
\end{definition}

If $q=\infty$, we  write $\mathcal{B}_p^s$ instead of $\mathcal{B}_{p,\infty}^s$. 
We have the following important embeddings between Besov spaces.
\begin{remark} \label{embedding2}
Let $s \in \mathbb{R}$, $1\leqslant p_1 \leqslant p_2 \leqslant \infty$ and $1\leqslant q_1 \leqslant q_2 \leqslant \infty$.  From 
\cite[Prop.~2.71]{BaDaCh}, the space $\mathcal{B}_{p_1,q_1}^s$ continuously embeds into $\mathcal{B}^{s-(p_1^{-1}-p_2^{-1})}_{p_2,q_2}$, which we write as ${\mathcal{B}_{p_1,q_1}^s \hookrightarrow \mathcal{B}^{s-(p_1^{-1}-p_2^{-1})}_{p_2,q_2}}$.
\end{remark}

\vspace{0.1cm}

\begin{remark} \label{embedding3}
Let $\mathcal{I}$ be a bounded open interval. Let $p_1,p_2,q_1,q_2 \in [1,\infty]$ and 
$-\infty<s_2<s_1<\infty$. If $s_1-1/{p_1}>s_2-1/{p_2}$, then 
from \cite[Th. p.196]{Triebel} we 
have $\mathcal{B}^{s_1}_{p_1,q_1}(\mathcal{I})\hookrightarrow 
\mathcal{B}^{s_2}_{p_2,q_2}(\mathcal{I})$.
\end{remark}

\vspace{0.1cm}

For $s \in \mathbb{R}_{+}\setminus \mathbb{N}$ and $p=q=\infty$, Besov spaces 
coincide with H\"older spaces (see \cite[p.99]{BaDaCh}). 
We now give a definition of H\"older spaces in space domain for $s\in(0,1]$.
\begin{definition}\label{def:Holder}
Let $E,F$ be Banach spaces, $U\subset E$ and $\beta \in (0,1]$.
\begin{itemize}

\item We denote the supremum norm of $f\in \mathcal{C}_{U}(F)$ by 
$\|f\|_{\mathcal{C}_{U}(F)} = \sup_{x\in U} \|f(x)\|_{F}$. 
When $U$ and $F$ are clear from the context, we might also denote $ \|f\|_{\infty} = \|f\|_{\mathcal{C}_{U}(F)}$.

\item  The H\"older space $\mathcal{C}^\beta_U(F)$ is the collection of all $f \in \mathcal{C}_U(F)$ such that $\|f\|_{\mathcal{C}_U^\beta (F)}$ is finite, where
\begin{align*}
    \|f\|_{\mathcal{C}_U^\beta (F)}:= [f]_{\mathcal{C}_U^\beta (F)} + \|f\|_{\mathcal{C}_{U}(F)}
    \quad\text{ with }\quad [f]_{\mathcal{C}_U^\beta (F)}:= \sup_{x \neq y \in U} 
    \frac{\|f(x)-f(y)\|_F}{\|x-y\|^\beta_E}.
\end{align*}
If $U=E$, we alleviate the notations and write  \(\|f\|_{\mathcal{C}^\beta (F)}\) and 
\([f]_{\mathcal{C}^\beta (F)}\). Similarly, if $F=\mathbb{R}$ or if \(F\) is clear from the context,  
we write  \(\|f\|_{\mathcal{C}_U^\beta}\) and \([f]_{\mathcal{C}_U^\beta}\).
Finally,  if $U=F=\mathbb{R}$ we just write $\mathcal{C}^\beta$. 
\item The space $\mathcal{C}^\beta_{E,loc}(F)$ of locally H\"older continuous functions is the collection of all $f \in \mathcal{C}_E(F)$ such that $ \|f\|_{\mathcal{C}_{D_R}^\beta (F)}$ is finite for all $R>0$.

\end{itemize}
\end{definition}

\begin{remark} \label{rem:measure}
	In some results, we assume that the drift is a (nonnegative) measure in some 
	$\mathcal{B}^\beta_p$. 
	This is actually equivalent to the assumption that \(b\) is a nonnegative distribution in $\mathcal{B}^\beta_p$. Indeed, thanks 
	to  \cite[Exercise 22.5]{Treves}, any nonnegative distribution is given by a
	Radon measure (i.e. a locally finite, complete measure fulfilling regularity conditions). Hence, it is sufficient to consider Radon measures lying in Besov spaces instead of considering general nonnegative distributions. Throughout the paper all measures are assumed to be Radon.
\end{remark}

\section{Main results}\label{sec:main}

\subsection{Definitions of solution}
We define here weak and strong solutions to \eqref{eq:skew}. 
\begin{definition} \label{def:beta-}
Let $\beta \in \mathbb{R}$, $p \in [1,\infty]$. We say that $(f_n)_{n \in \mathbb{N}}$ converges to $f$ in $\mathcal{B}_p^{\beta-}$ as $n \rightarrow \infty$ if $\sup_{n \in \mathbb{N}} \|f_n\|_{\mathcal{B}_p^\beta}<\infty$ and 
\begin{align*}
\forall \beta^\prime<\beta, \quad    \lim_{n \rightarrow 
\infty}\|f_n-f\|_{\mathcal{B}_p^{\beta^\prime}}=0.
\end{align*}
\end{definition}

\begin{remark} \label{rem:approximation}
For any  $f\in \mathcal{B}_p^\beta$, there exists a sequence \((f_n)_n\) of bounded smooth 
functions  converging to \(f\) in the sense of Definition~\ref{def:beta-}:
 \textit{e.g.} $f_n:=G_{\frac{1}{n}} f$, where \(G\) is the Gaussian 
semigroup introduced in \eqref{Gaussiansemigroup}. This can be seen to hold true by applying
Lemmas~\ref{A.3} and \ref{lem:heatkernel}. 
\end{remark} 

We recall here a link between fractional Brownian motion and Brownian motion.
For each $H\in (0,\tfrac{1}{2})$, there exist operators $\mathcal{A}$ and 
$\bar{\mathcal{A}}$, where both can be given in terms of 
fractional integrals and derivatives (see \eqref{operator} and \cite[Th. 11]{Picard}), such that
\begin{align}
\text{if $B$ is a fractional Brownian motion, } W=\mathcal{A}(B) \text{ is a Brownian motion}, 
\label{operatortildeA}\\
\text{if $W$ is a Brownian motion, } B=\bar{\mathcal{A}}(W) \text{ is a fractional Brownian 
motion}. \label{operatorA}
\end{align}
Besides, $B$ and $W$ generate the same filtration. 

We give here the definition of a $\mathbb{F}$-fractional Brownian 
motion, for a given filtration \(\mathbb{F}\). 
\begin{definition}
	Let  $\mathbb{F}$ be a filtration. We say that $B$ is a $\mathbb{F}$-fractional Brownian 
	motion if 
$W=\mathcal{A}(B)$ is a $\mathbb{F}$-Brownian motion.
\end{definition}

\begin{definition} \label{def:solution}
Let $\beta \in \mathbb{R}$, $p \in [1,\infty]$, $b \in \mathcal{B}_p^\beta$, $T>0$ and $X_0 \in \mathbb{R}$.
\begin{itemize}

\item \emph{Weak solution:} 
We call a couple $((X_t)_{t \in [0,T]},(B_t)_{t \in 
[0,T]})$ defined on some filtered probability space  
$(\Omega,\mathcal{F},\mathbb{F},\mathbb{P})$ a weak solution to \eqref{eq:skew} on 
$[0,T]$, with 
initial condition $X_0$, if 
\begin{compactitem}
\item $B$ is a $\mathbb{F}$-fBm;

\item $X$ is adapted to $\mathbb{F}$;

\item there exists a process $(K_t)_{t 
\in [0,T]}$ such that, a.s.,
\begin{equation}\label{solution1}
X_t=X_0+K_t+B_t\text{ for all } t \in [0,T] ;
\end{equation}

\item for every sequence $(b^n)_{n\in \mathbb{N}}$ of smooth bounded functions converging to $b$ in $\mathcal{B}^{\beta-}_p$, we have that
\begin{equation}\label{approximation2}
       \sup_{t\in [0,T]}\left|\int_0^t b^n(X_r) dr 
       -K_t\right|\limbashaut{\longrightarrow}{n\rightarrow \infty}{\mathbb{P}} 0.
\end{equation}
\end{compactitem}
If the couple is clear from the context, we simply say that  $(X_t)_{t \in [0,T]}$ is a weak 
solution.

\item \emph{Pathwise uniqueness:} As in the classical literature on SDEs, we say that pathwise uniqueness holds if for any two solutions $(X,B)$ and $(Y,B)$ defined on the same filtered probability space with the same
fBm $B$ and same initial condition $X_0 \in \mathbb{R}$, $X$ and $Y$
 are indistinguishable.
\item \emph{Strong solution:} A weak solution $(X,B)$ such that $X$ is 
$\mathbb{F}^B$-adapted is called a strong solution.

\end{itemize}
\end{definition}

\subsection{Existence and uniqueness results}

\begin{theorem} \label{prop:existence}
Let $\beta \in \mathbb{R}$, $p \in [1,\infty]$ and $H\in(0,\frac{1}{2}]$. Let $b \in \mathcal{B}^{\beta}_p$ be a measure. Assume that one of the following conditions holds:
\begin{enumerate}[label=(\roman*)]
\item \label{ex:onebis} $H\geqslant\frac{1}{3}$ and $\beta > 1+\frac{H}{2}-\frac{1}{2H}$;
\item \label{ex:twobis} $H<\frac{1}{3}$ and $\beta>2H-1$;
\item \label{ex:fourbis} $p \in [2,\infty]$ and $\beta>-\frac{1}{2H}+1$.
\end{enumerate}

Then, 
\begin{enumerate}[label=(\alph*)]
\item \label{en:weak1} there exists a weak solution $X$ to \eqref{eq:skew} such that the convergence in \eqref{approximation2} holds a.s.
\item \label{en:weak2} Additionally, $X-B \in \mathcal{C}^\kappa_{[0,T]}(L^m)$ for any 
$\kappa \in (0,1+H(\beta-\frac{1}{p})\wedge 0]\setminus \{1\}$ and $m\geqslant2$.
\end{enumerate}
\end{theorem}

\vspace{0.1cm}

\begin{corollary} \label{cor:existence}
For any finite measure $b$, there exists a weak solution to 
\eqref{eq:skew} for $H<\sqrt{2}-1$. If $b= a \delta_0$, for some $a \in \mathbb{R}$, we call it 
an \emph{$a$-skew fractional Brownian motion}.
\end{corollary}

\vspace{0.1cm}

\begin{remark}\label{rk:weakExBrownian}
In the Brownian motion case (\(H = 1/2\)), we obtain existence of a weak solution in new cases. 
For instance, consider 
\begin{equation*}
b(x)=|x|^{-3/4+\varepsilon} a(x),
\end{equation*}
 where $a$ is a smooth, compactly supported, nonnegative function equal to $1$ on $[-1,1]$. We have that 
 $b\in \mathcal{B}^{\frac{1}{4}+}_{1}$ (see~\cite[Prop. 2.21]{BaDaCh} for similar 
 computations), and the space $\mathcal{B}^{\frac{1}{4}+}_{1}$ is covered by 
 Assumption~\ref{ex:onebis} of Theorem~\ref{prop:existence}. Since $b$ is neither in 
 $\bigcup_{p \geqslant 2} L^p(\R)$ nor in $\mathcal{C}^{-\frac{2}{3}+}$, we cannot directly 
 apply results from \cite{KrylovRockner} or \cite{DelarueDiel}. 
\end{remark}

\vspace{0.1cm}

The following theorem gives, as Theorem~\ref{prop:existence} does, conditions on the drift and the Hurst parameter such that \eqref{eq:skew} has a weak solution. Note that this time, there is no nonnegativity assumption on $b$. Even in the case of considering $b$ to be a measure, none of the two theorems is stronger than the other. However, if $b$ is a finite measure, Theorem~\ref{prop:existence} allows for a wider range of Hurst parameters to get existence of weak solutions to \eqref{eq:skew}. Moreover, even though Theorem~\ref{prop:existence}\ref{ex:fourbis} is fully covered by \eqref{eq:assumptionweak}, it still adds value as it gives a.s. convergence in \eqref{approximation2}.

\begin{theorem} \label{th:approx}
Let $\beta \in \mathbb{R}$, $p \in [1,\infty]$, $b \in \mathcal{B}_p^\beta$ and 
$X_0 \in \mathbb{R}$. 
Let $(b^n)_{n \in \mathbb{N}}$ be a sequence of smooth bounded functions converging 
to $b$ in $\mathcal{B}_p^{\beta-}$.
 Let $X^n$ be the unique strong solution to
 \eqref{eq:skew} with drift $b^n$. Assume
\begin{equation} \label{eq:assumptionweak}
\beta-\frac{1}{p}>-\frac{1}{2H}+ \frac{1}{2}.
\end{equation}
Then, there exists a subsequence $(n_k)_{k \in \mathbb{N}}$ such that $(X^{n_k})_{k \in 
\mathbb{N}}$ converges in law w.r.t. $\|\cdot\|_\infty$ to a process $X$ which is a weak 
solution to \eqref{eq:skew} with drift $b$. Furthermore, $X-B \in 
\mathcal{C}^\kappa_{[0,T]}(L^m)$ for any $\kappa \in (0,1+H(\beta-\frac{1}{p})\wedge 
0]\setminus \{1\}$ and $m\geqslant2$.
\end{theorem}

\begin{remark}
Without loss of generality, the previous theorem can directly be formulated for H\"older spaces by fixing $p=\infty$, using the embedding from Remark~\ref{embedding2}. However, this is not the case for Theorem~\ref{thm:uniqueness} below. Hence we keep working in general Besov spaces for a better comparison of the results.
\end{remark}

Under slightly stronger assumptions than \eqref{eq:assumptionweak}, the following theorem states strong existence and pathwise uniqueness. In particular, it implies that under this stronger condition, convergence in probability of the approximation scheme in Theorem~\ref{th:approx} holds without passing to a subsequence.

\begin{theorem} \label{thm:uniqueness}
Let $H<1/2$, $\beta \in \mathbb{R}$, $p \in [1,\infty]$, $b \in \mathcal{B}_p^\beta$ and $X_0 \in \mathbb{R}$. 
Assume
\begin{equation} \label{eq:assumptionstrong}
\beta>- \frac{1}{2H}+1 ~\mbox{ and }  ~ \beta- \frac{1}{p} \geqslant - \frac{1}{2H}+1 .
\end{equation}
Then,
\begin{enumerate}[label=(\alph*)]

    \item \label{uniqueness(1)} there exists a strong solution $X$ to \eqref{eq:skew} such that $X-B \in \mathcal{C}^{\frac{1}{2}+H}_{[0,T]}(L^m)$ for any $m\geqslant2$;

    \item \label{uniqueness(1-2)} pathwise uniqueness holds in the class of all solutions $X$ such that $X-B \in \mathcal{C}^{\frac{1}{2}+H}_{[0,T]}(L^2)$;
    \item \label{uniquelimit}for any sequence $(b^n)_{n \in \mathbb{N}}$ of smooth bounded functions converging to $b$ in $\mathcal{B}^{\beta-}_p$, the corresponding sequence of strong solutions $(X^n)_{n \in \mathbb{N}}$ to \eqref{eq:skew} with drift $b^n$ converges in probability to the unique strong solution for which $X-B \in \mathcal{C}^{\frac{1}{2}+H}_{[0,T]}(L^2)$. In particular $X$ is independent of the chosen sequence of approximations.
    \item if $b$ is a measure, there exists a pathwise unique 
    strong solution to \eqref{eq:skew}. \label{uniqueness(2)}
\end{enumerate}
\end{theorem}

\begin{remark}
\begin{itemize}
\item For $b$ being a finite measure (hence in $\mathcal{B}^0_{1}$), Theorem 
\ref{thm:uniqueness} gives existence of a unique strong 
solution to \eqref{eq:skew} for $H\leqslant 1/4$ and Theorem~\ref{th:approx} 
gives existence of a weak solution to \eqref{eq:skew} for $H<1/3$. 

Such a finite measure $b$ is also in $\mathcal{B}^{-1}_\infty$. In this space, the existence of a unique strong solution was shown in \cite{HarangGaleati} for $H<1/4$ (elaborating on the path-by-path result of \cite{CatellierGubinelli}).
 Hence in this case, Theorem~\ref{thm:uniqueness} extends this result to $H=1/4$.

\item 
In the Brownian motion case, Theorem~\ref{th:approx} gives existence of weak 
solutions for $b \in \mathcal{B}^\beta_p$ when $\beta-1/p>-1/2$. In this regime strong existence and pathwise uniqueness are already known by \cite{BassChen}.
\item Note that $H=1/2$ is excluded from Theorem \ref{thm:uniqueness}. Strong existence and uniqueness result are already known under our assumptions (see \cite{KrylovRockner}).
\end{itemize}
\end{remark}

The following diagrams display for which Besov-valued distributions $b$ we have well-posedness for Equation \eqref{eq:skew}. 
The black-hatched area and the turquoise area correspond to the result obtained in 
Theorem~\ref{prop:existence}. The graphics visualize that the 
weak solution constructed in Theorem~\ref{prop:existence} is a solution that, in some cases, does not arise from the weak solution constructed in 
Theorem~\ref{th:approx}.

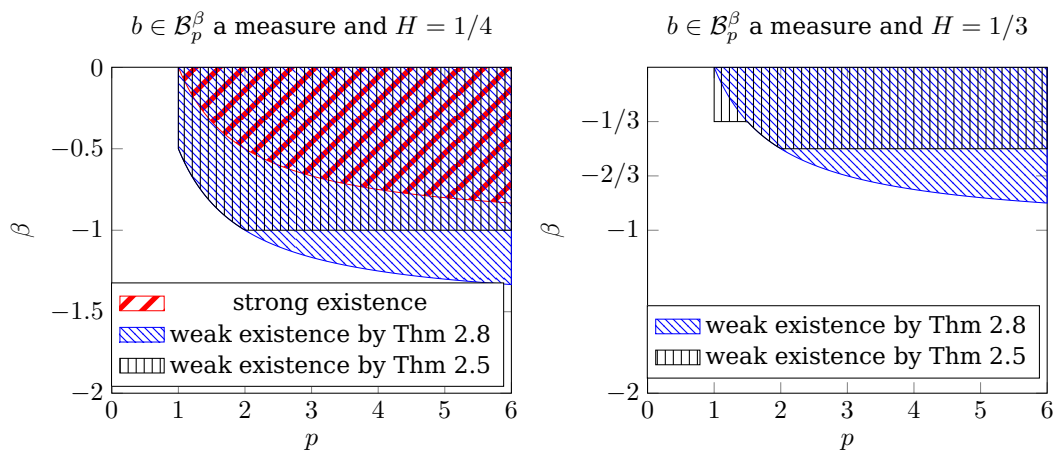
\begin{figure}[H]
\centering
\begin{tikzpicture}[scale=0.97,
        hatch distance/.store in=\hatchdistance,
        hatch distance=10pt,
        hatch thickness/.store in=\hatchthickness,
        hatch thickness=2pt,
                declare function={
    func2(\x)= (\x<2) * (-3/2+1/x) +
              and(\x >= 2, \x < 7) * (-1) +
              (\x >= 7) * (-1)
   ;
	}
    ]
    \makeatletter
    \pgfdeclarepatternformonly[\hatchdistance,\hatchthickness]{flexible hatch}
    {\pgfqpoint{0pt}{0pt}}
    {\pgfqpoint{\hatchdistance}{\hatchdistance}}
    {\pgfpoint{\hatchdistance-1pt}{\hatchdistance-1pt}}%
    {
        \pgfsetcolor{\tikz@pattern@color}
        \pgfsetlinewidth{\hatchthickness}
        \pgfpathmoveto{\pgfqpoint{0pt}{0pt}}
        \pgfpathlineto{\pgfqpoint{\hatchdistance}{\hatchdistance}}
        \pgfusepath{stroke}
    }

    \begin{axis}[
    	title style={align=center},
        title={ $b \in \mathcal{B}^\beta_p$ a measure and $H=1/4$},
        xmin=0,xmax=6,
        xlabel={$p$},
        xtick={0,1,2,3,4,5,6},
        ylabel={$\beta$},
        ymin=-2,ymax=0,
        legend style={at={(0.49,+0.34)},anchor=north,{cells={align=left}}}]
    \addlegendentry{strong existence}
    \addplot+[mark=none,
        domain=1:6,
        samples=100,
        pattern=flexible hatch,
        area legend,
        draw=red,
        pattern color=red]{-1+1/x} \closedcycle;
    \addlegendentry{weak existence by Thm~\ref{th:approx}}
    \addplot+[mark=none,
        domain=1:6,
        samples=100,
        pattern=north west lines,
        hatch distance=5pt,
        hatch thickness=6pt,
        draw=blue,
        pattern color=blue,
        area legend]{-1.5+1/x} \closedcycle;    
    \addlegendentry{weak existence by Thm~\ref{prop:existence}}
    \addplot+[mark=none,
        domain=1:6,
        samples=100,
        pattern=vertical lines,
        hatch distance=5pt,
        hatch thickness=6pt,
        draw=black,
        pattern color=black,
        area legend]{func2(x)} \closedcycle;    
    \end{axis}
    
\end{tikzpicture}
\begin{tikzpicture}[scale=0.97,
        hatch distance/.store in=\hatchdistance,
        hatch distance=10pt,
        hatch thickness/.store in=\hatchthickness,
        hatch thickness=2pt,
        declare function={
    func(\x)= (\x<3/2) * (-1/3) + and(\x>=3/2, \x<2) * (-1+1/x)   +
              and(\x >= 2, \x < 7) * (-1/2)     +
              (\x >= 7) * (-1/2)
   ;
	}
    ]
    \makeatletter
    \begin{axis}[
    	title style={align=center},
        title={ $b \in \mathcal{B}^\beta_p$ a measure and $H=1/3$},
        xmin=0,xmax=6,
        xlabel={$p$},
        xtick={0,1,2,3,4,5,6},
        ylabel={$\beta$},
        ytick={-0.33333,-0.6666,-1,-2},
        yticklabels={$-1/3$,$-2/3$,$-1$,$-2$},
        ymin=-2,ymax=0,
        legend style={at={(0.49,+0.27)},anchor=north,{cells={align=left}}}]
    \addlegendentry{weak existence by Thm~\ref{th:approx}}
    \addplot+[mark=none,
        domain=1:6,
        samples=100,
        pattern=north west lines,
        hatch distance=5pt,
        hatch thickness=6pt,
        draw=blue,
        pattern color=blue,
        area legend]{-1+1/x} \closedcycle;    
    \addlegendentry{weak existence by Thm~\ref{prop:existence}}
    \addplot+[mark=none,
        domain=1:6,
        samples=100,
        pattern=vertical lines,
        hatch distance=5pt,
        hatch thickness=6pt,
        draw=black,
        pattern color=black,
        area legend]{func(x)} \closedcycle;  
    \end{axis}
    
\end{tikzpicture}
\caption{Existence (and uniqueness) for $b$ a measure and $H$ fixed}
\end{figure}
\begin{figure}[H]
\centering
\begin{tikzpicture}[scale=0.94,
        hatch distance/.store in=\hatchdistance,
        hatch distance=10pt,
        hatch thickness/.store in=\hatchthickness,
        hatch thickness=2pt
    ]
    \makeatletter

    \begin{axis}[
    	title style={align=center},
        title={Existence and uniqueness for $b \in \mathcal{B}^\beta_p$},
        xmin=0,xmax=5,
        xlabel={$H$},
        xtick={0,1,2,3,4,5},
        xticklabels = {$0$,$0.1$,$0.2$,$0.3$,$0.4$,$0.5$},
        ylabel={$\beta-1/p$},
        ymin=-9,ymax=0,
        legend style={at={(0.63,+0.37)},anchor=north,{cells={align=left}}}]
    \addlegendentry{strong existence \\ and uniqueness}
    \addplot+[mark=none,
        domain=0:5,
        samples=100,
        pattern=flexible hatch,
        area legend,
        draw=red,
        pattern color=red]{1-1/(0.2*x)} \closedcycle;   
    \addlegendentry{weak existence}
    \addplot+[mark=none,
        domain=0:5,
        samples=100,
        pattern=north west lines,
        hatch distance=5pt,
        hatch thickness=6pt,
        draw=blue,
        pattern color=blue,
        area legend]{1/2-1/(0.2*x)} \closedcycle;    
    \end{axis}
   
\end{tikzpicture}
\begin{tikzpicture}[
        hatch distance/.store in=\hatchdistance,
        hatch distance=10pt,
        hatch thickness/.store in=\hatchthickness,
        hatch thickness=2pt
    ]
    \makeatletter

    \begin{axis}[
    	title style={align=center},
        title={Weak existence \\ for a measure $b \in \bigcup_{p\in [1,\infty]}\mathcal{B}^{\beta}_p$},
        xmin=0,xmax=5,
        xlabel={$H$},
        xtick={0,1,2,3,4,5},
        xticklabels = {$0$,$0.1$,$0.2$,$0.3$,$0.4$,$0.5$},
        ylabel={$\beta$},
        ymin=-2,ymax=1,
        ytick={1,1/4,0,-1/3,-1,-2},
        yticklabels={$1$,$1/4$,$0$,$-1/3$,$-1$,$-2$},
        legend style={at={(0.63,+0.37)},anchor=north,{cells={align=left}}}]
\addplot [name path = A,opacity=0,
    domain = 0:5,
    samples = 1000] {1} ;

\addplot [name path = B,opacity=0,
    domain = 3.3333:5] {1+(0.05*\x)-1/(0.2*\x)} ;
\addplot [name path = D,opacity=0,
    domain = 0.25:3.3333] {0.2*\x-1}
;
\addplot [teal!30] fill between [of = A and D, soft clip={domain=0:5}];
\addplot [teal!30] fill between [of = A and B, soft clip={domain=3.3333:5}];
\addplot [name path = C,opacity=0,
    domain = 0:3.3333] {-1/(0.2*\x)+3/2} ;
\addplot [teal!30] fill between [of = A and C, soft clip={domain=0:2.5}];

    \end{axis}
    
\end{tikzpicture}
\caption{Existence (and uniqueness) for variable $H$}
\end{figure}
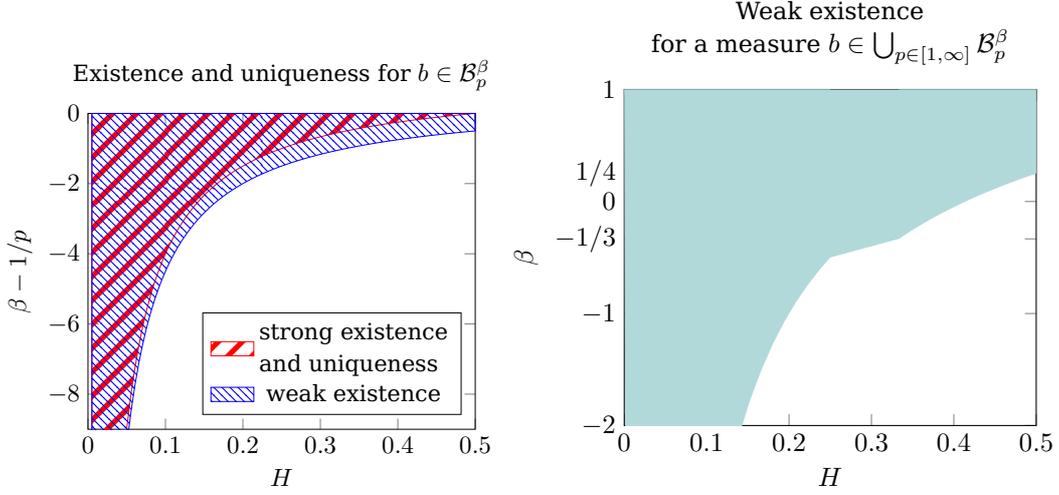

\subsection{Reformulation as a nonlinear Young equation} \label{reformulate}

Let $\mathit{w} \in \mathcal{C}_{[0,T]}$. Rewriting equation \eqref{eq:skewZ} with $Z\equiv \mathit{w}$ and $\tilde{X}_t=X_t-\mathit{w}_t$, we formally obtain
\begin{align} \label{eq:rewritten}
    \tilde{X}_t=\tilde{X}_0 + \int_0^t b(\tilde{X}_r+\mathit{w}_r) \, dr. 
\end{align}
For a bounded measurable function $b$, we define the averaging operator $T$ by 
\begin{align} \label{traditionalaveraging}
    T^\mathit{w}_t b(x):= \int_0^t b(x+\mathit{w}_r) \, dr, ~ \mbox{ for } (t,x) \in [0,T]\times \mathbb{R}.
\end{align} 
Assuming that $\mathit{w}$ has a local time $L$ and using the occupation time formula 
\eqref{occupation}, it comes that
\begin{align} \label{eq:motivation}
     T_{s,t}^\mathit{w} b(x)=\int_s^t b(x+ \mathit{w}_r) \, dr&=\int_\mathbb{R} b(x+z) L_{s,t}(z) \, dz 
    = b*\check{L}_{s,t}(x), \ \forall x \in 
    \mathbb{R},
\end{align}
where \(\check{L}_{s,t}(x) = L_{s,t}(-x)\).
This operator and its connection to the local time was already considered in  \cite{CatellierGubinelli}. 
In view of the expression of $T^\mathit{w} b$ as a convolution, one can expect that for 
$\mathit{w}$ with a sufficiently regular local time, the definition of $T^\mathit{w} b$ will 
extend to $b$ with lower regularity (we will consider suitable Besov spaces, see 
Section~\ref{sec:existWeak}). 
This idea was exploited in \cite{HarangPerkowski}, in the case of noises with infinitely differentiable local times.

Using \eqref{eq:rewritten}, we get that for $b$ continuous and bounded, $\tilde{X} \in \mathcal{C}_{[0,T]}$ and a sequence of partitions $\{t_i^n\}_{i=1}^{N_n}$ of $[0,t]$ with mesh size converging to $0$, 
\begin{align} \label{eq:reformulation}
    \int_0^t b(\tilde{X}_r+\mathit{w}_r)\, dr &= \lim_{n \rightarrow \infty} \sum_{i=1}^{N_n-1} 
    \int_{t_i^n}^{t_{i+1}^n} b(\tilde{X}_{t_i^n}+ \mathit{w}_r) \, dr \nonumber\\
    &= \lim_{n \rightarrow \infty} \sum_{i=1}^{N_n-1} T^{\mathit{w}}_{t_i^n,t_{i+1}^n} 
    b(\tilde{X}_{t_i}^n) \nonumber\\
    &= \int_0^t T_{dr}^{\mathit{w}} b(\tilde{X}_r) ,
\end{align}
where the final equality is only formal at this point. We give a rigorous definition of this 
integral in 
Section~\ref{nonlinearyoung} and call it a nonlinear Young integral. This also suggests to 
rewrite Equation \eqref{eq:skewZ} as a nonlinear Young integral equation. Thus we give 
another definition of a solution to \eqref{eq:skewZ} and \eqref{eq:skew}. Combining the 
theory of nonlinear Young integrals from Section~\ref{nonlinearyoung} and the extension of 
the averaging operator $T$ applied to distributions (see Section~\ref{sec:existWeak}) will give sense to 
the following definition.

\begin{definition} \label{def:solution2}
Let $Z$ be a stochastic process. We call $X:\Omega \rightarrow 
\mathcal{C}_{[0,T]}$ a path-by-path solution to \eqref{eq:skewZ} if there exists a null-set 
$\mathcal{N}$,  $\eta \in (0,1]$ and  $p,q \geqslant 1$  such that \(1/p 
+\eta/q>1\) and for any $\omega \notin 
\mathcal{N}$, $T^{Z(\omega)}b \in 
\mathcal{C}^{p\text{-}\text{var}}_{[0,T]}(\mathcal{C}^\eta)$, $X(\omega)-Z(\omega) 
\in \mathcal{C}^{q\text{-}\text{var}}_{[0,T]}$ and
\begin{align} \label{eq:definitionNLY}
    X_t(\omega)=X_0+\int_0^t T_{dr}^{Z(\omega)} b(X_r(\omega)-Z_r(\omega)) + Z_{t}(\omega),\text{ for all } t \in [0,T].
\end{align}
\end{definition}
The assumption involving $\eta \in (0,1]$ and $p,q \geqslant 1$ is simply the (sufficient) condition formulated in Theorem \ref{int} that ensures the existence of the nonlinear Young integral. Note that in the above definition no measurability or adaptedness of $X$ is required.

Theorem~\ref{generalsolution} gives conditions on the local time of the process $Z$ such that there exists a path-by-path solution to \eqref{eq:skewZ}. These conditions will be needed multiple times thoughout the paper, hence we give them here. 

\begin{assumption}\label{assumption21}
Let $\gamma, \eta \in (0,1), \beta \in \mathbb{R}$ and $p \in [1,\infty]$. Let $b \in \mathcal{B}^\beta_p$ and $\mathit{w} \in \mathcal{C}_{[0,T]}$ with local time $L$.
\begin{enumerate}[label=(\Roman*)]
\item \label{ass:1} There exists $\tilde{p} \in [1,\infty]$ with $1/p + 1/{\tilde{p}}>1$ and 
$L \in \mathcal{C}^\gamma_{[0,T]}(\mathcal{B}^{-\beta + \eta 
	+1/p+1/{\tilde{p}}-1}_{\tilde{p}})$;

\item \label{ass:2}There exists $\tilde{p} \in [1,\infty]$ with $1/p + 1/{\tilde{p}} \leqslant 
1$ and $L \in \mathcal{C}^\gamma_{[0,T]}(\mathcal{B}^{-\beta+\eta}_{\tilde{p}})$.
\end{enumerate}
\end{assumption}

\begin{theorem} \label{generalsolution}
Let $\gamma, \eta \in (0,1)$ with $\gamma + \eta>1$. 
Let $b \in \mathcal{B}^\beta_p$ be a measure with $\beta \in \mathbb{R}$, 
$p\in[1,\infty]$. Let $Z: \Omega \rightarrow \mathcal{C}_{[0,T]}$ be a continuous stochastic 
process with a local time $L^Z:\Omega \times [0,T]\times \mathbb{R} \rightarrow \mathbb{R}$. Assume that there exists a null-set $\mathcal{N}$ such that for any $\omega \notin 
\mathcal{N}$, $L^Z(\omega)$ satisfies \ref{ass:1} or \ref{ass:2}. Then there exists a 
path-by-path solution to Equation~\eqref{eq:skewZ}.
\end{theorem} 

Recall that Theorem~\ref{generalsolution} does not imply existence of a measurable/adapted solution. However in the case of fBm, using properties of its local time, additionally adaptedness of the path-by-path solution can be proven (see Theorem~\ref{prop:existence}).

The following theorem provides a comparison between solutions constructed by approximation with a smooth drift and solutions in the nonlinear Young sense. More precisely, it 
shows that being a solution to \eqref{eq:skewZ} (i.e. for a noise $Z$) in the sense of 
Definition~\ref{def:solution2} implies being a solution in the sense of Definition~\ref{def:solution}. Under some regularity restrictions, the reverse implication holds 
as well. We rephrase \eqref{solution1} and \eqref{approximation2} (without specifying the mode of convergence yet) for a noise $Z$ instead of $B$:
\begin{align}
X_t=X_0+K_t+Z_t\text{ for all } t \in [0,T] ; \label{solution1Z} \\
\sup_{t\in [0,T]}\left|\int_0^t b^n(X_r) dr -K_t\right|
\underset{n\rightarrow \infty}{\longrightarrow} 0  \label{approximation2Z}
\end{align}
for every sequence $(b^n)_{n \in \mathbb{N}}$ of smooth bounded functions converging to 
\(b\) in $\mathcal{B}^{\beta-}_p$.
\begin{theorem} \label{solutionsagree}
Let $\gamma, \eta \in (0,1)$ and $q\geqslant 1$ with $\gamma+\eta/q>1$, $X_0 \in 
\mathbb{R}$, $p\in[1,\infty]$, $\beta\in \R$, $b \in \mathcal{B}_p^{\beta}$ and $X:\Omega 
\rightarrow \mathcal{C}_{[0,T]}$. Let $Z: \Omega \rightarrow \mathcal{C}_{[0,T]}$ be a 
continuous stochastic 
process and $L^Z:\Omega \times [0,T]\times \mathbb{R} \rightarrow \mathbb{R}$ its local 
time.  Assume that there exists a null-set $\mathcal{N}$ such that for any $\omega \notin 
\mathcal{N}$, $L^Z(\omega)$ satisfies \ref{ass:1} or \ref{ass:2}.
\begin{enumerate}[label=(\alph*)]
\item \label{theorem:eq.a} Assume that $X$ is a path-by-path solution to \eqref{eq:skewZ} and that $X-Z \in \mathcal{C}^{q\text{-}\text{var}}$. Then for any sequence of smooth bounded functions $(b^{n})_{n\in \N}$ that converges to $b$ in $\mathcal{B}_{p}^{\beta-}$, $X_{\cdot}(\omega)-Z_{\cdot}(\omega)-X_{0} = K_{\cdot}(\omega)=\int_0^{\cdot} T^{Z(\omega)}_{dr}b(X_r(\omega)-Z_r(\omega))$ is the uniform limit of $\int_{0}^\cdot b^n(X_{r}(\omega))\, dr$ for all $\omega \notin \mathcal{N}$ (i.e. \eqref{approximation2Z} holds on the set of full measure $\mathcal{N}^{\mathsf{c}}$).
\item \label{theorem:eq.b} Assume that there exists a process $K: \Omega \rightarrow 
\mathcal{C}_{[0,T]}$ such that \eqref{solution1Z} and \eqref{approximation2Z} hold, where the convergence in \eqref{approximation2Z} is in probability. 
Assume further that a.s., $X(\omega)-Z(\omega) 
\in \mathcal{C}^{q\text{-}\text{var}}_{[0,T]}$. Then $X$ is a path-by-path solution to \eqref{eq:skewZ}.
\end{enumerate}
\end{theorem}

\begin{remark}
In particular, statement \ref{theorem:eq.a} implies that if $X$ is an $\mathbb{F}$-adapted path-by-path solution, it is a weak solution. Moreover, for the weak solution constructed in Theorem~\ref{prop:existence}, the convergence in \eqref{approximation2} holds on a set of full measure instead of convergence in probability. 
\end{remark}

\subsection{Organisation of the proofs}\label{subsec:OrgaProof}

In Section~\ref{nonlinearyoung}, we construct nonlinear Young integrals in 
$p$-variation \red{via a classical sewing argument}. In particular, we establish existence of solutions to a nonlinear Young integral 
equation with monotone drift under milder regularity constraints than in the non-monotone 
case. 

Then we rigorously rewrite Equation \eqref{eq:skew} as a nonlinear Young integral equation. The existence of a 
solution to the more general Equation \eqref{eq:skewZ}, when $Z$ has a sufficiently regular local time, is stated in Theorem~\ref{generalsolution} and proven in Section~\ref{subsec:proofs-pbp}. We then establish that a solution to 
\eqref{eq:skew} in this nonlinear Young sense is also a solution in the sense of 
Definition~\ref{def:solution} (see Theorem~\ref{solutionsagree}, which is proven in Section~\ref{subsec:proofs-pbp}). For instance, when applied to the case of a fractional Brownian noise (after investigating 
the regularity properties of its local time), these theorems imply the existence of path-by-path solutions. However, in Theorem~\ref{prop:existence}\ref{en:weak1} one wants to prove the existence of solutions that are adapted to the filtration $\mathbb{F}$ of the underlying filtered probability space. Theorems~\ref{generalsolution} and \ref{solutionsagree} cannot be applied directly and a modified approach is developed using the tightness of the averaging operator of the fBm (Lemma \ref{lem:LT2}) and the 
continuity of the operator \(\mathcal{A}\) transforming a fBm to a Bm, see 
\eqref{operatortildeA} and 
Lemma~\ref{lem:operatorcontinuity}. 
These 
arguments are given in Section 
\ref{sec:existenceweak} and lead to the proof of Theorem~\ref{prop:existence}\ref{en:weak1}.

In Section~\ref{sec:RegWeak}, we use some new regularity estimates on conditional expectations of the fBm (Lemma~\ref{lem:Cs}) and the stochastic sewing Lemma with random control (see Lemma~\ref{lem:stsrandomcontrols}) to establish that any weak solution $X$ satisfies $X-B\in \mathcal{C}^\kappa_{[0,T]}(L^m)$ for any $\kappa \in (0,1+H(\beta-\frac{1}{p})\wedge 0]\setminus \{1\}$ and $m\geqslant2$
 when $b$ is a measure in $\mathcal{B}^\beta_{p}$. This proves Theorem~\ref{prop:existence}\ref{en:weak2} and Theorem~\ref{thm:uniqueness}\ref{uniqueness(2)}.

 In Section~\ref{uniqueness}, in order to establish pathwise uniqueness of weak solutions to \eqref{eq:skew} (see Proposition~\ref{unique}), we adapt an 
 approach developed recently for the stochastic heat equation with distributional drift, see 
 \cite{Atetal}. This requires several regularity estimates on solutions which are derived from 
 the crucial regularity Lemma~\ref{regulINT} and the stochastic sewing Lemma with critical exponent (Theorem 4.5 in \cite{Atetal}). The proof of Lemma~\ref{regulINT} relies on the
 stochastic sewing 
 Lemma and the aforementioned regularity estimates on conditional expectations of the fBm (Lemma~\ref{lem:Cs}).

Theorem~\ref{th:approx} is proven in Section~\ref{thmproof} by an approximation of the drift with smooth bounded 
functions. The corresponding sequence of strong solutions will be shown to be tight  and furthermore reveal a stability property, such that we can identify the limit as a solution to \eqref{eq:skew}, where continuity of the operator linking fBm to Brownian motion (see Lemma~\ref{lem:operatorcontinuity}) is needed to prove adaptedness. This works thanks to \emph{a priori} regularity estimates of solutions, see 
Lemma~\ref{regularity3.2} and Lemma~\ref{lem:KHoelder}. 

Strong existence will follow by a Yamada-Watanabe-type argument, using Gy\"ongy-Krylov's 
lemma (\cite[Lem. 1.1]{Krylov}) and requires the uniqueness result Proposition~\ref{unique}. 
The proofs of Theorems~\ref{th:approx} and  
\ref{thm:uniqueness} are then completed in Section~\ref{mainresults}.

\section{Nonlinear Young integrals and nonlinear Young equations in 
$\mathcal{C}^{p\text{-}\text{var}}$} \label{nonlinearyoung}

\subsection{Construction of nonlinear Young integrals and properties}

Throughout this subsection, $E$ and $F$ denote arbitrary Banach spaces. Theorem~\ref{int} provides conditions for the existence of a nonlinear Young integral 
in terms of $p$-variations (rather than H\"older continuity as in  
\cite[Th.~2.4]{CatellierGubinelli} and  \cite{Galeati}).

\begin{theorem} \label{int}
Let $\eta \in (0,1]$ and $p,q \in [1,\infty)$ such that \(\theta:=1/p 
+\eta/q>1\). Let
$A \in 
\mathcal{C}^{p\text{-}\text{var}}_{[0,T]}(\mathcal{C}_{E,loc}^\eta(F))$ and $x 
\in 
\mathcal{C}^{q\text{-}\text{var}}_{[0,T]}(E)$.
Then for $(s,t) 
\in \Delta_{[0,T]}$ and any sequence of partitions 
$(\Pi_{n})_n$ of $[s,t]$ with \(\lim_n |\Pi_n| = 0\), the sum $\sum_{\Pi_{n}} 
A_{t_i,t_{i+1}}(x_{t_i})$ converges.
Besides, the 
limit is independent of the sequence of partitions \((\Pi_n)_{n}\). We denote it by $\int_s^t A_{dr}(x_r)$ and call it the nonlinear Young integral with respect to $A$ and $x$.

In addition, there exists $C(\theta)>0$ such that for $(s,t) \in \Delta_{[0,T]}$ and \(R > 0\) with $\|x\|_\infty \leqslant R$,  
one has
\begin{align}\label{eq:YoungEstimate}
    \left\|\int_s^t A_{dr}(x_r)-A_{s,t}(x_s)\right\|_F \leqslant C(\theta)
    \VCnReta{A}{[s,t]}{p}[x]_{\mathcal{C}_{[s,t]}^{q\text{-}\text{var}}}^\eta.
\end{align}
\end{theorem}

\begin{proof}
We apply the sewing lemma \cite[Theorem 2.2 and Remark 2.3]{Frizzhang} formulated with controls. Let $0 \leqslant u \leqslant v \leqslant w \leqslant T$. Then, for $R>0$ such that $\|x\|_{\infty}\leqslant R$,
\begin{align*}
|A_{u,w}(x_u)-A_{u,v}(x_u)-A_{v,w}(x_v)|&=|A_{v,w}(x_u)-A_{v,w}(x_v)|\\
&\leqslant [A_{v,w}]_{\mathcal{C}^\beta_{D_{R}}} [x]_{\mathcal{C}_{[u,w]}^{q\text{-}\text{var}}}^\eta\\
&\leqslant \VCnReta{A}{[u,w]}{p}[x]_{\mathcal{C}_{[u,w]}^{q\text{-}\text{var}}}^\eta,
\end{align*}
which gives the result as the last expression in the above  defines a control raised to the power $\theta$ by \cite[Exercise 1.9]{FrizVictoir}.
\end{proof}

As a corollary we obtain the following result.
\begin{corollary} \label{cor:approximation2}
Let $T>0$. Let $p,q \in [1,\infty)$ and $\eta \in (0,1]$ with $\theta := 1/p +\eta/q>1$. Let 
$x \in \mathcal{C}^{q\text{-}\text{var}}_{[0,T]}$ and $A \in \VCeta$. If a sequence 
$A^n$ converges to $A$ in $\VCeta$, then
\begin{align*}
    \sup_{t \in [0,T]}\left|\int_0^t (A^n-A)_{dr}(x_r)\right|\underset{n \rightarrow 
    \infty}{\longrightarrow} 0.
\end{align*}
\end{corollary}

\begin{proof}
By \eqref{eq:YoungEstimate}, we have 
\begin{align*}
    \left|\int_0^t (A^n-A)_{dr}(x_r)\right|&\leqslant \left|\int_0^t (A^n-A)_{dr}(x_r)-(A^n-A)_{0,t}(x_0)\right| + \left|(A^n-A)_{0,t}(x_0)\right|\\
    &\leqslant C(\theta) 
    \VCneta{A^n-A}{[0,T]}{p}[x]_{\mathcal{C}_{[0,T]}^{q\text{-}\text{var}}}^\eta + 
    \VCneta{A^n-A}{[0,T]}{p}.
\end{align*}
Taking the supremum over all $t \in [0,T]$ and letting $n$ go to $\infty$ gives the result.
\end{proof}

The next lemma establishes the sensitivity in the ``$x$'' variable of the nonlinear Young integral. See  \cite[Th. 2.7]{Galeati} for its counterpart in the H\"older setting.

\begin{lemma} \label{intineq}
Let $\eta \in (0,1]$ and $p,q \in [1,\infty)$ with $\theta:=1/p+\eta/q>1$. 
Let $A\in 
\mathcal{C}^{p\text{-}\text{var}}_{[0,T]}(\mathcal{C}_{E,loc}^\eta(F))$,  $x, y \in 
\mathcal{C}^{q\text{-}\text{var}}_{[0,T]}(E)$. We denote \(R = \Max(\|x\|_\infty, \|y\|_\infty)\). Then, for $\delta\in (q(1-\frac{1}{p}), \eta)$ and $(s,t) \in \Delta_{[0,T]}$, one has
\begin{align} \label{eq:intineq}
 \left\| \int_s^t A_{dr}(x_r) - \int_s^t A_{dr}(y_r)\right\|_F \leqslant
 C(\theta)  \VCnReta{A}{[s,t]}{p} 
 &\left([x]_{\mathcal{C}_{[s,t]}^{q\text{-}\text{var}}} +  
 [y]_{\mathcal{C}_{[s,t]}^{q\text{-}\text{var}}}\right)^{\delta}\|x-y\|_\infty^{\eta-\delta} \nonumber\\
 &+ 
 \|A_{s,t}\|_{\mathcal{C}^\eta_{D_R}} \|x-y\|_\infty^\eta.
\end{align}
\end{lemma}

\begin{proof} 
For $(u,w) \in \Delta_{[0,T]}$, let
\begin{align*}
    \Gamma_{u,w}\coloneqq(A_{u,w}(x_u) -A_{u,w}(y_u)).
\end{align*}
Again, we aim to apply the sewing lemma. For any $z \in 
\mathcal{C}^{q\text{-}\text{var}}_{[0,T]}(E)$, let 
$\varkappa_z(u,w)^\theta:=\VCnReta{A}{[u,w]}{p}[z]_{\mathcal{C}_{[u,w]}^{q\text{-}\text{var}}}^\eta$. Furthermore, let
\begin{align*}
		\hat{\varkappa}(u,w)^{\delta/q+p^{-1}}:=(\varkappa_x(u,w)+\varkappa_y(u,w))^{\tfrac{\delta}{\eta} 
		\theta}\Big(\VCnReta{A}{[u,w]}{p}\Big)^{1-\frac{\delta}{\eta}}.
\end{align*}
By \cite[Exercise 1.9]{FrizVictoir} and by definition of a control, we get that $\varkappa_x,\,  \varkappa_y,\, \varkappa_x+\varkappa_y$ and $\hat{\varkappa}$ are control functions. %
Note that, for $v \in [u,w]$,
\begin{align*}
    \left\|\Gamma_{u,w}-\Gamma_{u,v}-\Gamma_{v,w}\right\|_F&\leqslant \|A_{v,w}(x_v)-A_{v,w}(x_u)\|_F + \|A_{v,w}(y_v)-A_{v,w}(y_u)\|_F\\
    &\leqslant \VCnReta{A}{[v,w]}{p}([x]_{\mathcal{C}_{[u,v]}^{q\text{-}\text{var}}}^\eta 
    +[y]_{\mathcal{C}_{[u,v]}^{q\text{-}\text{var}}}^\eta)\\
    &\leqslant (\varkappa_x(u,w)+\varkappa_y(u,w))^\theta .
\end{align*}
Furthermore, we also have that
\begin{align*}
    \left\|\Gamma_{u,w}-\Gamma_{u,v}-\Gamma_{v,w}\right\|_F&\leqslant \|A_{v,w}(x_v)-A_{v,w}(y_v)\|_F + \|A_{v,w}(x_u)-A_{v,w}(y_u)\|_F\\
    &\leqslant 2 \VCnReta{A}{[v,w]}{p} \|x-y\|_\infty^\eta.
\end{align*} 
It clearly holds true that $a\wedge b\leqslant a^\xi b^{1-\xi}$ for $a,b\geqslant 0$ and $\xi 
\in [0,1]$. Hence, for $\xi = \delta/\eta \in (0,1)$, it comes
	\begin{align*}
		\left\|\Gamma_{u,w}-\Gamma_{u,v}-\Gamma_{v,w}\right\|_F
		\leqslant C \hat{\varkappa}(u,w)^{\frac{\delta}{q}+p^{-1}}\, \|x-y\|_\infty^{\eta 
		(1-\frac{\delta}{\eta})}.
	\end{align*}
Applying the sewing lemma and the inequality
\begin{align*}
\|\Gamma_{u,w}\|_F \leqslant \|A_{u,w}\|_{\mathcal{C}^\eta_{D_R}} \|x-y\|_\infty^\eta%
\end{align*}
gives \eqref{eq:intineq}.
\end{proof}

\subsection{Solving nonlinear Young equations}

We are now ready to state a result of existence of solutions to nonlinear Young
 integral equations with a positive drift. From now on we will work with real-valued functions and vector fields. 
 The proof resembles the one in the H\"older setting, see~\cite[Th. 3.2]{Galeati}.
However, there is one crucial step where the $p$-variation permits to take into account the nonnegative drift and which then allows to assume milder regularity conditions than in the H\"older setting.

\begin{theorem} \label{thm:existence}
Let $\eta \in (0,1]$ and $p\geqslant 1$ with $1/p + \eta >1$. Let $A\in \VCeta$ 
with $A_{s,t}(y)\geqslant 0$
for all $y\in \mathbb{R}$ and all $(s,t) \in \Delta_{[0,T]}$. Then for any $x_0\in 
\mathbb{R}$, there exists a solution $x\in \mathcal{C}^{1\text{-}\text{var}}_{[0,T]}$ to 
the nonlinear 
Young equation
\begin{align}\label{eq:NYE}
    x_t=x_0+\int_0^t A_{dr}(x_r), \ \forall t \in [0,T].
\end{align}
\end{theorem}

\begin{proof}
Without loss of generality, let $T=1$. Let $\theta:=1/p+\eta$. 
For $n \in \mathbb{N}$ and $0 \leqslant k \leqslant n$, let $t_k^n:=k/n$, 
$\bar{x}_0^n:=x_0$ and define recursively
\begin{align*}
 \bar{x}_{k+1}^n=\bar{x}_k^n+ A_{t_k^n,t_{k+1}^n}(\bar{x}_k^n). 
\end{align*}
We embed $(\bar{x}_k^n)_{k=0}^n$ into $\mathcal{C}_{[0,1]}$ by setting
\begin{align*} 
    x_t^n=x_0 + \sum_{0\leqslant k \leqslant \lfloor nt \rfloor} A_{t_k^n,t \wedge 
    t_{k+1}^n}(\bar{x}_k^n),
\end{align*}
which we write as follows
 \begin{align} \label{drei}
     x_t^n=x_0 + \int_0^t A_{dr}(x_r^n) + \sum_{0\leqslant k \leqslant \lfloor nt \rfloor} \left( \int_{t_k^n}^{t\wedge t_{k+1}^n} A_{dr}(x_{t_k^n}^n)-A_{dr}(x_r^n)\right).
 \end{align}
Denote the sum in \eqref{drei} by $\psi_t^n$. Let us introduce the control function 
 \begin{align*}
    \varkappa^n(s,t):=\left(\VCneta{A}{[s,t]}{p}\Vn{x^n}{[s,t]}^\eta\right)^{\frac{1}{\theta}}.
\end{align*}
Then using \eqref{eq:YoungEstimate} and the superadditivity property \eqref{eq:superadditive} of a control,
\begin{align} \label{psi}
    |\psi_t^n|&=\left|\sum_{0\leqslant k \leqslant \lfloor nt \rfloor} \left( \int_{t_k^n}^{t\wedge t_{k+1}^n} A_{dr}(x_{t_k^n}^n)-A_{dr}(x_r^n)\right)\right| \nonumber\\
    &\leqslant C(\theta) \sum_{0\leqslant k \leqslant \lfloor nt \rfloor} \varkappa^n(t_k^n,t\wedge 
    t_{k+1}^n)^\theta 
    \nonumber\\ 
    &\leqslant C(\theta)\, \varkappa^n(0,t) \, \max_{0\leqslant k \leqslant \lfloor nt 
    \rfloor}\{\varkappa^n(t_k^n,t\wedge t_{k+1}^n)\}^{\theta -1} .
\end{align}

Let \(0\leqslant s\leqslant u \leqslant T =1\). We note \(i =  \lfloor ns \rfloor \) and 
\(i + j = \lfloor nu \rfloor\). We assume that $j>0$ as the case $j=0$ is simpler. Then
\begin{align*}
	|\psi_{s,u}^n|&\leqslant 
	\left|\int_s^{t_{i+1}^n}A_{dr}(x_{t_{i}^n}^n)-A_{dr}(x_r^n)\right|+\sum_{k=1}^{j}\left|
	\int_{t_{i+k}^n}^{u\wedge t_{i+k+1}^n} 
	A_{dr}(x^n_{t_{i+k}^n})-A_{dr}(x_r^n)\right|\\
	&\leqslant 	
\left|\int_s^{t_{i+1}^n}A_{dr}(x_s^n)-A_{dr}(x_r^n)\right|+|A_{s,t_{i+1}^n}(x_s^n)-A_{s,t_{i+1}^n}(x^n_{t_{i}^n})|\\
	 &\hspace{70pt}+\sum_{k=1}^{j}\left| \int_{t_{i+k}^n}^{u\wedge t_{i+k+1}^n} 
	A_{dr}(x^n_{t_{i+k}^n})-A_{dr}(x_r^n)\right|,
\end{align*}
so by the estimate \eqref{eq:YoungEstimate} 
it comes
\begin{align} \label{st}
    |\psi_{s,u}^n|    &\leqslant C(\theta) \left(\varkappa^n(s,u)^\theta+ \VCneta{A}{[t_i^n,t_{i+1}^n]}{p} \Vn{x^n}{[t_i^n,t_{i+1}^n]}^\eta \right).
\end{align}
We now look for a bound on the $1$-variation norm of $x^n$. Due to the non-negativity assumption on $A$, $x^n$ is non-decreasing. Hence
\begin{align*}
    \Vn{x^n}{[t^n_i,t^n_{i+1}]}=|x_{t^n_i,t^n_{i+1}}^n|&= |A_{t^n_i,t^n_{i+1}}(x^n_{t_i^n})|\leqslant \VCneta{A}{[t_i^n,t_{i+1}^n]}{p}\, .
\end{align*}
Then in view of \eqref{drei}, \eqref{st} and applying \eqref{eq:YoungEstimate} to $ \int_s^u A_{dr}(x_r^n)$, we get
\begin{align*}
    \Vn{x^n}{[s,u]}&\leqslant C(\theta)\left(\left(|A_{s,u}(x_s^n)| + \varkappa^n(s,u)^\theta \right) + \left(\varkappa^n(s,u)^\theta + \VCneta{A}{[t_i^n,t_{i+1}^n]}{p}^{1+\eta}\right)\right).
\end{align*}
Let now $\vep>0$ and choose $n$ such that 
$\VCneta{A}{[t_i^n,t_{i+1}^n]}{p}<\vep^{1/(1+\eta)}$ 
 for all $i=0,...,n-1$. 
 Then one obtains
\begin{align*}
    \Vn{x^n}{[s,u]}&\leqslant 2 C(\theta)\left(\VCneta{A}{[s,u]}{p} + \VCneta{A}{[s,u]}{p}\Vn{x^n}{[s,u]}^\eta + \vep\right).
\end{align*}
Using that $a^\eta\leqslant 1+a$ for $a\geqslant 0$ and $\eta \in [0,1]$, it follows that
\begin{align*}
    \Vn{x^n}{[s,u]}&\leqslant 2 C(\theta)\left( \VCneta{A}{[s,u]}{p}(1+\Vn{x^n}{[s,u]}) + \varepsilon\right).
\end{align*}
Hence, for $s\leqslant u$ such that $2 C(\theta)\VCneta{A}{[s,u]}{p}<1$, 
this leads to
\begin{align} \label{ineq:xn}
    \Vn{x^n}{[s,u]}\leqslant 2 C(\theta)\frac{\VCneta{A}{[s,u]}{p}+\varepsilon}{1-2C(\theta)\VCneta{A}{[s,u]}{p}},
\end{align}
which gives both uniform boundedness and equicontinuity.
In view of Equation \eqref{psi}, the uniform boundedness of $\Vn{x^n}{[0,t]}$ combined with 
the continuity of $\VCneta{A}{[0,t]}{p}$ give that $\psi^n$ converges to $0$ uniformly on 
$[0,1]$.  By the Arzel\`a-Ascoli Theorem, we deduce that $x^n$ converges uniformly along a 
subsequence to some non-decreasing $x \in \mathcal{C}_{[0,1]}$. Without loss of generality, 
we still denote by $x^n$ this subsequence. Then using the uniform boundedness of 
$\Vn{x^n}{[0,1]}$ and the uniform convergence of $x^n$ to $x$, one gets from 
Lemma~\ref{intineq} that for any $t$,
\begin{align*}
\int_0^t A_{dr}(x^n_r)\underset{n \rightarrow \infty}{\longrightarrow} \int_0^t A_{dr}(x_r) .
\end{align*}
Hence, passing to the limit in \eqref{drei}, we finally obtain that $x$ solves \eqref{eq:NYE}.
\end{proof}

\begin{remark}\label{rem:extendholder}
Theorem~\ref{thm:existence} extends to mappings $A$ which are only locally H\"older in 
space, giving existence possibly up to a blow-up time. Its proof can be done using typical 
localisation arguments.
\end{remark}

Since we are interested in differential equations perturbed by noise, it is natural to look for an 
extension of Theorem~\ref{thm:existence} in case $A$ is random, and to look for a measurable solution. %
Therefore we conclude this section with an extension of Theorem~\ref{thm:existence} for random $A$ . We omit the proof as it is similar to the first part of the proof of Theorem~\ref{prop:existence} which is presented in Section~\ref{sec:existenceweak} (i.e. showing tightness and then using Skorokhod's representation Theorem to pass to an almost sure limit).
\begin{corollary} \label{cor:measurable}
Let $\eta \in (0,1]$ and $p\geqslant 1$ with $1/p + \eta >1$. 
Let $A:\Omega \rightarrow \mathcal{C}^{p\text{-}\text{var}}_{[0,T]}(\mathcal{C}^\eta)$ 
be a random variable such that 
almost surely, $A_{s,t}(y)\geqslant 0$ 
for all $y\in \mathbb{R}$ and $(s,t) \in \Delta_{[0,T]}$. Furthermore, assume that for any $\lambda>0$, 
\begin{align} \label{Adelta}
\lim_{\delta\rightarrow 0}\mathbb{P}(\Omega_{\delta,\lambda})=1, \text{ where } 
\Omega_{\delta,\lambda}:=\{\omega:\sup_{|t-s|<\delta} 
[A]_{\mathcal{C}^{p\text{-}\text{var}}_{[s,t]}(\mathcal{C}^\eta)}<\lambda\}.
\end{align}
Then for any $Y_0\in \mathbb{R}$, there exists a probability space 
$(\tilde{\Omega},\tilde{\mathcal{E}},\tilde{\mathbb{P}})$, a measurable map 
$\tilde{A}$ which satisfies $\text{Law}(\tilde{A})=\text{Law}(A)$ and a measurable 
map $Y:\tilde{\Omega} \rightarrow \mathcal{C}^{1\text{-}\text{var}}_{[0,T]}$ such that 
almost surely,
\begin{align*}
    Y_t=Y_0+\int_0^t \tilde{A}_{dr}(Y_r), \ \forall t \in [0,T] .
\end{align*}
\end{corollary}

\section{Existence of weak solutions with nonnegative drift}\label{sec:existWeak}

We prove Theorem~\ref{prop:existence}, Corollary 
\ref{cor:existence}, Theorem~\ref{generalsolution} and Theorem~\ref{solutionsagree} in this 
section.
First, we extend the averaging operator defined in \eqref{traditionalaveraging} to distributions
$b$ in Besov spaces. Obtaining some H\"older regularity properties for this object allows to 
prove 
Theorem~\ref{solutionsagree}, which states roughly that solutions to \eqref{eq:skew} 
in the sense of Definition~\ref{def:solution} are equivalent to solutions in the sense 
of Definition~\ref{def:solution2}. Hence, we then only consider solutions in the sense of Definition~\ref{def:solution2}. 

The remaining subsections are dedicated to the proof of Theorem~\ref{prop:existence}, 
using some fine results on the joint regularity in time and space of the local time of fractional 
Brownian motion and the results on nonlinear Young equations from Section 
\ref{nonlinearyoung}.

\subsection{Definition and properties of the averaging 
operator}\label{subsec:defofaverop}
In this section we give the definition of the averaging operator $T^{\mathit{w}} b$ for distributional $b$ and $\mathit{w} \in \mathcal{C}_{[0,T]}$, extending the construction of Section~\ref{reformulate}. Note that this was already done in a very general setup in \cite[Section 3.1]{GaleatiGubinelli}. We take a less general approach by directly mollifying $b$, which is in line with the definition in \cite{GaleatiGubinelli} by Lemma 3.9 therein. %
\begin{definition}\label{averagingoperator}
	Let $\mathit{w} \in \mathcal{C}_{[0,T]}$. Let $\beta \in \mathbb{R}$ and $p \in [1,\infty]$,  $b \in \mathcal{B}_{p}^\beta$.
	The averaging operator is defined by
	\begin{equation*}
		T^\mathit{w}_t b(x):=\lim_{n\to \infty} T^{\mathit{w}}_t b^n(x),
	\end{equation*}
	if the limit exists for any sequence {\color{black} \((b^n)\)} of smooth bounded functions 
	converging to $b$ in $\mathcal{B}^{\beta-}_p$ and is independent of the choice of the 
	sequence.
\end{definition}

\begin{lemma} \label{lem:mollification}
Consider $\beta,\tilde{\beta} \in \mathbb{R}$ with $\beta-\tilde{\beta} \in (0,1]$, $\gamma \in (0,1]$, $p \in [1,\infty]$,  $b \in \mathcal{B}_{p}^\beta$ and $\mathit{w} \in \mathcal{C}_{[0,T]}$. Assume that  $\mathit{w}$ has a local time $L \in 
	\mathcal{C}^\gamma_{[0,T]}(\mathcal{B}^{-\tilde{\beta}}_{p^\prime})$.
For the above choice of $b$ and ${\mathit{w}}$, $T^{\mathit{w}} b$ is well-defined in 
$\mathcal{C}^\gamma(\mathcal{C}^{\rho})$ for any $\rho \in (0,\beta-\tilde{\beta})$.
\end{lemma}

\begin{remark} \label{rem:bracket}
Recall from \cite[Prop. 2.76]{BaDaCh} that for $\eta \in \mathbb{R}$, $p,p^\prime,q,q^\prime \in [1,\infty]$ with 
$1/p+1/{p^\prime}=1$ and $1/q+1/{q^\prime}=1$, there is a continuous bilinear functional ${\langle\cdot,\cdot\rangle: \mathcal{B}_{p,q}^\eta \times 
\mathcal{B}^{-\eta}_{p^\prime,q^\prime} \rightarrow \mathbb{R}}$ extending the $L^2$ inner product.
\end{remark}

\begin{proof}[Proof of Lemma~\ref{lem:mollification}]

Let $\varepsilon=\beta-\tilde{\beta}$. After an embedding of Besov spaces (see Remark~\ref{embedding2}), we know that $b \in \mathcal{B}^{\beta-\varepsilon/2}_p$ and $L \in \mathcal{C}^\gamma_{[0,T]}(\mathcal{B}^{-\tilde{\beta}-\varepsilon/2}_{p^\prime,1})$. Let $(b^n)_{n \in \mathbb{N}}$ be any sequence of smooth bounded functions converging to $b$ in $\mathcal{B}^{\beta -}_p$. By the convolutional representation \eqref{eq:motivation} and Remark~\ref{rem:bracket}, we have that
\begin{align*}
T^{\mathit{w}}_t b^n(x)=\langle b^n,L_t(\cdot-x)\rangle.
\end{align*}
Using the continuity of the bilinear functional $\langle \cdot, \cdot \rangle$, we get, for $n,m \in \mathbb{N}$,
\begin{align*}
\|T_{\cdot}^{\mathit{w}} b^n(\cdot)-T_{\cdot}^{\mathit{w}} b^m(\cdot)\|_\infty\leqslant C 
\|L\|_{\mathcal{C}^\gamma (\mathcal{B}^{-\tilde{\beta}-\varepsilon/2}_{p^\prime,1})} 
\|b^n-b^m\|_{\mathcal{B}^{\beta-\varepsilon/2}_p}.
\end{align*}
Therefore $T^{\mathit{w}} b^n$ forms a Cauchy sequence and is uniformly convergent. Hence, $T^{\mathit{w}} b$ is well defined and is easily seen to be independent of the approximating sequence. 

Let $\rho \in (0,\beta-\tilde{\beta})$. In order to show that $T^{\mathit{w}} b$ actually has the 
required regularity, we have to check that 
\begin{align} \label{eq:fourterms}
\|T^\mathit{w}b\|_{\mathcal{C}_{[0,T]}^{\gamma}(\mathcal{C}^{\rho})}&= \sup_{s\neq t, x \neq y}\frac{|T^\mathit{w}_{s,t}b(y)-T^\mathit{w}_{s,t}b(x)|}{|t-s|^{\gamma} |x-y|^{\rho}} + \sup_{s \neq t}\sup_{x} \frac{|T^\mathit{w}_{s,t}b(x)|}{|t-s|^{\gamma}}\\
&\quad \quad +\sup_{t} \sup_{x \neq y} \frac{|T^\mathit{w}_t b(y)-T^\mathit{w}_t b(x)|}{|x-y|^{\rho}}+ \sup_{t}\sup_{x} |T^\mathit{w}_t b(x)| \nonumber
\end{align}
is finite. Fix $s \neq t$ and $x \neq y$. For any $n \in \mathbb{N}$,
\begin{align*}
|T^{\mathit{w}}_{s,t} b(x)-T^{\mathit{w}}_{s,t}b(y)|\leqslant 2\|T_{s,t}^{\mathit{w}} b(\cdot)- T_{s,t}^{\mathit{w}} b^n(\cdot)\|_\infty+|T^{\mathit{w}}_{s,t} b^n(x) -T^{{\mathit{w}}}_{s,t} b^n(y)|.
\end{align*}
Choosing $n=n(s,t,x,y)$ large enough, we have
\begin{align*}
\|T_{s,t}^{\mathit{w}} b(\cdot)- T_{s,t}^{\mathit{w}} b^n(\cdot)\|_\infty \leqslant |t-s|^\gamma |x-y|^{\beta-\tilde{\beta}} .
\end{align*}
Moreover, using continuity of the bilinear form, an embedding of Besov spaces and Lemma~\ref{A.2}\ref{A.5}, we have, for $\varepsilon=\beta-\tilde{\beta}-\rho>0$,
\begin{align*}
|T^{\mathit{w}}_{s,t} b^n(x) -T^{{\mathit{w}}}_{s,t} b^n(y)|&\leqslant C \|b^n\|_{\mathcal{B}^{\beta-\varepsilon/2}_{p}}\|L_{s,t}(\cdot-x)-L_{s,t}(\cdot-y)\|_{\mathcal{B}^{-\beta+\varepsilon}_{p^\prime}}\\
&\leqslant C |t-s|^\gamma |x-y|^{\rho} \|b\|_{\mathcal{B}^\beta_{p}} \|L\|_{\mathcal{C}^\gamma 
(\mathcal{B}^{-\tilde{\beta}}_{p^\prime})}.
\end{align*}
As the other terms in \eqref{eq:fourterms} can be controlled similarly, the result follows.
\end{proof}

\begin{lemma} \label{lem:regTL}
	Let $p,\tilde{p} \in [1,\infty]$, $\beta \in \mathbb{R}$, $\eta, \gamma 
	\in (0,1)$ and $b \in \mathcal{B}^\beta_p$. Assume that the local time $L$ of $\mathit{w}$ satisfies 
	Assumption~\ref{assumption21} \ref{ass:1} or \ref{ass:2}. Then,	
for any $\vep\in (0,\eta)$, ${T^{\mathit{w}}b \in 
		\mathcal{C}^\gamma_{[0,T]}(\mathcal{C}^{\eta-\varepsilon})}$. 
		Moreover, for any bounded open interval $\mathcal{I}$ that contains the (compact) support of $L$, there exists a constant $C_{\mathcal{I}}$ such that:
		\begin{itemize}
		\item If \ref{ass:1} holds, then $ 
		[T^{\mathit{w}}b]_{\mathcal{C}^\gamma_{[0,T]}(\mathcal{C}^{\eta-\vep})} \leqslant C_{\mathcal{I}} \|b\|_{\mathcal{B}^\beta_p} 
		[L]_{\mathcal{C}^\gamma_{[0,T]}(\mathcal{B}_{\tilde{p}}^{-\beta+\eta+1/{p}+1/{\tilde{p}}
		 - 1 
				})}$ and \\ $ 
		\|T^{\mathit{w}}b\|_{\mathcal{C}^\gamma_{[0,T]}(\mathcal{C}^{\eta-\vep})} \leqslant C_{\mathcal{I}}  \|b\|_{\mathcal{B}^\beta_p} 
		\|L\|_{\mathcal{C}^\gamma_{[0,T]}(\mathcal{B}_{\tilde{p}}^{-\beta+\eta+1/{p}+1/{\tilde{p}}
		 - 1 
				})}$;
			
		\item If \ref{ass:2} holds, then $[T^{\mathit{w}}b]_{\mathcal{C}^\gamma_{[0,T]}(\mathcal{C}^{\eta-\vep})} \leqslant C_{\mathcal{I}} \|b\|_{\mathcal{B}^\beta_p}  
		[L]_{\mathcal{C}^\gamma_{[0,T]}(\mathcal{B}_{\tilde{p}}^{-\beta+\eta})}$ and \\ $\|T^{\mathit{w}}b\|_{\mathcal{C}^\gamma_{[0,T]}(\mathcal{C}^{\eta-\vep})} \leqslant C_{\mathcal{I}}  \|b\|_{\mathcal{B}^\beta_p} 
		\|L\|_{\mathcal{C}^\gamma_{[0,T]}(\mathcal{B}_{\tilde{p}}^{-\beta+\eta})}$.
		
		\end{itemize}
\end{lemma}

\begin{remark}\label{rem:regTL}
In particular Lemma~\ref{lem:regTL} shows the following: Assume \ref{ass:1} or \ref{ass:2} 
holds for $\gamma,\eta \in (0,1)$ with $\gamma+\eta>1$. Then $T^{\mathit{w}}b \in 
\mathcal{C}_{[0,T]}^{\gamma}(\mathcal{C}^{\tilde{\eta}})$ for $\tilde{\eta} 
\in (0,\eta)$ with $\gamma+\tilde{\eta}>1$.
\end{remark}

\begin{proof}
The proof follows along the same lines as Lemma~\ref{lem:mollification}, making use of several Besov embeddings, including
 $\mathcal{B}_{\tilde{p}}^{-\beta+\eta+1/{p}+1/{\tilde{p}} - 1}(\mathcal{I}) 
 \hookrightarrow \mathcal{B}^{-\beta+\eta -\varepsilon}_{p^\prime,1}(\mathcal{I})$ (see 
 Remark~\ref{embedding3}).%
\end{proof}

\subsection{Path-by-path solutions: Existence and comparison of solutions}\label{subsec:proofs-pbp}

In this section we prove Theorem~\ref{generalsolution} on the existence of path-by-path solutions to Equation \eqref{eq:skewZ} and Theorem~\ref{solutionsagree} on the comparison of solutions.

\begin{proof}[Proof of Theorem~\ref{generalsolution}]

The idea of the proof is to identify a set of full measure on which $T^Z b$ is sufficiently regular and has nonnegative increments. Then, for $\omega$ in this set, we can apply the (deterministic) theory of nonlinear Young integral equations developed in Section~\ref{nonlinearyoung}. In particular, for any such $\omega$, we use Theorem~\ref{thm:existence} to pick a solution. As the proof of Theorem~\ref{thm:existence} is non-constructive (it relies on Arzel\`a-Ascoli's theorem), the axiom of choice is needed to pick a solution simultaneously for all such $\omega$. For $\omega$ outside of this full-measure set, we can define the solution to be identically $0$. Note that this construction is indeed done in a path-by-path sense so that the solution solves \eqref{eq:skewZ} on a set of full measure. However, \emph{a priori} there is no reason why the constructed solution should be adapted.

Let $0\leqslant s <t$. Provided that
\begin{align}
&T^{Z} b \in 
\mathcal{C}^\gamma_{[0,T]}(\mathcal{C}^{\tilde{\eta}}) \text{ for } \gamma,\tilde{\eta} \in (0,1) \text{ with } \gamma+\tilde{\eta}>1 ~ \text{ and } \label{ass:regularity} \\
&T^{Z}_{s,t}b(x)\geqslant 0 \text{ for all } x \in \mathbb{R}\label{ass:positivity}
\end{align}
hold a.s., using \(\mathcal{C}^\gamma_{[0,T]}(\mathcal{C}^{\tilde{\eta}})
\subset \mathcal{C}^{1/\gamma\text{-}\text{var}}_{[0,T]}(\mathcal{C}^{\tilde{\eta}})\),
Theorem~\ref{thm:existence} will give a solution.

Under the assumptions of Theorem~\ref{generalsolution}, Lemma~\ref{lem:regTL} ensures 
that \eqref{ass:regularity} holds.
To see that \eqref{ass:positivity} holds, let $\varepsilon\in (0,\eta)$ and 
$\delta:=\eta-\varepsilon$. Both under Assumption~\ref{assumption21} 
\ref{ass:1} or \ref{ass:2}, we have that $L^Z_{s,t} \in  
\mathcal{B}_{p^{\prime},1}^{-\beta+\delta}$ and $L^Z_{s,t} \in  
\mathcal{B}_{p^{\prime}}^{-\beta+2\delta/3}$ by the Besov embeddings in Remark~\ref{embedding2} and Remark~\ref{embedding3}. By Remark~\ref{rem:approximation}, the 
sequence of nonnegative function $\phi^x_n:=G_{1/n}L^Z_{s,t}(\cdot - x)$ converges to  
$L^Z_{s,t}(\cdot-x)$ in 
$\mathcal{B}_{p^{\prime}}^{(-\beta+2\delta/3)-}$. Hence
\begin{align}\label{eq:approxLT}
\lim_n\|L^Z_{s,t}(\cdot-x)-\phi^x_n\|_{\mathcal{B}_{p^{\prime},1}^{-\beta}}\leqslant C \, 
\lim_n\|L^Z_{s,t}(\cdot-x)-\phi^x_n\|_{\mathcal{B}_{p^{\prime}}^{-\beta+\delta/3}}=0.
\end{align}
\red{Then we get, for any sequence of nonnegative smooth bounded functions $b^m$ converging to $b$ in $\mathcal{B}^{\beta-}_p$,
\begin{align*}
T^Z_{s,t} b(x)&=\lim_{m \rightarrow \infty} T^Z_{s,t} b^m(x)\\
&=\lim_{m \rightarrow \infty} \lim_{n \rightarrow \infty} \langle b^m,\phi^x_n\rangle,
\end{align*}
where the first line holds true by definition and the second line by the continuity of $\langle \cdot, \cdot \rangle$ and \eqref{eq:approxLT}.
The inequality 
\eqref{ass:positivity} now follows from the fact that $\langle b^m, \phi^x_n\rangle \geqslant 0$, 
for any $n,m$  and \(x\).}
\end{proof}

\begin{proof}[Proof of Theorem~\ref{solutionsagree}]
\ref{theorem:eq.a}: 
Let $\varepsilon \in (0,\eta)$ such that $\gamma+(\eta-\varepsilon)/q>1$. 
As Assumption~\ref{assumption21}  \ref{ass:1} or \ref{ass:2} is fulfilled, 
we know by Lemma~\ref{lem:regTL} that 
$T^Z b\in \mathcal{C}^\gamma_{[0,T]}(\mathcal{C}^{\eta-\varepsilon})$ \red{on $\mathcal{N}^{\mathsf{c}}$} (which we do not mention in the rest of the proof, although all equalities and 
convergences happen on this set). Then, by the definition of the 
nonlinear Young integral of $T^Z b^n$ in Theorem~\ref{int}, we get that 
Equality \eqref{eq:reformulation} does hold for $b^n$:
\begin{align} \label{eq:Tsmooth}
\int_0^t b^n(X_r) dr = \int_0^t T_{dr}^Z b^n(X_r-Z_r).
\end{align}
By Lemma~\ref{lem:regTL}, we get that $T^Z b^n$ converges to $T^Z b$ in 
$\mathcal{C}^\gamma_{[0,T]}(\mathcal{C}^{\eta-\varepsilon})$. Hence, by 
Corollary~\ref{cor:approximation2}, we obtain that 
\begin{align*}
    \sup_{t \in [0,T]}&\Big|\int_0^t b^n(X_r) dr-\int_0^t T_{dr}^Z b(X_r-Z_r)\Big|
    \underset{n\rightarrow \infty}{\longrightarrow} 0.
\end{align*}
Therefore the convergence in the statement holds for 
$K_t:=\int_0^t T_{dr}^Z b(X_r-Z_r)$. %

\ref{theorem:eq.b}: By \eqref{approximation2Z}, we know that for some subsequence $(n_k)_{k \in \mathbb{N}}$,
\begin{align*}
    \sup_{t \in [0,T]} \Big| \int_0^t b^{n_k}(X_r) dr -K_t\Big| \longrightarrow 0 \text{ a.s.}
\end{align*}
We have again that \eqref{eq:Tsmooth} holds true and we deduce that
\begin{align*}
    \sup_{t \in [0,T]} \Big| \int_0^t T_{dr}^Z b^{n_k}(X_r-Z_{r})-K_t\Big| \longrightarrow 0 \text{ a.s.}
\end{align*}
By Lemma~\ref{lem:regTL} and Corollary~\ref{cor:approximation2} , we get 
that almost surely, $K_t= \int_0^t T_{dr}^Z b(X_r-Z_{r})$ for all $t 
\in [0,T]$. Equation \eqref{eq:definitionNLY} now follows from 
\eqref{solution1Z}.
\end{proof}

\subsection{Joint regularity of the local time of the fractional Brownian motion}

In the rest of this section, $L$ denotes the local time of a one-dimensional fBm 
 \(B\) and $T^B$ denotes the averaging operator associated to it, as 
constructed in Section~\ref{subsec:defofaverop}.

First, we recall Theorem 3.1 from \cite{HuLe}. This result is stated in \cite{HuLe} for a compact hypercube with side length equal to one. By dilatation, the result also holds for an arbitrary large hyperrectangle.

\begin{theorem}[Th. 3.1 in \cite{HuLe}]\label{kolmogorov}
Consider the rectangle $\mathcal{R} = [m_{x},M_{x}]\times [m_{y},M_{y}]$ for some $m_{x}< M_{x}$ and $m_{y}<M_{y}$.  
Let $Y:\mathbb{R}^2 \rightarrow \mathbb{R}$ be a continuous stochastic process. Suppose that for $\gamma,\beta >0$, $\alpha>1$ and all $(x_1,y_1),(x_2,y_2) \in \mathcal{R}$,
\begin{align*}
    \EE\left[|Y(x_2,y_2)-Y(x_2,y_1)-Y(x_1,y_2)+Y(x_1,y_1)|^\alpha\right]\leqslant K |x_2-x_1|^{1+\gamma}|y_2-y_1|^{1+\beta}.
\end{align*}
Then for every $\varepsilon_1,\varepsilon_2$ with $0<\varepsilon_1 \alpha<\gamma$ and $0<\varepsilon_2 \alpha<\beta$, there exists a random variable $\rho$ with $\EE [\rho^\alpha]\leqslant K$ and a constant $C=C(\mathcal{R})>0$ such that almost surely,
\begin{align*}
    |Y(x_2,y_2)-Y(x_2,y_1)-Y(x_1,y_2)+Y(x_1,y_1)|\leqslant C\, \rho\, |x_2-x_1|^{\gamma/\alpha -\varepsilon_1}|y_2-y_1|^{\beta/\alpha-\varepsilon_2},
\end{align*}
for all $(x_1,y_1),(x_2,y_2) \in \mathcal{R}$.
\end{theorem}

\begin{remark} \label{rem:Xiao}
Let $T>0$ and $n \in \mathbb{N}$. By Lemma 8.12 in \cite{Xiao}, for any $n\geq1$ there exists a constant $C>0$ such that for any $(x,y) \in \mathbb{R}^2$, $(s,t) \in \Delta_{[0,T]}$ and $0<\bar{\beta}<(1/(2H)-1/2)\wedge 1$, the local time $L$ of a one-dimensional fBm fulfills
\begin{align} \label{eq:xiao}
    \EE\left[|L_{s,t}(y)-L_{s,t}(x)|^n\right]\leqslant C\, |t-s|^{n(1-H(1+\bar{\beta}))}\, |x-y|^{n \bar{\beta}}.
\end{align}
\end{remark}

\begin{lemma} \label{lem:regL}
Let $0<\bar{\beta}<(1/(2H)-1/2)\wedge 1$ and $0<\gamma<1-H(1+\bar{\beta})$. Then, almost surely, ${L\in \mathcal{C}^\gamma_{[0,T]}(\mathcal{C}^{\bar{\beta}})}$. \red{Additionally for any $n \in \mathbb{N}$ there exists a random variable $\rho$ with $\EE[\rho^n]<\infty$ such that for any $M>0$, $s,t \in [0,T]$ and $x,y$ with $|x|,|y|\leqslant M$, there exists a constant $C_M>0$ such that
\begin{align} \label{eq:Lholder}
|L_{s,t}(x)-L_{s,t}(y)|\leqslant C_M \rho \,|t-s|^\gamma |x-y|^{\tilde{\beta}}.
\end{align}}
\end{lemma}
\begin{proof}
Note that by \cite[Th.~26.1]{Geman} we can assume $L$ to be jointly continuous in $(t,x)$. 
Then choosing $n$ large enough in \eqref{eq:xiao} and by Theorem~\ref{kolmogorov}, there 
exists a random variable $\rho$ with finite $n$-th moment such that, a.s., for $s,t \in [0,T]$ 
and $x,y$ with $|x|,|y|\leqslant M$, 
\begin{align*}
|L_{s,t}(x)-L_{s,t}(y)|\leqslant C_M \rho \, |t-s|^{\gamma} |x-y|^{\tilde{\beta}}.
\end{align*} 
As $L$ is a.s. compactly supported, after exhausting $[0,T]\times \mathbb{R}$ with an increasing sequence of compacts $[-M_i,M_i]$, it follows that $L \in \mathcal{C}^{\gamma}_{[0,T]}(\mathcal{C}^{\tilde{\beta}})$ a.s.
\end{proof}

Note that Lemma~\ref{lem:regL} does not give differentiability in space for the local time of a 
fBm, even in case of a small Hurst parameter. It is also possible to get time-space 
Sobolev regularity of $L$ by following the methodology of \cite{HarangPerkowski}, and in 
particular to obtain differentiability for small enough $H$. This is used in proving Theorem~\ref{prop:existence} under Assumption \ref{ex:fourbis}.

\begin{remark} \label{rem:regL}
With the same parameters as in the previous lemma, if $b \in \mathcal{B}^{\beta}_p$ for 
$\beta \in (-\bar{\beta},1-\bar{\beta})$ and $p \in [1,\infty]$, then 
Lemma~\ref{lem:regTL} and Lemma~\ref{embedding} imply that 
$T^B b \in \mathcal{C}^\gamma_{[0,T]}(\mathcal{C}^{\bar{\beta}+\beta-\varepsilon})$ almost surely, for any $\varepsilon\in (0,\bar{\beta}+\beta)$.
\end{remark}

\subsection{Existence of a weak solution} \label{sec:existenceweak}

We  first show in Proposition~\ref{lem:pbp} 
that the assumptions of Theorem~\ref{prop:existence} imply that either Assumption \ref{ass:1} or Assumption \ref{ass:2} holds. In particular, we get existence of a path-by-path solution by Theorem~\ref{generalsolution}. Then we observe, by ``randomizing'' the Euler scheme, that it is actually possible to construct a weak solution (i.e. adapted solution).
\begin{proposition} \label{lem:pbp}
Let $\beta, p, b$ be as in Theorem~\ref{prop:existence}.
\begin{enumerate}[label=(\alph*)]
\item \label{en:first2} Assume that \ref{ex:onebis} or \ref{ex:twobis} from 
Theorem~\ref{prop:existence} holds. Then, a.s., $L$ satisfies 
Assumption~\ref{assumption21} \ref{ass:2} for $\gamma, \eta \in (0,1)$ with 
$\gamma+\eta>1$, $0 < -\beta + \eta< 1$ and $\tilde{p}=\infty$.

\item \label{en:third2} Assume that \ref{ex:fourbis} from Theorem~\ref{prop:existence} holds. 
Then, a.s., $L$ satisfies Assumption~\ref{assumption21} \ref{ass:2} for $\gamma, \eta \in 
(0,1)$ with $\gamma+\eta>1$, $\eta <(\tfrac{1}{2H}+\beta-1/2) $ and $\tilde{p}=2$.
\end{enumerate}

Hence, by Theorem~\ref{generalsolution}, whenever one of the conditions 
\ref{ex:onebis}-\ref{ex:fourbis} from Theorem~\ref{prop:existence} holds, there exists a 
path-by-path solution to \eqref{eq:skewZ}.
\end{proposition}

\begin{proof}[Proof of Proposition~\ref{lem:pbp}]

It is used multiple times throughout the proof that we can also consider $b$ to be in $\mathcal{B}^{\tilde{\beta}}_p$ for any $\tilde{\beta}<\beta$ by an embedding.

\ref{en:first2}: Assume that \ref{ex:onebis} holds. W.l.o.g. assume that $\beta \in 
(1+\frac{H}{2}-\frac{1}{2H}, \frac{3}{2} - \frac{1}{2H})$. In view of Lemma~\ref{lem:regL} and 
the assumption $H\geqslant \frac{1}{3}$, $L \in 
\mathcal{C}^\gamma_{[0,T]}(\mathcal{C}^{-\beta+\eta})$ for $\eta \in \mathbb{R}$ such that 
$0<-\beta+\eta < \frac{1}{2H} - \frac{1}{2}$ and $0<\gamma <1-H(1-\beta+\eta)$. Thus 
Assumption \ref{ass:2} is fulfilled for $\tilde{p}=\infty$. Choose $\eta = \beta + 
\frac{1}{2H} - \frac{1}{2}-\vep$ and $\gamma = 1- H(\frac{1}{2}-\vep + \frac{1}{2H}) -\vep$. 
Then for small enough $\varepsilon$ one gets $\gamma \in (0,1)$, $\eta \in (0\vee \beta,1)$ 
and $\gamma+\eta>1$.

Assume that \ref{ex:twobis} holds. W.l.o.g. assume that $\beta \in (2H-1,0)$ . In view of 
Lemma \ref{lem:regL} and the assumption $H<\frac{1}{3}$, we have that $L \in 
\mathcal{C}^\gamma_{[0,T]}(\mathcal{C}^{-\beta+\eta})$ for $\eta \in \mathbb{R}$ such that 
$0<-\beta+\eta < 1$ and $0<\gamma <1-H(1-\beta+\eta)$. Thus Assumption \ref{ass:2} is 
fulfilled for $\tilde{p}=\infty$.
Choose $\eta = \beta+1-\vep$ and $\gamma = 1-H(2-\vep)-\vep$. Then again, one gets 
$\gamma \in (0,1)$, $\eta \in (0\vee \beta,1)$ and $\gamma+\eta>1$ for small enough $\vep$.

\ref{en:third2}: Assume that \ref{ex:fourbis} holds. By Theorem 3.1 in \cite{HarangPerkowski}, we know that $L \in 
\mathcal{C}^\gamma_{[0,T]}(\mathcal{B}^\lambda_{2,2})$ almost surely for $\lambda<\tfrac{1}{2H}-1/2$ 
and $0<\gamma < 1-H(\lambda+1/2)$~\footnote{Actually in \cite{HarangPerkowski} they use a 
Bessel space instead of $\mathcal{B}^\lambda_{2,2}$, but by Proposition 2.(iii) on page 47 and the Theorem on page 88 in \cite{Triebel}, this is 
equivalent.}. Hence, after a Besov space embedding (see Remark~\ref{embedding3}), Assumption \ref{ass:2} is fulfilled for $\tilde{p}=2$, $0<\eta <(\tfrac{1}{2H}-1/2+\beta) \wedge 1$ and 
$0<\gamma<(1-H(\eta-\beta+1/2))\wedge 1$. The assumption $\beta>1-\frac{1}{2H}$ ensures that we can choose $\eta$ and $\gamma$ such that $\eta+\gamma>1$.
\end{proof}

\red{
	To construct weak solutions, we will proceed with an approximation by an Euler scheme,
	 similarly to the proof of Theorem~\ref{thm:existence}. 
	 Although this time, we need the following lemma, 
	 which gives tightness and is a crucial step in showing adaptedness by an argument using 
	 Skorokhod's representation Theorem.}

\begin{lemma}\label{lem:LT2}
Assume that one of the assumptions \ref{ex:onebis}-\ref{ex:fourbis} in 
Theorem~\ref{prop:existence} holds. Then there exist $\gamma, \tilde{\eta} \in (0,1)$ with 
$\gamma + \tilde{\eta} >1$ such that for any $\lambda>0$,
\begin{equation*}
\lim_{\delta\rightarrow 0}\mathbb{P}(\Omega_{\delta,\lambda})=1, \text{ where } 
\Omega_{\delta,\lambda}:=\Big\{\omega:\sup_{|t-s|<\delta} 
[T^B 
b]_{\mathcal{C}^{1/\gamma\text{-}\text{var}}_{[s,t]}(\mathcal{C}^{\tilde{\eta}})}<\lambda\Big\}.
\end{equation*}
\end{lemma}

\begin{proof}
In view of Proposition~\ref{lem:pbp}, $L$ satisfies either Assumption~\ref{assumption21} 
\ref{ass:1} or \ref{ass:2} for some $\gamma,\eta\in (0,1)$ such that $\gamma+\eta>1$. Hence 
by Lemma~\ref{lem:regTL}, there exists $\tilde{\eta}<\eta$ such that 
$\tilde{\eta}+\gamma>1$ and  
$T^B b \in \mathcal{C}_{[0,T]}^\gamma(\mathcal{C}^{\tilde{\eta}})$. 

For $M>0$, define $L^M:=\mathbbm{1}_{\Omega_M} L$, where $\Omega_M=\{\omega:\sup_{t \in 
[0,T]} |B_t|\leqslant M\}$. 
Then,
\begin{align} \label{eq:26b}
\mathbb{P}(\Omega_{\delta,\lambda}^{\mathsf{c}})\leqslant\mathbb{P}\Big(\Big\{\omega:\delta^\gamma\sup_{|t-s|<\delta}
 [T^B b]_{\mathcal{C}^{\gamma}_{[s,t]}(\mathcal{C}^{\tilde{\eta}})}\geqslant\lambda\Big\}\cap 
\Omega_M\Big)+\mathbb{P}(\Omega_M^\mathsf{c}),
\end{align}
as $[T^B 
b]_{\mathcal{C}^{1/\gamma\text{-}\text{var}}_{[s,t]}(\mathcal{C}^{\tilde{\eta}})}\leqslant 
|t-s|^\gamma [T^B b]_{\mathcal{C}^{\gamma}_{[s,t]}(\mathcal{C}^{\tilde{\eta}})}$.

~

We will distinguish two cases, depending on whether \ref{ex:onebis} or \ref{ex:twobis} is satisfied (first case), or \ref{ex:fourbis} is satisfied (second case).

\emph{First case.} \red{Assume that \ref{ex:onebis} or \ref{ex:twobis} in 
Theorem~\ref{prop:existence} holds true. Then by Proposition~\ref{lem:pbp}\ref{en:first2}, 
$L\in \mathcal{C}_{[0,T]}^\gamma (\mathcal{C}^{-\beta+\eta})$ with \(0 < -\beta + \eta < 1\) and by Lemma~\ref{lem:regL} $\EE[[L^M]^m_{\mathcal{C}^{\gamma}_{[s,t]}(\mathcal{C}^{-\beta+\eta})}]$ is finite for any $m\geqslant 1$.
  Moreover, by Lemma~\ref{lem:regTL}, we have on $\Omega_{M}$ that $[T^B b]_{\mathcal{C}_{[0,T]}^\gamma (\mathcal{C}^{\tilde{\eta}})}\leqslant C_M [L]_{\mathcal{C}_{[0,T]}^\gamma (\mathcal{C}^{-\beta+\eta})}$, where $C_M$ depends on $M$ (but not on the realisation $\omega$).
Hence, we can bound the right hand side of \eqref{eq:26b} from above by
\begin{align}
\forall m \geqslant 1,\quad \mathbb{P}\Big(\Big\{\omega:\delta^\gamma\sup_{|t-s|<\delta} C_M[&L^M]_{\mathcal{C}^{\gamma}_{[s,t]}(\mathcal{C}^{-\beta+\eta})}\geqslant\lambda\Big\}\Big)+\mathbb{P}(\Omega_M^{\mathsf{c}}) \nonumber\\
&\leqslant C_M^m\lambda^{-m}\delta^{m\gamma}\EE\big[[L^M]^m_{\mathcal{C}^{\gamma}_{[0,T]}(\mathcal{C}^{-\beta+\eta})}\big] +\mathbb{P}(\Omega_M^{\mathsf{c}}).\nonumber\\
    &\leqslant C(m,M)\,\lambda^{-m} \delta^{m\gamma}+\mathbb{P}(\Omega_M^\mathsf{c}).\label{ineq:case1}
\end{align}}

Let $\varepsilon>0$. By Fernique's theorem, $\|B\|_{\infty}$ has exponential moments and we can choose $M$ such that $\mathbb{P}(\Omega_M^{\mathsf{c}})<\varepsilon/2$. Then $\delta$ can be chosen such that the other term in \eqref{ineq:case1} can also be controlled by $\varepsilon/2$.

~

\emph{Second case.} \red{Now assume that \ref{ex:fourbis} in 
Theorem~\ref{prop:existence} holds. By 
Proposition~\ref{lem:pbp}\ref{en:third2}, $L$ satisfies Assumption~\ref{assumption21} 
\ref{ass:2} with $\gamma<(1-H(-\beta+\eta+1/2))\wedge 1$, $\tilde{p}=2$ and $\eta <(\tfrac{1}{2H}-1/2+\beta)\wedge 1$. In view of 
Lemma~\ref{lem:regTL} and an embedding of Besov spaces (see Remark~\ref{embedding2}), we have that on $\Omega_M$, 
\begin{align}\label{ineq:TL}
[T^B b]_{\mathcal{C}_{[0,T]}^\gamma (\mathcal{C}^{\tilde{\eta}})}\leqslant C_M 
[L]_{\mathcal{C}_{[0,T]}^\gamma (\mathcal{B}_{2,2}^{-\beta+\eta})} 
\end{align}
for some constant $C_{M}>0$  that depends on $M$ but not on the realisation $\omega$. Next, we use the chain of inequalities on page 12 in \cite{HarangPerkowski} setting $\lambda^\prime$ therein equal to $-\beta+\eta+1/2+\varepsilon$. To check the condition $\lambda^\prime<1/(2H)$ appearing in \cite{HarangPerkowski}, note that $-\beta+\eta+1/2+\varepsilon<1/(2H)$
for $\varepsilon$ small enough. Hence, we get for any $m\geqslant 2$ and some $\tilde{\varepsilon}>0$,
\[\EE[\|L^M_{u,v}\|_{\mathcal{B}_{2,2}^{-\beta+\eta}}^m] 
\leqslant C\, |u-v|^{m(1-H(-\beta+\eta+1/2+\varepsilon))}\leqslant C\, |u-v|^{m(\gamma+\tilde{\varepsilon})}.
\]
An application of Kolmogorov's continuity theorem for Banach-valued stochastic processes 
(see \cite[Th.~4.3.2]{Stroock}) gives that 
$\EE[[L^M]_{\mathcal{C}^\gamma_{[0,T]}(\mathcal{B}_{2,2}^{-\beta+\eta})}^m]$ 
is finite. 
Using \eqref{ineq:TL}, we deduce as in the first case (see \eqref{ineq:case1}) that
\begin{align*}
\PP(\Omega_{\delta,\lambda}^{\mathsf{c}}) \leqslant C(m,M) \lambda^{-m} \delta^{m\gamma}+\mathbb{P}(\Omega_M^{\mathsf{c}}),
\end{align*}
for a constant \(C(m,M)\) only depending on \(m\) and \(M\). We can make the right-hand side arbitrarily small by choosing $M$ large enough and $\delta$ small enough.}
\end{proof}

\begin{proof}[Proof of Theorem~\ref{prop:existence}]

Proof of \ref{en:weak1}: 
 By Proposition~\ref{lem:pbp}, Assumption~\ref{assumption21} \ref{ass:1} or \ref{ass:2} is 
 fulfilled and thus there are some $\gamma\in(0,1)$ and $\tilde{p}\in [1,\infty]$ such that $L\in 
 \mathcal{C}^\gamma_{[0,T]}(\mathcal{B}^{\xi}_{\tilde{p}})$  a.s., for some $\xi$, with 
 its precise value given in Assumption~\ref{assumption21} \ref{ass:1} or \ref{ass:2}. Besides 
 it follows by Remark~\ref{rem:regTL} that $T^B b \in 
 \mathcal{C}^\gamma(\mathcal{C}^{\tilde{\eta}})$ for some $\tilde{\eta}\in(0,1)$ such that 
 $\gamma + \tilde{\eta}>1$.

W.l.o.g. let $T=1$. First, we use the Euler scheme as in the proof of Theorem~\ref{thm:existence} in order to construct a measurable solution. Let $X^n$ be the random counterpart of $x^n$ in the proof of Theorem~\ref{thm:existence} for $A=T^B b$, $p=1/\gamma$ and $\theta:=\gamma+\tilde{\eta}$. Let $\nu>0$. The computations done in the proof of 
Theorem~\ref{thm:existence} hold for almost every $\omega \in \Omega$ until equation 
\eqref{ineq:xn}. Choose $\lambda \in (0,1)$ such that $2\lambda C(\theta)<1$ for $C(\theta)$ 
as in \eqref{ineq:xn}. Let $\delta$ such that 
$\mathbb{P}(\Omega_{\delta,\lambda})\geqslant 1-\nu$, which is possible by Lemma~\ref{lem:LT2}. Then one can choose $N = N(\delta)$ 
large enough so that for any $n\geqslant N$ and any $\omega \in \Omega_{\delta,\lambda}$,
\begin{align*}
    \Vn{X^n}{[s,t]}&\leqslant 2C(\theta) \frac{[T^B b]_{\mathcal{C}_{[s,t]}^{1/\gamma \text{-var}}(\mathcal{C}^{\tilde{\eta}})}+\lambda}{1-2C(\theta) [T^B b]_{\mathcal{C}_{[s,t]}^{1/\gamma \text{-var}}(\mathcal{C}^{\tilde{\eta}})}}\\
    &\leqslant 2C(\theta)\frac{2\lambda}{1-2C(\theta)\lambda}, ~ \forall\, (s,t) \in \Delta_{[0,1]} \text{ with } |t-s|<\delta.
\end{align*}
It follows that we can choose $M$ sufficiently large so that for any $n\geqslant N$,
\begin{align*}
    \mathbb{P}(\|X^n\|_{\infty}>M)\leqslant \mathbb{P}(\Omega_{\delta,\lambda}\cap \{\|X^n\|_{\infty}>M\})+\mathbb{P}(\Omega^\mathsf{c}_{\delta,\lambda})<\nu.
\end{align*}
Therefore the sequence $(\text{Law}(X^n))_{n\in \N}$ is tight in the space of continuous 
functions.
Hence, along some subsequence that we do not relabel, $(X^n,B,L)$ converges in 
law in $\mathcal{C}_{[0,T]}\times \mathcal{C}_{[0,T]}\times 
\mathcal{C}^\gamma_{[0,T]}(\mathcal{B}^{\xi}_{\tilde{p}})$ to some $(X,\tilde{B},\tilde{L})$. By Skorokhod's representation Theorem, there exists a sequence $(Y^n,B^n,L^{B^n})_{n \in \mathbb{N}}$ with $\text{Law}(Y^n,B^n,L^{B^n})=\text{Law}(X^n,B,L)$ for all $n \in \mathbb{N}$, such that $(Y^n,B^n,L^{B^n})_{n \in \mathbb{N}}$ converges a.s. to some $(Y,\hat{B},\hat{L})$. To get 
that $\hat{L}$ is the local time of $\hat{B}$, observe that for any bounded measurable 
function $f$,
\begin{align} \label{eq:law2}
1=\mathbb{P}\Big(\Big\{\omega:\forall t \in [0,T]&,\,\int_0^t f(B_r) dr=\int_{\mathbb{R}} f(x) L_t(x) dx \Big\}\Big)\nonumber\\
&=\hat{\mathbb{P}}\Big(\Big\{\omega:\forall t \in [0,T],\, \int_0^t f(\hat{B}_r) dr= \int_{\mathbb{R}} f(x) \hat{L}_t(x) dx \Big\}\Big).
\end{align}
As the local time of a fBm is characterised by the occupation time formula, we deduce that $\hat{L}$ is  the local time of $\hat{B}$.
By \eqref{psi}, we have that for $T^B b(x)=\langle b, L^{B}(\cdot-x) \rangle$,
\begin{align*}
\left|X_t^n -X_{0}-\int_0^t T^{B}_{dr}b(X_r^n)\right| \limbashaut{\longrightarrow}{n \rightarrow 
\infty}{\text{a.s.}} 0.
\end{align*}
Hence for $T^{B^n} b(x)=\langle b, L^{B^n}(\cdot-x) \rangle$,
\begin{align*}
\left|Y_t^n -X_{0}-\int_0^t T^{B^n}_{dr}b(Y_r^n)\right| 
\limbashaut{\longrightarrow}{n\rightarrow \infty}{\mathbb{P}}0,
\end{align*}
and therefore a.s. along a subsequence, which we do not relabel. 
Hence, using that $\hat{L}\in \mathcal{C}^\gamma_{[0,T]}(\mathcal{B}^{\xi}_{\tilde{p}})$, we get by Remark~\ref{rem:regTL} and Lemma~\ref{intineq} that for 
$T^{\hat{B}}b(x)=\langle b,\hat{L}(\cdot-x)\rangle$,
\begin{align*}
Y_t=X_0+\int_0^t T^{\hat{B}}_{dr}b(Y_r),\ \text{for all } t \in [0,T] ~\text{ a.s.}
\end{align*}

In order for $Y$ to be a weak solution, it remains to show that $Y$ is adapted to a filtration $\hat{\mathbb{F}}$ such that $\hat{B}$ is an $\hat{\mathbb{F}}$-fBm. First note that by construction $X^n$ is adapted to $\mathbb{F}^B$. Hence, $Y^n$ is $\mathbb{F}^{B^n}$ measurable as
\begin{equation*}
\text{Law}(Y^n,B^n,L^{B^n})=\text{Law}(X^n,B,L). 
\end{equation*}
Therefore $B^n$ is an $\hat{\mathbb{F}}^n$-fBm for $\hat{\mathcal{F}}_t^n:=\sigma(Y^n_s,B^n_s,s \in [0,t])$. By definition this implies that, for $s<t$, $W^n_t-W^n_s=\mathcal{A}(B^n)_t-\mathcal{A}(B^n)_s$ is independent of $\hat{\mathcal{F}}^n_s$. After passing to the limit and using Lemma~\ref{lem:operatorcontinuity} we infer that $W_t-W_s=\mathcal{A}(\hat{B})_t-\mathcal{A}(\hat{B})_s$ is independent of $\hat{\mathcal{F}}_s:=\sigma(Y_r,\hat{B}_r, r \in [0,s])$.
Hence, $\hat{B}=\bar{\mathcal{A}}(W)$ is an $\hat{\mathbb{F}}$-fBm and therefore $(Y,\hat{B})$ is a weak solution as $Y$ is clearly $\hat{\mathbb{F}}$-adapted.

Proof of \ref{en:weak2}: \red{this proof will be presented at the end of Section~\ref{sec:RegWeak} as it follows from the stochastic sewing arguments developed in the next section.}
\end{proof}

\begin{proof}[Proof of Corollary~\ref{cor:existence}]
Any \red{finite measure} lies in $\mathcal{B}_1^0$ by similar computations as in 
Proposition 2.39 in \cite{BaDaCh}. Hence, if $b$ is a measure, we can choose $\beta=0$ in Theorem~\ref{prop:existence}. \red{For $H<1/3$ condition \ref{ex:onebis} therein is clearly fulfilled and condition \ref{ex:twobis} is fulfilled if $H^2+2H-1<0$, giving existence of a weak solution for $H<\sqrt{2}-1$.}
\end{proof}

\section{Regularity of weak solutions}\label{sec:RegWeak}

We first state Lemma~\ref{lem:Cs} which establishes various regularity estimates on the conditional expectation of fractional Brownian motion. 
It is an extension to the fBm of  
\cite[Lemma C.4]{Atetal} (which was for standard Brownian motion). It is used several times in the remainder of the paper and its proof is postponed to 
Appendix \ref{app:LND}. In particular, the proof of Lemma~\ref{lem:Cs}\ref{(C.9)} relies on a \red{variant of} local nondeterminism of the fBm, see Lemma~\ref{condvar2}. Note that Lemma~\ref{lem:Cs}\ref{(C.8)} was already stated and proven in Proposition 3.6(iii) of \cite{ButkovskyEtAl}.

\begin{lemma} \label{lem:Cs}
Let $(\Omega,\mathcal{F},\mathbb{F},\mathbb{P})$ be a filtered probability space and $B$ be an $\mathbb{F}$-fBm. 
Let $\gamma<0$ and $p \in [1,\infty]$. Let $d \in \mathbb{N}$, $(t_1,t_2) \in \Delta_{[0,T]}$ and $f:\R\times \R^d\to \R$ be a bounded measurable function and $\Xi$ be an $\mathcal{F}_{t_{1}}$-measurable $\R^d$-valued random variable. Assume that $\|f(\cdot,\Xi)\|_{\mathcal{C}^1}<\infty$ almost surely.  
Then there exists a constant $C>0$ such that
\begin{enumerate}[label=(\alph*)] 
\item $\EE^{t_1}[f(B_{t_2},\Xi)]= \int_{\R} g_{\sigma_{{t_1},{t_2}}^2}(x)\, f(\EE^{t_1}[B_{t_2}]-x,\Xi)\, dx$, also written $\EE^{t_1}[f(B_{t_2},\Xi)]=G_{\sigma_{{t_1},{t_2}}^2}f(\EE^{t_1}[B_{t_2}],\Xi)$, where $g$ is the Gaussian density and $G$ is the Gaussian semigroup introduced in \eqref{Gaussiansemigroup} and $\sigma_{{t_1},{t_2}}^2:=\var{(B_{t_2}-\EE^{t_1}[B_{t_2}])}$; 
\label{(C.8)}
\item $|\EE^{t_1}[f(B_{t_2},\Xi)]|\leqslant C \|f(\cdot,\Xi)\|_{\mathcal{B}_p^\gamma} ({t_2}-t_1)^{H(\gamma-1/p)}$;\label{(C.10)}
\item $\|f(B_{t_2},\Xi)-\EE^{t_1}[f(B_{t_2},\Xi)]\|_{L^1}\leqslant C \|\|f(\cdot, \Xi)\|_{\mathcal{C}^1}\|_{L^2} ({t_2}-{t_1})^H.$ \label{(C.11)}
\end{enumerate}
\red{Furthermore, for $n \in [1,p]$, there exists a constant $C>0$ such that for any $\tilde{t}$
in the interval $(t_1,t_2)$,}
\begin{enumerate}[resume, label=(\alph*)]
\item $\|\EE^{\tilde{t}}[f(B_{t_2},\Xi)]\|_{L^n}\leqslant C  \|\|f(\cdot,\Xi)\|_{\mathcal{B}_p^\gamma}\|_{L^n}(t_2-\tilde{t})^{H\gamma}(\tilde{t}-t_1)^{-\tfrac{1}{2p}}(t_2-t_1)^{\tfrac{1-2H}{2p}}$. \label{(C.9)}
\end{enumerate}
\end{lemma}

\begin{remark}\label{rem:betap}
In  this section we assume that $b \in \mathcal{B}^\beta_p$ for $p \in 
[1,\infty]$ and $\beta \in \mathbb{R}$ with $\beta-1/p>-1/(2H)$. Note that this condition 
allows negative values of $\beta$ for any $H<\frac{1}{2}$. In the proofs of this section, we 
 consider the case $\beta<0$ and $p\geqslant m$ for some $m \geqslant 2$. Indeed, it is 
always possible to come back to these cases in the following way:
If $\beta\geqslant 0$, $p \in [1,\infty]$ and $m\geqslant 2$, there exist $\tilde{\beta}<0$, $\tilde{p}\geqslant m$ fulfilling $\tilde{\beta}-1/{\tilde{p}}>-1/(2H)$ such that $\mathcal{B}^{\beta}_{p} \hookrightarrow \mathcal{B}^{\tilde{\beta}}_{\tilde{p}}$. This can be seen using the embeddings $\mathcal{B}^{\beta}_{p} \hookrightarrow \mathcal{B}^{\beta-(\frac{1}{p}-\frac{1}{\tilde{p}})}_{\tilde{p}}$ (see Remark \ref{embedding2}) and $\mathcal{B}^{s}_p \hookrightarrow \mathcal{B}^{\tilde{s}}_p$ for $s>\tilde{s}$.
\end{remark}

The following proposition ensures smoothness of \(X - B\) for any solution \(X\) of 
\eqref{eq:skew}.

\begin{proposition} \label{prop:regularityALT}
Let $\beta \in \mathbb{R}$, $p \in [1,\infty]$ with $0>\beta-1/p> -1/(2H)$. 
Suppose that $b \in \mathcal{B}_p^\beta$ is \red{a measure}. Then every 
weak solution $X$ to \eqref{eq:skew} fulfills $X-B \in \mathcal{C}^{1+H(\beta-1/p)}_{[0,T]}(L^m)$, for any $m\geqslant 2$.
\end{proposition}
Thus if \(X\) is a solution to \eqref{eq:skew} with $b$ \red{a finite measure}, we 
deduce that  $X-B \in 
\mathcal{C}^{1-H}_{[0,T]}(L^m)$, for any $m\geqslant 2$.
It will imply Theorem~\ref{prop:existence}\ref{en:weak2} and  is the main step in the proof of 
Theorem~\ref{thm:uniqueness}\ref{uniqueness(2)}.
\begin{proof}[Proof of Proposition~\ref{prop:regularityALT}]
Let $X$ be a weak solution to \eqref{eq:skew} and $m\geqslant 2$. 
W.l.o.g. we assume that \(X_0 = 0\). We choose a 
sequence $(b^k)_{k \in \mathbb{N}}$ of smooth nonnegative bounded functions 
converging to $b$ in $\mathcal{B}_p^{\beta-}$ with 
$\|b^k\|_{\mathcal{B}^\beta_p}\leqslant\|b\|_{\mathcal{B}^\beta_p}$ for all $k$. For $k 
\in \mathbb{N}$ and $t \in [0,T]$, let 
\begin{equation*}
    K_t^k:=\int_0^t b^k(X_r) \, dr =\int_0^t b^k(B_r+K_r) \, dr.
\end{equation*}

By Definition~\ref{def:solution} of a weak solution, we know that $K^k$ converges in 
probability to $K$ with respect to $\|\cdot\|_\infty$ on $[0,T]$. Hence it also converges 
almost surely, up to passing to a subsequence (without loss of generality, 
we do not relabel $K^k$ and assume it converges a.s.). By nonnegativity of $b^k$, $K^k$ is 
monotone for all $k$, and therefore $K$ is monotone as well.
For any \((s,t)\in\Delta_{[0,T]}\), let
\begin{align*}
   A_{s,t}^k&:=\int_s^t b^k(B_r+K_s)\, dr.
\end{align*}

We now want to apply the stochastic sewing Lemma with random controls (see Lemma~\ref{sts:randomcontrols}) for $K^k=\mathcal{A}^k$. In order to see that all conditions of this theorem are fulfilled (i.e. 
\eqref{sts2} and \eqref{sts:randomcontrols}), 
we will show the following for a constant $C$ that is independent of $k$ and for any \(u\in (s,t)\),
\begin{enumerate} [label=(\roman*)]
\item \label{enum:1ALT} $\|A_{s,t}^k\|_{L^m}\leqslant C \|b\|_{\mathcal{B}_p^\beta}(t-s)^{1+H(\beta-1/p)}$;

\item \label{enum:3ALT} $|\EE^u[\delta A_{s,u,t}^k]|\leqslant C \|b\|_{\mathcal{B}^\beta_p} |K_u-K_s|(t-u)^{H(\beta-1/p-1)+1}$;

\item \label{enum:2ALT} %
\red{Identifying $K^k$ with the limit of $\sum_{i=0}^{N_n-1} A^k_{t^n_i,t^n_{i+1}}$ along any sequence of partitions $\Pi_n=\{t_i^n\}_{i=0}^{N_n}$ of $[0,t]$ 
with mesh converging to $0$.}
\end{enumerate}
Notice that \ref{enum:1ALT} gives \eqref{sts2} for $n=m$
and that \( 1 + H (\beta - 1/p) > 1/2\).
 Furthermore, \ref{enum:3ALT} 
gives condition \eqref{sts:randomcontrols} for $\alpha_1=H(\beta-1/p-1)+1>0$, 
$\beta_1=1$ and $\lambda(s,t):=K_t-K_s$, which is a random control function by 
monotonicity of $K$.

Assume, for the moment, that \ref{enum:1ALT}-\ref{enum:3ALT}-\ref{enum:2ALT} hold true. Then by Lemma~\ref{lem:stsrandomcontrols}, there exists a process $D^k$ such that
\begin{align}\label{eq:KtnALT}
    |K_t^k-K_s^k - A^k_{s,t}|\leqslant C\|b\|_{\mathcal{B}_p^\beta} (K_t-K_s) (t-s)^{H(\beta-1/p-1)+1} + D_{s,t}^k,
\end{align}
with $\|D^k_{s,t}\|_{L^m} \leqslant C \|b\|_{\mathcal{B}_p^\beta} (t-s)^{1+H(\beta-1/p)}$.
Hence, by \ref{enum:1ALT},
\begin{align*}
\|K^k_{t}-K^k_{s}\|_{L^m} \leqslant C\|b\|_{\mathcal{B}_p^\beta} \|K_t-K_s\|_{L^m} (t-s)^{{H(\beta-1/p-1)+1}} + C \|b\|_{\mathcal{B}^\beta_p} (t-s)^{1+H(\beta-1/p)}.
\end{align*}
Hence, after letting 
$k$ go to $\infty$ in the previous inequality, we obtain that for $(s,t) \in \Delta_{[0,T]}$ such that 
$C\|b\|_{\mathcal{B}_p^\beta} (t-s)^{H(\beta-1/p-1)+1}<1/2$,
\[
    \|K_t-K_s\|_{L^m}\leqslant C\|b\|_{\mathcal{B}_p^\beta} (t-s)^{1+H(\beta-1/p)}.
\]
After covering $[0,T]$ with a finite number of small enough intervals, we obtain the result.

Let us now verify \ref{enum:1ALT}-\ref{enum:3ALT}-\ref{enum:2ALT}.
 
Proof of \ref{enum:1ALT}: By Lemma~\ref{regulINT} applied to $\Xi=K_{s}$,  $f(z,x) = 
b^k(z+x)$ and $n=m$, we have 
\begin{align*}
 \|A_{s,t}^k\|_{L^m}&\leqslant C \big\| \|b^k(\cdot+K_{s})\|_{\mathcal{B}_p^{\beta}} \big\|_{L^m} (t-s)^{1+H(\beta-1/p)} .
\end{align*}
Using Lemma~\ref{A.2}\ref{A.4}, we thus get
\begin{align} \label{regulAnALT}
    \|A_{s,t}^k\|_{L^m}&\leqslant C \|b^k\|_{\mathcal{B}_p^{\beta}} (t-s)^{1+H(\beta-1/p)} \nonumber\\
    &\leqslant C \|b\|_{\mathcal{B}_p^\beta}(t-s)^{1+H(\beta-1/p)}.
\end{align}

Proof of \ref{enum:3ALT}:
 	By 
Lemma~\ref{lem:Cs}\ref{(C.8)} applied to $\Xi = (K_{s},K_{u})$ 
	and $f(z,(x_{1},x_{2}))=b^k(z+x_{1}) - b^k(z+x_{2})$, we get
\begin{align*}
    |\EE^u[\delta A_{s,u,t}^k]|= \Big| \int_u^t \EE^u [b^k(B_r+K_s)-b^k(B_r+K_u)] dr\Big|  &\leqslant \int_u^t \|G_{\sigma^2_{u,r}}b^k\|_{\mathcal{C}^1}|K_u-K_s| dr,
   \end{align*}
where we recall $\sigma^2_{u,r}= \var(B_r-\EE^u[B_r])$. From \eqref{LND}, we have that 
$\sigma^2_{u,r}= C(u-r)^{2H}$.
Apply now Lemma~\ref{A.3}\ref{A.3.4}, which is possible as $\beta-1/p<0$, and then use that 
$\beta-1/p>-1/2H$
 to ensure integrability. This gives
\begin{align*}
    |\EE^u[\delta A_{s,u,t}^k]|
    &\leqslant C \int_u^t |r-u|^{H(\beta-1/p-1)} \|b^k\|_{\mathcal{B}^\beta_p} |K_u-K_s|dr \\
    &\leqslant C \|b\|_{\mathcal{B}^\beta_p} |K_u-K_s|(t-u)^{H(\beta-1/p-1)+1}.
\end{align*}

Proof of \ref{enum:2ALT}: For a sequence $\Pi_n=\{t_i^n\}_{i=0}^{N_n}$ of partitions of 
$[0,t]$ with mesh size going to $0$, we have
\begin{align*}
    |K_t^k - \sum_{i=0}^{N_{n}-1} A_{t_i^n,t^n_{i+1}}^k| &\leqslant \sum_i \int_{t^n_i}^{t^n_{i+1}} |b^k(B_r+K_r)-b^k(B_r+K_{t_i^n})|dr\\
    &\leqslant \sum_i \int_{t^n_i}^{t^n_{i+1}} \|b^k\|_{\mathcal{C}^1} |K_r-K_{t^n_i}| dr \\
    &\leqslant \sum_i \|b^k\|_{\mathcal{C}^1} (t^n_{i+1}-t^n_i) |K_{t^n_{i+1}}-K_{t^n_i}|\\
    &\leqslant \|b^k\|_{\mathcal{C}^1} |\Pi_n| (K_t-K_0)
    \limbashaut{\longrightarrow}{n\rightarrow \infty}{\text{a.s.}} 0.
\end{align*}
\end{proof}

\red{We conclude this section with the proof of Theorem~\ref{prop:existence}\ref{en:weak2}.}
\begin{proof}[Proof of Theorem~\ref{prop:existence}\ref{en:weak2}]
It is a direct consequence of Proposition~\ref{prop:regularityALT}.
	More precisely, in order to apply Proposition~\ref{prop:regularityALT}, we check that
	 $\beta-1/p>-1/(2H)$ holds in all cases \ref{ex:onebis}-\ref{ex:fourbis}. 
	 In case \ref{ex:onebis} holds, the assumption $\beta-1>-1/(2H)+H/2$ and
	 \(p\geqslant 1\) yields $\beta-1/p>-1/(2H)$.  In case \ref{ex:twobis} holds, then 
	 $\beta-1/p>2H-2$ and $2H-2>-1/(2H)$ for $H<1/3$. %
	Finally, in case \ref{ex:fourbis} the result 
	 follows from $p \geqslant 1$.
\end{proof}

\section{Uniqueness} \label{uniqueness}

In this section we state and prove Proposition~\ref{unique}, which gives 
	pathwise uniqueness to \eqref{eq:skew} among all weak solutions $X$ fulfilling $X-B \in 
	\mathcal{C}_{[0,T]}^{1/2+H}(L^2)$.
	 This is a crucial step towards the proof of the 
	uniqueness part in Theorem~\ref{thm:uniqueness}. 
Note that we do not assume anymore that $b$ \red{is a measure}.
The scheme of proof of Proposition~\ref{unique}, which is briefly described in 
Section~\ref{subsec:OrgaProof}, 
 is inspired by the proof of Proposition 2.1 in 
\cite{Davie} and closely follows the steps in the proof of Proposition 3.6 in \cite{Atetal}. More 
precisely, we establish that if two solutions $X$ and $Y$ are such that $X-B$ and $Y-B$ are 
$(\frac{1}{2}+H)$-H\"older continuous,
 then their difference $Z$ satisfies inequality 
\eqref{eq:z}, and since any continuous function which satisfies such an inequality must be 
$0$, the uniqueness follows. The most technical part is to establish \eqref{eq:z}. This relies on 
variations of the stochastic sewing Lemma (Sections \ref{proofunique}, \ref{proofs} and 
Appendix \ref{app:sewing}).

\red{Even when it is not explicitly written, we assume in the whole section that $H<1/2$.}

In the rest of the paper we need the following definition. Let $0 \leqslant s \leqslant t$. Let $\psi: [s,t] \times \Omega \rightarrow \mathbb{R}$ 
be a stochastic process. For $\alpha \in (0,1]$ and $m,n \in [1,\infty]$, define
\begin{align}\label{eq:defbracket}
   [\psi]_{\mathcal{C}_{[s,t]}^{\alpha}(L^{m,n})}&:= \sup_{(u,v) \in \Delta_{[s,t]}}\frac{\|\EE^u[|\psi_v-\psi_u|^m]^{\frac{1}{m}}\|_{L^n}}{(v-u)^\alpha}.
\end{align}
where the conditional expectation is taken w.r.t. the filtration the space is equipped with. 
By the tower property and Jensen's inequality for conditional expectation, we know that, for ${1\leqslant m \leqslant n \leqslant \infty}$,
\begin{align} \label{3.4}
   [\psi]_{\mathcal{C}_{[s,t]}^{\alpha}(L^m)}= [\psi]_{\mathcal{C}^{\alpha}_{[s,t]}(L^{m,m})}\leqslant  [\psi]_{\mathcal{C}_{[s,t]}^{\alpha}(L^{m,n})}\leqslant [\psi]_{\mathcal{C}_{[s,t]}^{\alpha}(L^n)}.
\end{align}

\begin{proposition} \label{unique}
Let $H<1/2$. Let $\beta \in \mathbb{R}$, $p \in [1,\infty]$ such that \eqref{eq:assumptionstrong} holds. Let $b\in \mathcal{B}^\beta_p$, and $(X_t)_{t\in [0,T]}$ and $(Y_t)_{t\in 
[0,T]}$ be two weak solutions to \eqref{eq:skew} in the sense of Definition~\ref{def:solution} 
with the same initial condition $X_0$, both being defined on the same probability 
space and adapted to 
the same filtration $\mathbb{F}$. Suppose that $(X-B)$ and $(Y-B)$ are in $\mathcal{C}_{[0,T]}^{1/2+H}(L^2)$. Further assume that
\begin{equation} \label{eq:regXB}
    [X-B]_{\COO{1/2+H}{2}{\infty}{[0,T]}}<\infty \quad a.s.
\end{equation}
Then $X$ and $Y$ are indistinguishable.
\end{proposition}

\begin{remark}\label{rk:betaneg}
In the above proposition, note that condition \eqref{eq:assumptionstrong} allows negative values of $\beta$ for any $H<\frac{1}{2}$. In the proofs in this section, we will thus only consider the case $\beta<0$ and $p\geqslant 2$. Indeed, it is always possible to come back to these cases in the following way:
If $\beta\geqslant 0$ and $p \in [1,\infty]$, there exist $\tilde{\beta}<0$, $\tilde{p}\geqslant 
2$ fulfilling \eqref{eq:assumptionstrong} such that $\mathcal{B}^{\beta}_{p} \hookrightarrow 
\mathcal{B}^{\tilde{\beta}}_{\tilde{p}}$. This can be seen using the embeddings 
$\mathcal{B}^{\beta}_{p} \hookrightarrow 
\mathcal{B}^{\beta-(\frac{1}{p}-\frac{1}{\tilde{p}})}_{\tilde{p}}$ (see Remark 
\ref{embedding2}) and $\mathcal{B}^{s}_p \hookrightarrow \mathcal{B}^{\tilde{s}}_p$ for 
$s>\tilde{s}$.
\end{remark}

\subsection{Uniqueness: Proof of Proposition~\ref{unique}}\label{proofunique}

For the proof of Proposition~\ref{unique}, which is detailed at the end of this subsection, we will use Lemma~\ref{Hst}, Lemma~\ref{Hstzst} 
and Lemma~\ref{zst}.
From now on, assume w.l.o.g. that $X_0=0$. Let $K:=X-B$ and 
$\tilde{K}:=Y-B$. Let $Z:=X-Y=K-\tilde{K}$.

\begin{lemma} \label{Hst}
Let the assumptions 
of Proposition~\ref{unique} hold. Let $(b^n)_{n \in \mathbb{N}}$ be a sequence of smooth bounded functions converging to $b$ in $\mathcal{B}^{\beta-}_p$ with $\sup_n \|b^n\|_{\mathcal{B}_p^\beta}\leqslant \|b\|_{\mathcal{B}_p^\beta}$. Then for any $(s,t) \in \Delta_{[0,T]}$, there exist random variables $\tilde{T}^K_{s,t}$ and $\tilde{T}^{\tilde{K}}_{s,t}$ such that
\begin{align} \label{eq:Hst}
   \int_s^t b^n(B_r + K_s) dr 
   \limbashaut{\longrightarrow}{n\rightarrow \infty}{L^2}
\tilde{T}_{s,t}^K ~ \text{ and } ~ \int_s^t 
   b^n(B_r + \tilde{K}_s) dr  \limbashaut{\longrightarrow}{n\rightarrow \infty}{L^2}
   \tilde{T}_{s,t}^{\tilde{K}}.
\end{align}
Moreover, there exists $C>0$ such that, for any $(s,t) \in \Delta_{[0,T]}$, we have that
\begin{align}
    \|\tilde{T}_{s,t}^K-\tilde{T}_{s,t}^{\tilde{K}}\|_{L^2} &\leqslant C\|b\|_{\mathcal{B}_p^\beta} \|Z_s\|_{L^2}(t-s)^{1/2} \ \text{and} \label{eq:Hst1}\\
     \|\tilde{T}_{s,t}^K-\tilde{T}_{s,t}^{\tilde{K}}\|_{L^2} &\leqslant C\|b\|_{\mathcal{B}_p^\beta}(t-s)^{1/2+H}. \label{eq:Hst2}
\end{align}
\end{lemma}
The proof of this lemma is moved to Section~\ref{proofs}.

~

For $(s,t) \in \Delta_{[0,T]}$, let 
\begin{align*}
    R_{s,t}:= (K_t-K_s-\tilde{T}_{s,t}^K)-(\tilde{K}_t-\tilde{K}_s-\tilde{T}_{s,t}^{\tilde{K}}).
\end{align*}
It is now necessary to estimate the regularity of this remainder term. The proof of the 
following lemma, which heavily relies on stochastic sewing, is moved to Section~\ref{proofs}.
\begin{lemma} \label{Hstzst}
Let the assumptions 
of Proposition~\ref{unique} hold. 
There exists $\delta>0$ such that for any $\alpha \in (1/2,1)$, there exists
 $C>0$ for which the following holds: for all $(u,v) \in \Delta_{[0,T]}$, we have
\begin{align}
    \|Z_{v}-Z_{u}\|_{L^2}&\leqslant C\|R\|_{\mathcal{C}_{[u,v]}^{\alpha}(L^2)} (v-u)^{1/2+\delta} \label{(5.37)}\\
    &+C\|Z\|_{\mathcal{C}_{[u,v]}(L^2)}|\log\|Z\|_{\mathcal{C}_{[u,v]}(L^2)}|\,(v-u)+C\|Z\|_{\mathcal{C}_{[u,v]}(L^2)} (v-u)^{1/2}, \nonumber\\
    \|R_{u,v}\|_{L^2}&\leqslant C\left(\|Z\|_{\mathcal{C}_{[u,v]}(L^2)}+\|R\|_{\mathcal{C}_{[u,v]}^{\alpha}(L^2)}\right) (v-u)^{1/2+\delta} \label{(5.38)}\\
    &+C\|Z\|_{\mathcal{C}_{[u,v]}(L^2)}|\log\|Z\|_{\mathcal{C}_{[u,v]}(L^2)}|\,(v-u). \nonumber
\end{align}
\end{lemma}

One can now deduce a regularity estimate on $Z$ alone.
\begin{lemma} \label{zst}
Let the assumptions 
of Proposition~\ref{unique} hold. 
Let $\tilde{T} \in (0,T]$ such that ${\|Z\|_{\mathcal{C}_{[0,\tilde{T}]}(L^2)}\leqslant 1/e}$. 
Then there exist $C>0$ and $l>0$ such that for any $(s,t) \in \Delta_{[0, \tilde{T}]}$ with 
$t-s<l$, we have
\begin{equation} \label{eq:z}
    \|Z_{t}-Z_{s}\|_{L^2}\leqslant C \|Z\|_{\mathcal{C}_{[s,t]}(L^2)}(t-s)^{1/2}+C\|Z\|_{\mathcal{C}_{[s,t]}(L^2)}|\log\|Z\|_{\mathcal{C}_{[s,t]}(L^2)}|\,(t-s).
\end{equation}
\end{lemma}

\begin{proof}
Consider  \(\delta \)  as in Lemma~\ref{Hstzst} and assume w.l.o.g. that \(\delta < 2H\).
By assumption on $X$ 
and $Y$ and by \eqref{eq:Hst2} we know that $\|R\|_{\mathcal{C}_{[0,T]}^{1/2+H}(L^2)}<\infty$. Let $(s,t) \in \Delta_{[0, \tilde{T}]}$. After dividing both 
sides of equation \eqref{(5.38)} by $(v-u)^{1/2+\delta/2}$ and taking the 
supremum over all $(u,v) \in \Delta_{[s,t]}$, we obtain for $\alpha=1/2+\delta/2$
that 
\begin{align} \label{ineq:R}
    \|R\|_{\mathcal{C}_{[s,t]}^{\alpha}(L^2)}\leqslant C&\left(\|Z\|_{\mathcal{C}_{[s,t]}(L^2)}+\|R\|_{\mathcal{C}_{[s,t]}^{\alpha}(L^2)}\right)(t-s)^{\delta/2}\\
    &+C\|Z\|_{\mathcal{C}_{[s,t]}(L^2)}|\log\|Z\|_{\mathcal{C}_{[s,t]}(L^2)}|\,(t-s)^{1/2-\delta/2}. \nonumber
\end{align}
In the above we used that
\begin{align*}
\sup_{(u,v) \in \Delta_{[s,t]}} \|Z\|_{\mathcal{C}_{[u,v]}(L^2)}|\log\|Z\|_{\mathcal{C}_{[u,v]}(L^2)}|\leqslant \|Z\|_{\mathcal{C}_{[s,t]}(L^2)}|\log\|Z\|_{\mathcal{C}_{[s,t]}(L^2)}|
\end{align*}
as $f(x)=x \log(1/x)$ is increasing on $[0,1/e]$ and $\|Z\|_{\mathcal{C}_{[s,t]}(L^2)} \leqslant 
1/e$ by assumption.
Using \(\delta < 2H\), we have that 
$\|R\|_{\mathcal{C}_{[0,T]}^{\alpha}(L^2)}\leqslant C 
\|R\|_{\mathcal{C}_{[0,T]}^{1/2+H}(L^2)}<\infty$. 
Let $l\in (0,\tilde{T})$ be such that $C l^{\delta/2}<1/2$ for $C$ as in \eqref{ineq:R}. For 
$t-s<l$, we have
\begin{align*}
     \frac{1}{2}\|R\|_{\mathcal{C}_{[s,t]}^{\alpha}(L^2)}\leqslant C\|Z\|&_{\mathcal{C}_{[s,t]}(L^2)} (t-s)^{\delta/2}\\
     &+C\|Z\|_{\mathcal{C}_{[s,t]}(L^2)}|\log\|Z\|_{\mathcal{C}_{[s,t]}(L^2)}|\,(t-s)^{1/2-\delta/2}.
\end{align*}
Plugging this into \eqref{(5.37)} for $u=s$ and $v=t$ finishes the proof.
\end{proof}

\begin{proof}[Proof of Proposition~\ref{unique}]
Let $\tilde{T}:= \sup\{t \in [0,T]: \sup_{0\leqslant s\leqslant t}\|Z_s\|_{L^2}\leqslant 1/e\}$. We 
show 
in the next paragraph that $Z$ is indistinguishable from $0$ on $[0,\tilde{T}]$. 
Note that this implies that $\tilde{T}=T$: Indeed, by definition of $\tilde{T}$ and continuity of 
$Z: [0,T] \rightarrow L^2$, $\|Z_{\tilde{T}}\|_{L^2}=1/e$ if $\tilde{T}<T$. This would 
contradict $Z$ being indistinguishable from $0$ on $[0,\tilde{T}]$. Hence, we get that $Z$ is 
actually indistinguishable from $0$ on $[0,T]$.

Assume that $\|Z_t\|_{L^2}$ is not identically $0$ on $[0,\tilde{T}]$. Choose $k_0 \in \mathbb{N}$ such 
that $2^{-{k_0}}< \sup_{t \in [0,\tilde{T}]} \|Z_t\|_{L^2}$. For $k\geqslant k_0$ let $t_k:= 
\inf\{t:\|Z_t\|_{L^2}\geqslant 2^{-k}\}$. We have that $\|Z_t\|_{L^2}<2^{-k}$ for 
$t<t_k$ and $\|Z_{t_k}\|_{L^2}=2^{-k}$ as $t\mapsto \|Z_t\|_{L^2}$ is continuous by 
\eqref{eq:z}. For $l$ as in Lemma~\ref{zst}, choose $a \in 
(0,l)$ such that $C a^{1/2}<1/4$ for $C$ as in \eqref{eq:z}. As the sequence $t_k$ is 
strictly decreasing we can choose $M$ such that $t_k-t_{k+1}<a$ for $k \geqslant M$.
Therefore by \eqref{eq:z} we have, for $k\geqslant M$, that
\begin{equation*}
2^{-(k+1)}\leqslant \|Z_{t_k}-Z_{t_{k+1}}\|_{L^2}\leqslant 2^{-(k+2)}+ C\, 2^{-k} k\log(2)\, (t_k-t_{k+1})
\end{equation*}
and therefore
\begin{equation} \label{summable}
\frac{1}{4C\log(2) k} \leqslant  t_k-t_{k+1},
\end{equation}
which leads to a contradiction as the sequence $t_k$ is strictly decreasing and the 
left hand side of \eqref{summable} is diverging when summing over $k\geqslant M$. 
Hence, $\|Z_t\|_{L^2}$ is identically $0$ on $[0,\tilde{T}]$. Now using that $Z\in \mathcal{C}^{1/2+H}_{[0,T]}(L^2)$, we get that $\EE[|Z_{t}-Z_{s}|^2] \leqslant \|Z\|_{\mathcal{C}^{1/2+H}_{[0,T]}(L^2)}^2\, |t-s|^{1+2H}$ for any $s,t\in [0,\tilde{T}]$. Hence by Kolmogorov's continuity theorem, $Z$ is a.s. continuous and is therefore indistinguishable from $0$.
\end{proof}

\subsection{Intermediate regularity results: Proofs of Lemma~\ref{Hst} and Lemma~\ref{Hstzst}} 
\label{proofs}

The proofs of Lemma~\ref{Hst} and Lemma~\ref{Hstzst} rely on several results 
about the regularity of $\int_{s}^t f(B_{r},\Xi)\, dr$ in terms of $t-s$ and of the Besov norm of 
the random mapping $f(\cdot,\Xi)$, for $\Xi$ an $\mathcal{F}_{s}$-measurable random variable.
These regularity results are stated and proven in Appendix~\ref{app:sewing}, and are derived from two main 
ingredients: the stochastic sewing Lemma 
(Lemma~\ref{sts}) and Lemma~\ref{lem:Cs}.

We now turn to the proofs of Lemmas \ref{Hst} and \ref{Hstzst}.

\begin{proof} [Proof of Lemma~\ref{Hst}]
Recall that we assumed $\beta< 0$ at the beginning of this section (Remark \ref{rk:betaneg}). Let 
$\beta^\prime \in (-\tfrac{1}{2H},\beta)$ and $(s,t) \in \Delta_{[0,T]}$. Let $k,l \in 
\mathbb{N}$. Applying the crucial regularity Lemma~\ref{regulINT} for $\Xi = K_{s}$ and $f(z, x) = b^k(z+x)-b^l(z+x)$,
we get that 
\begin{align*}
    \left\|\int_s^t b^k(B_r+K_s)dr-\int_s^t b^l(B_r+K_s)dr\right\|_{L^2} &\leqslant C\left\| \|b^k(\cdot+K_{s})-b^l(\cdot+K_{s})\|_{\mathcal{B}_p^{\beta^\prime}}\right\|_{L^2}\\
    &\leqslant C\|b^k-b^l\|_{\mathcal{B}_p^{\beta^\prime}},
\end{align*}
using Lemma~\ref{A.2}$(a)$ in the last inequality. 
Hence, $(\int_s^t b^k(B_r+K_s)dr)_{k\in \N}$ is a Cauchy sequence in $L^2$, so it converges 
to some $\tilde{T}_{s,t}^K$. The same holds for the sequence with $K$ being replaced by 
$\tilde{K}$. Using Fatou's Lemma and Corollary~\ref{cor:6.2.2}, for $\lambda \in [0,1]$ we 
obtain
\begin{align*}
    \|\tilde{T}_{s,t}^K-\tilde{T}_{s,t}^{\tilde{K}}\|_{L^2}&\leqslant \liminf_{n} \left\|\int_s^t b^n(B_r+K_s)dr-\int_s^t b^n(B_r+\tilde{K}_s)dr\right\|_{L^2}\\
    &\leqslant C \sup_n\|b^n\|_{\mathcal{B}_p^\beta} \|K_s-\tilde{K}_s\|_{L^2}^\lambda (t-s)^{1+H(\beta-\lambda-1/p)}\\
    &\leqslant C \|b\|_{\mathcal{B}_p^\beta} \|K_s-\tilde{K}_s\|_{L^2}^\lambda (t-s)^{1/2+H-\lambda H},
\end{align*}
using that $\beta-1/p\geqslant -1/(2H)+1$.
Setting $\lambda=1$ gives \eqref{eq:Hst1} and setting $\lambda=0$ gives 
\eqref{eq:Hst2}.
\end{proof}

\begin{proof} [Proof of Lemma~\ref{Hstzst}]
Let $(b^k)_{k \in \mathbb{N}}$ be a sequence of smooth bounded functions converging to $b$ in $\mathcal{B}^{\beta-}_p$ with $\sup_k \|b^k\|_{\mathcal{B}_p^\beta}\leqslant \|b\|_{\mathcal{B}_p^\beta}$. Fix $(u,v) \in \Delta_{[0,T]}$. For $(s,t) \in \Delta_{[u,v]}$, let 
\begin{align}\label{eq:defAcalk}
    A_{s,t}^k&:=\int_s^t \left(b^k(B_r+K_s)-b^k(B_r+\tilde{K}_s)\right) dr ,\nonumber \\
    \mathcal{A}_t^k&:=\int_{u}^t \left(b^k(B_r+K_r)-b^k(B_r+\tilde{K}_r)\right) dr.
\end{align}
Then, for $\theta \in (s,t)$,
\begin{align} \label{hsulsu}
    \delta A_{s,\theta,t}^k := \int_{\theta}^t \left(b^k(B_r+K_s)-b^k(B_r+\tilde{K}_s)-b^k(B_r+K_\theta)+b^k(B_r+\tilde{K}_\theta)\right) dr.
\end{align}

We now verify the conditions of the stochastic sewing Lemma with critical exponents, as 
stated in \cite[Th.~4.5]{Atetal}. To show that the conditions of this theorem hold (i.e. \eqref{sts1} and \eqref{sts2} with $\beta_{2}=0$ and $m=n=2$
 in the current paper, and (4.11) from \cite{Atetal}), we verify that there exists $C>0$ 
 independent of $s, t, u, v$ and $\theta$ such that for any $k\in \N$,
\begin{enumerate}[label=(\roman*)]

\item \label{en:2}$\|\EE^s\delta A_{s,\theta,t}^k\|_{L^2}\leqslant C (t-s)^{1+H}$;

\item \label{en:1}
	$\|\delta A_{s,\theta,t}^k\|_{L^2}
		\leqslant C\left(\|Z\|_{\mathcal{C}_{[u,v]}(L^2)}+\|R\|_{\mathcal{C}_{[u,v]}^{\alpha}(L^2)}\right)(t-s)^{1/2+\delta}$;

\item \label{en:4}
$\|\EE^s \delta A_{s,\theta,t}^k\|_{L^2}\leqslant C\|Z\|_{\mathcal{C}_{[u,v]}(L^2)}(t-s)+C \|R\|_{\mathcal{C}_{[u,v]}^{\alpha}(L^2)} (t-s)^{1/2+\alpha}$;

\item \label{en:3} \red{Identifying $\mathcal{A}^k$, as given in \eqref{eq:defAcalk}, with the limit of $\sum_{i=0}^{N_n-1} A^k_{t^n_i,t^n_{i+1}}$ along any sequence of 
partitions $\Pi_n=\{t_i^n\}_{i=0}^{N_n}$ of $[u,t]$ with mesh converging to $0$.}
\end{enumerate}

Assume for the moment that the above properties hold. Then applying Theorem 4.5 in 
\cite{Atetal}, we get that for any $(s,t) \in \Delta_{[u,v]}$,
\begin{align} \label{AtAsAst}
    \|\mathcal{A}^k_{t}-&\mathcal{A}^k_{s}-A^k_{s,t}\|_{L^2} \nonumber\\
    &\leqslant C \|Z\|_{\mathcal{C}_{[u,v]}(L^2)}\left(1+\left| \log \frac{T^H}{\|Z\|_{\mathcal{C}_{[u,v]}(L^2)}}\right|\right)|t-s|\nonumber \\
    &\quad + C\left(\|Z\|_{\mathcal{C}_{[u,v]}(L^2)}+\|R\|_{\mathcal{C}_{[u,v]}^{\alpha}(L^2)}\right)(t-s)^{1/2+\delta}+C\|R\|_{\mathcal{C}_{[u,v]}^{\alpha}(L^2)}(t-s)^{1/2+\alpha} \nonumber\\
    &\leqslant C\left(\|Z\|_{\mathcal{C}_{[u,v]}(L^2)}+\|R\|_{\mathcal{C}_{[u,v]}^{\alpha}(L^2)}\right) (t-s)^{1/2+\delta} \nonumber\\
    &\quad +C\|Z\|_{\mathcal{C}_{[u,v]}(L^2)}|\log\|Z\|_{\mathcal{C}_{[u,v]}(L^2)}|\,(t-s).
\end{align}
By Definition~\ref{def:solution} and Lemma~\ref{Hst}, we have the following convergence in probability: 
\begin{align*}
    \lim_{k\rightarrow \infty} \mathcal{A}^k_{t}-&\mathcal{A}^k_{s}-A^k_{s,t} = Z_{t}-Z_{s}-(\tilde{T}^{K}_{s,t}-\tilde{T}^{\tilde{K}}_{s,t})=R_{s,t} .
\end{align*}
Hence and letting $k$ go to $\infty$ in \eqref{AtAsAst}  and using Fatou's 
Lemma, we get \eqref{(5.38)} by choosing $(s,t)=(u,v)$. Putting together \eqref{(5.38)} and \eqref{eq:Hst1}, we 
get \eqref{(5.37)}.

Proof of \ref{en:2}: Recall that we consider times $0\leqslant u\leqslant s \leqslant \theta\leqslant t \leqslant v \leqslant T$.
Let $f_{1}:\R\times \R^4\to \R$ and $f_{2}:\R\times\R^2\to \R$ be defined by
\begin{align*}
    f_{1}(z,(x_{1},x_{2},x_{3},x_{4}))&= b^k(z+x_{1})-b^k(z+x_{2})+b^k(z+x_{3})-b^k(z+x_{4}),\\
    f_{2}(z,(x_{1},x_{2}))&=b^k(z+x_{1})-b^k(z+x_{2}),
\end{align*}
and consider the $\mathcal{F}_{\theta}$-measurable random vectors $\Xi_{1}$ and $\Xi_{2}$ given by:
\begin{align*}
\Xi_{1} &= (K_{s},\tilde{K}_{s},K_{\theta}+\tilde{K}_{s}-K_{s},K_{\theta})\\
\Xi_{2} &= (\tilde{K}_{\theta},K_{\theta}+\tilde{K}_s-K_s) .
\end{align*}
We can rewrite the integrand on the right hand side of \eqref{hsulsu} as 
$f_{1}(B_{r},\Xi_{1})+f_{2}(B_r,\Xi_{2})$. Hence using Lemma~\ref{lem:Cs}\ref{(C.10)},
 we get that 
for $\lambda \in [0,1]$,
\begin{equation} \label{hsulsu2}
	\begin{split}
		    |\EE^s \delta A^k_{s,\theta,t}|&=|\EE^s \EE^\theta \delta A^k_{s,\theta,t}|\\
    &\leqslant C\int_\theta^t \Big\{(r-\theta)^{H(\beta-1-\lambda-1/p)}\EE^s\|f_{1}(\cdot,\Xi_{1})\|_{\mathcal{B}_p^{\beta-1-\lambda}}\\
    &\qquad \qquad +(r-\theta)^{H(\beta-1-1/p)}\EE^s\|f_{2}(\cdot,\Xi_{2})\|_{\mathcal{B}_p^{\beta-1}} \Big\}dr.
	\end{split}
\end{equation}
Apply Lemma~\ref{A.2}\ref{A.6} first, then Jensen's inequality for 
conditional expectation to get that
\begin{align*}
    \EE^s\|f_{1}(\cdot,\Xi_{1})\|_{\mathcal{B}_p^{\beta-1-\lambda}}& \leqslant C \|b^k\|_{\mathcal{B}_p^\beta}\, |K_s-\tilde{K}_s|^\lambda\, \EE^s|K_\theta-K_s|\\
    &\leqslant C \|b\|_{\mathcal{B}_p^\beta}\, |K_s-\tilde{K}_s|^\lambda (\EE^s[|K_\theta-K_s|^2])^{1/2}\\
    &\leqslant C \|b\|_{\mathcal{B}_p^\beta}\, |Z_s|^\lambda\, [K]_{\COO{1/2+H}{2}{\infty}{[0,T]}}(\theta-s)^{1/2+H}.
\end{align*}
Hence,
\begin{align} \label{hsu}
    \|\EE^s\|f_{1}(\cdot,\Xi_{1})\|_{\mathcal{B}_p^{\beta-1-\lambda}}\|_{L^2}\leqslant C\|b\|_{\mathcal{B}_p^\beta} \|Z\|_{\mathcal{C}_{[u,v]}(L^2)}^\lambda [K]_{\COO{1/2+H}{2}{\infty}{[0,T]}}(t-s)^{1/2+H}.
\end{align}
By Lemma~\ref{A.2}\ref{A.5}, we have that 
$\EE^s\|f_{2}(\cdot,\Xi_{2})\|_{\mathcal{B}_p^{\beta-1}}\leqslant C \|b^k\|_{\mathcal{B}_p^\beta} 
\EE^s|Z_\theta-Z_s|$. Therefore using the contraction property of conditional 
expectation and \eqref{3.4}, we get that
\begin{align}
    \|\EE^s\|f_{2}(\cdot,\Xi_{2})\|_{\mathcal{B}_p^{\beta-1}}\|_{L^2}&\leqslant C \|b\|_{\mathcal{B}_p^\beta}\|Z_\theta-Z_s\|_{L^2}\label{eq:lsu1}\\
    &\leqslant C \|b\|_{\mathcal{B}_p^\beta}\left([K]_{\COO{1/2+H}{2}{\infty}{[0,T]}}+[\tilde{K}]_{\mathcal{C}_{[0,T]}^{1/2+H}(L^2)}\right)(t-s)^{1/2+H}\label{lsu}.
\end{align}
After putting $\lambda=0$ in \eqref{hsulsu2} and using inequalities \eqref{hsu} and 
\eqref{lsu}, we get
\begin{align*}
|\EE^s \delta &A^k_{s,\theta,t}| \\
&\leqslant C \, \|b\|_{\mathcal{B}_p^\beta} \left([K]_{\COO{1/2+H}{2}{\infty}{[0,T]}} +[\tilde{K}]_{\mathcal{C}_{[0,T]}^{1/2+H}(L^2)}\right) (t-s)^{1/2+H} \int_\theta^t (r-\theta)^{H(\beta-1-1/p)} \, dr. 
\end{align*}
Now using that $\beta-1/p\geqslant -1/(2H)+1$, we obtain that for some 
$C$ depending only on \(T\), \(\|b\|_{\mathcal{B}_p^\beta}\), 
\([K]_{\COO{1/2+H}{2}{\infty}{[0,T]}}\) and \([\tilde{K}]_{\mathcal{C}_{[0,T]}^{1/2+H}(L^2)}\),
\begin{align*}
    \|\EE^s\delta A_{s,\theta,t}^k\|_{L^2}\leqslant C (t-s)^{1+H}.
\end{align*}
Proof of \ref{en:1}: 
We use Corollary~\ref{cor:6.2.3} with $f=b^k$, $\gamma=\beta$, $\varepsilon>0$ 
small enough to ensure that $0>\beta>-1/(2H)+1+\varepsilon$, $\lambda=\lambda_1=1$, $\lambda_2=\varepsilon$, 
$\kappa_1=K_s$, $\kappa_2=\tilde{K}_s$, $\kappa_3=K_{\theta}$ and 
$\kappa_4=\tilde{K_{\theta}}$.

Hence, we get
\begin{align}\label{Asut}
\|\delta A_{s,\theta,t}^k\|_{L^2}\leqslant C \|b\|_{\mathcal{B}_p^\beta} &\|\EE^s[|K_\theta-K_s|^2]^{1/2}\|_{L^\infty}^{\varepsilon}\|K_s-\tilde{K}_s\|_{L^2}(t-\theta)^{1+H(\beta-1/p-1-\varepsilon)}\\
&+ C \|b\|_{\mathcal{B}_p^\beta}  \|Z_\theta-Z_s\|_{L^2}(t-\theta)^{1+H(\beta-1-1/p)}.\nonumber
\end{align}
We have $\|Z_\theta-Z_s\|_{L^2}\leqslant \|R_{s,\theta}\|_{L^2}+\|\tilde{T}_{s,\theta}^K-\tilde{T}_{s,\theta}^{\tilde{K}}\|_{L^2} $. Thus by Lemma~\ref{Hst},
\begin{align} \label{zuzs}
    \|Z_\theta-Z_s\|_{L^2}\leqslant \|R\|_{\mathcal{C}_{[u,v]}^{\alpha}(L^2)}(\theta-s)^\alpha + C \|b\|_{\mathcal{B}_p^\beta} \|Z_s\|_{L^2}(\theta-s)^{1/2}.
\end{align}
It follows directly from the definitions that
\begin{align*}
    \|\EE^s[|K_\theta-K_s|^2]^{1/2}\|_{L^\infty}\leqslant C (\theta-s)^{1/2+H}[K]_{\COO{1/2+H}{2}{\infty}{[0,T]}}
\end{align*}
and that $\|K_s-\tilde{K}_s\|_{L^2}\leqslant 
\|Z\|_{\mathcal{C}_{[u,v]}(L^2)}$. Plugging the last three inequalities into \eqref{Asut} we 
get that
\begin{align*}
    \|\delta A_{s,\theta,t}^k\|_{L^2}\leqslant C\|b\|_{\mathcal{B}_p^\beta} &[K]_{\COO{1/2+H}{2}{\infty}{[0,T]}}^{\varepsilon}\|Z\|_{\mathcal{C}_{[u,v]}(L^2)}(t-s)^{1-H+H(\beta-1/p)+\varepsilon /2}\\
    &+C \|b\|_{\mathcal{B}_p^\beta} \|R\|_{\mathcal{C}_{[u,v]}^{\alpha}(L^2)}(t-s)^{\alpha+1+H(\beta-1-1/p)}\\
    &\quad + C \|b\|_{\mathcal{B}_p^\beta}^2 \|Z\|_{\mathcal{C}_{[u,v]}(L^2)} (t-s)^{3/2-H+H(\beta-1/p)}.
\end{align*}
After using that $\beta-1/p\geqslant-1/(2H)+1$, we obtain that for some $C$ depending only 
on \(T\), \(\|b\|_{\mathcal{B}_p^\beta}\) and \([K]_{\COO{1/2+H}{2}{\infty}{[0,T]}}\) and some 
$\delta>0$, 
\begin{align*}
    \|\delta A_{s,\theta,t}^k\|_{L^2} \leqslant C\left(\|Z\|_{\mathcal{C}_{[u,v]}(L^2)}+\|R\|_{\mathcal{C}_{[u,v]}^{\alpha}(L^2)}\right)(t-s)^{1/2+\delta}.
\end{align*}

Proof of \ref{en:4}: We use \eqref{hsulsu2} with $\lambda=1$, \eqref{hsu}, 
\eqref{eq:lsu1} and \eqref{zuzs} to obtain that
\begin{align*}
    \|\EE^s \delta A_{s,\theta,t}^k\|_{L^2}\leqslant C\|Z\|_{\mathcal{C}_{[u,v]}(L^2)}(t-s)+ C\|R\|_{\mathcal{C}_{[u,v]}^{\alpha}L^2} (t-s)^{1/2+\alpha}.
\end{align*}

\red{That \ref{en:3} holds can be shown by similar arguments as for \ref{enum:2ALT} in the proof of Proposition~\ref{prop:regularityALT}.}
\end{proof}

\section{Existence of weak and strong solutions}\label{thmproof}

In this section we prove Theorem~\ref{th:approx} and Theorem~\ref{thm:uniqueness}. 
As in Section~\ref{uniqueness}, recall that $b$ is not assumed to be \red{a measure}.
	
	In the proofs, we assume that $p \in [m,\infty]$ and $\beta<0$ for some 
$m \geqslant 2$. As explained in Remarks~\ref{rem:betap} and \ref{rk:betaneg}, it is not a 
restriction as we can always reduce to this case under our assumptions.

The scheme of proof is the following: we use a sequence of smooth approximations of the 
drift $(b^{n})_{n\in \N}$ 
that converges in $\mathcal{B}^{\beta-}_{p}$ to $b$. We prove that the sequence 
\((X^n)_{n\in \N}\) of solutions associated to $b^{n}$
is tight 
and that each 
limit point is a weak solution to \eqref{eq:skew}. To prove the existence of a strong solution 
(which is in some sense a Yamada-Watanabe result), we rely on a classical argument of 
Krylov \cite{Krylov}, a new result on the continuity of fractional operators (Lemma 
\ref{lem:operatorcontinuity}) and the aforementionned construction of weak solutions.

On the technical level, we follow the approach of \cite{Atetal} which is based on applications of the stochastic sewing Lemma. However there are noticeable differences due to the non-Markovian character of the fBm. In particular we state immediately below:
\vspace{-0.2cm}
\begin{itemize}
\item Lemma~\ref{lem:mainregularisation}, which is a crucial estimate on the regularity of the integrals of the fBm, relying on fine properties of the 
fBm (notably Lemma \ref{lem:Cs}) and stochastic sewing, and which will be used in Section \ref{subsec:apriori} to establish \emph{a priori} estimates on the solutions;
\item Lemma \ref{lem:operatorcontinuity}, which will be used in Section \ref{mainresults} to prove that solutions are adapted.
\end{itemize}

\begin{lemma}\label{lem:mainregularisation}
Let $(\psi_t)_{t\in[0,T]}$ be a stochastic process adapted to $\mathbb{F}$. Let 
$m\in[2,\infty)$, $n\in [m,\infty]$, $p \in [n,\infty]$ and $\gamma<0$ such that 
$(\gamma-1/p)H>-1/2$. Let $\alpha \in (0,1)$ such that $H(\gamma-1/p-1)+\alpha>0$. There 
exists a constant $C>0$
 such that for any $f\in \mathcal{C}^\infty_b(\mathbb{R})\cap \mathcal{B}_p^\gamma$ and any $(s,t) \in \Delta_{[0,T]}$ we have
\begin{equation} \label{eq:lem5.2}
	\begin{split}
		\bigg\|\EE^s\bigg[\left| \int_s^t f(B_r+\psi_r) dr \right|^m\bigg]^{1/m}\bigg\|_{L^n} \leqslant& C \|f\|_{\mathcal{B}_p^\gamma}(t-s)^{1+H(\gamma-1/p)}\\
    &+C \|f\|_{\mathcal{B}_p^\gamma} [\psi]_{\COO{\alpha}{m}{n}{[s,t]}} (t-s)^{1+H(\gamma-1-1/p)+\alpha},
	\end{split}
\end{equation}
where we recall that the seminorm $[\psi]_{\COO{\alpha}{m}{n}{[s,t]}}$ is defined in \eqref{eq:defbracket}.
\end{lemma}
A similar regularisation result was proposed in \cite[Lemma 4.7]{ButkovskyEtAl} \red{for functions of positive regularity}. There, the deterministic sewing lemma is used instead of the stochastic sewing Lemma, hence the conditional expectation does not appear. \red{Note that due to an embedding argument, it would be no restriction here to work in H\"older spaces, yet we directly work in Besov spaces to stay consistent throughout the paper.}

The proof of Lemma \ref{lem:mainregularisation} is postponed to
Appendix~\ref{app:sewing}.

\subsection{\emph{A priori} estimates}\label{subsec:apriori}

Before starting the main proofs of existence of solutions, we need \emph{a priori} estimates on solutions. These estimates are established for drifts in $\mathcal{C}_b^\infty(\mathbb{R}) \cap \mathcal{B}^\beta_p$ and therefore the solution is unique and strong. Note that the following lemma looks similar to Proposition~\ref{prop:regularityALT}.
However we work here within the assumptions of Theorem \ref{th:approx}, that is with a drift that may not be \red{a measure} and $\beta-1/p>1/2-1/(2H)$ (recall that Proposition~\ref{prop:regularityALT} allows the milder condition $\beta-1/p>-1/(2H)$).

\begin{lemma} \label{regularity3.2}
Let $m \in [2,\infty)$. There exists $C > 0$, 
 such that, for any ${b\in 
\mathcal{C}_b^\infty(\mathbb{R}) \cap \mathcal{B}^\beta_p}$ \red{with 
$\beta-1/p>1/2-1/(2H)$, one has} 
\begin{equation} \label{eq:regularity3.2}
    [X-B]_{\mathcal{C}_{[0,T]}^{1+H(\beta-1/p)}(L^{m,\infty})}\leqslant 
    C\, \|b\|_{\mathcal{B}^\beta_p}\left(1+\|b\|_{\mathcal{B}^\beta_p}^{\frac{H(1/p-\beta)}{1+H(\beta-1/p)-H}}\right)\leqslant C\, \left(1+\|b\|_{\mathcal{B}^\beta_p}^2\right),
\end{equation}
where $X$ is the strong solution to  \eqref{eq:skew} with 
drift $b$, and we recall that the seminorm $[\cdot]_{\COO{\alpha}{m}{n}{[s,t]}}$ was defined in \eqref{eq:defbracket}.
\end{lemma}
Another difference with Proposition~\ref{prop:regularityALT} is that we obtain here an estimate on the seminorm $[\cdot]_{\mathcal{C}_{[0,T]}^{1+H(\beta-1/p)}(L^{m,\infty})}$, which is stronger (see \eqref{3.4}) than the seminorm $[\cdot]_{\mathcal{C}_{[0,T]}^{1+H(\beta-1/p)}(L^{m})}$ used in Proposition~\ref{prop:regularityALT}.

\begin{proof}
Note that it is sufficient to show the first inequality in \eqref{eq:regularity3.2} as the second one follows immediately from
\(1+H(\beta-1/p) - H > H(1/p -\beta) > 0\). Without loss of generality, we assume that $X_{0}=0$ and denote $K=X-B$. Then $[K]_{\mathcal{C}^{\alpha}_{[0,T]}(L^{m,\infty})}$ is finite for any $\alpha \in (0,1]$ as
\begin{align*}
    |K_t-K_s|=\left|\int_s^t b(B_r+K_r)dr\right|\leqslant \|b\|_{\infty}|t-s|.
\end{align*}
We will apply Lemma~\ref{lem:mainregularisation} with
 $n=\infty$, $\alpha=1+H(\beta-1/p)$ and considering $b \in \mathcal{B}^{\beta-1/p}_{\infty}$ 
 after an embedding.
Remark that \(\alpha > \alpha - H  > 1/2 - H/2 > 0\) and since we assume $\beta<0$ 
in the whole section, \(\alpha \in (0,1)\).
In addition,  $H(\beta-1/p)>H/2-1/2>-1/2$ and
 $H(\beta-1/p-1)+\alpha>0$.  So, the assumptions of 
 Lemma~\ref{lem:mainregularisation} are fulfilled.
Then we get for some $C,\tilde{C}>0$ that
\begin{align} \label{eq:beforeiteration}
    \|\EE^s[|K_t-K_s|&^m]^{1/m}\|_{L^\infty}\nonumber\\
    &\leqslant C \|b\|_{\mathcal{B}_\infty^{\beta-1/p}}\left((t-s)^{1+H(\beta-1/p)}+[K]_{\COO{\alpha}{m}{\infty}{[s,t]}}(t-s)^{1+H(\beta-1/p)+\alpha-H}\right) \nonumber\\
    &\leqslant \tilde{C} 
    \|b\|_{\mathcal{B}_p^\beta}(t-s)^{1+H(\beta-1/p)}\left(1+[K]_{\COO{\alpha}{m}{\infty}{[s,t]}}(t-s)^{\alpha-H}\right),
\end{align}
where we used a Besov space embedding in the second line.

Choose $l = (4\tilde{C}	\|b\|_{\mathcal{B}^\beta_p})^{1/(H-\alpha)}$ so that $\tilde{C} 
	\|b\|_{\mathcal{B}^\beta_p}l^{\alpha-H}<1/2$.
	Let $u \in [0,T]$. After dividing both sides in 
\eqref{eq:beforeiteration} by $(t-s)^{1+H(\beta-1/p)}$ and taking the supremum over 
$(s,t) \in \Delta_{[u,(u+l)\wedge T]}$ we get
\begin{align*}
[K]_{\COO{1+H(\beta-1/p)}{m}{\infty}{[u,(u+l)\wedge T]}}\leqslant \left(\tilde{C}\|b\|_{\mathcal{B}_p^\beta}+1/2 [K]_{\COO{\alpha}{m}{\infty}{[u,(u+l)\wedge T]}}\right)
\end{align*}
and therefore
\begin{align} \label{eq:beforeiteration2}
    [K]_{\COO{1+H(\beta-1/p)}{m}{\infty}{[u,(u+l) \wedge T]}}\leqslant 2 \tilde{C} \|b\|_{\mathcal{B}_p^\beta}.
\end{align}
If $l>T$, then \eqref{eq:regularity3.2} follows immediately from \eqref{eq:beforeiteration2}. Hence, assume $l\leqslant T$. In order to obtain \eqref{eq:regularity3.2} we will iteratively apply inequality 
\eqref{eq:beforeiteration2}. Let $(s,t) \in \Delta_{[0,T]}$ be arbitrary. Let 
\(N = \lceil T/l\rceil\) 
 and define the sequence 
$(s_k)_{k=0}^N$ by $s_{k} = s+ k(t-s)/N$. 
Using the triangle inequality, that $\mathcal{F}_s \subset \mathcal{F}_{s_{k-1}}$ for $k\geqslant 1$, the 
contraction property of conditional expectation and \eqref{eq:beforeiteration2} we get
\begin{align*}
    \|\EE^s[|K_t-K_s|^m]^{1/m}\|_{L^\infty}&\leqslant \sum_{k=1}^N \|\EE^s[|K_{s_k}-K_{s_{k-1}}|^m]^{1/m}\|_{L^\infty}\\
    &\leqslant \sum_{k=1}^N \|\EE^{s_{k-1}}[|K_{s_k}-K_{s_{k-1}}|^m]^{1/m}\|_{L^\infty}\\
    &\leqslant \tilde{C} \|b\|_{\mathcal{B}_p^\beta} \sum_{k=1}^N (s_k-s_{k-1})^{1+H(\beta-1/p)}\\
    &\leqslant \tilde{C} \|b\|_{\mathcal{B}_p^\beta} \sum_{k=1}^N \left(\frac{t-s}{N}\right)^{1+H(\beta-1/p)}\\
    &\leqslant \red{\tilde{C} \|b\|_{\mathcal{B}_p^\beta} N^{H(1/p-\beta)} (t-s)^{1+H(\beta-1/p)}}.
\end{align*}
Using \(N \leqslant 1+\frac{T}{l} \leqslant 2 \frac{T}{l}\leqslant C\|b\|_{\mathcal{B}^\beta_p}^{-\frac{1}{H-\alpha}}$, we get
\begin{align*}
\|b\|_{\mathcal{B}^\beta_p} N^{H(1/p-\beta)}\leqslant C\|b\|_{\mathcal{B}^\beta_p}\|b\|_{\mathcal{B}^\beta_p}^{-\frac{H(1/p-\beta)}{H-\alpha}}
\end{align*}
and therefore \eqref{eq:regularity3.2}.
\end{proof}

\begin{lemma} \label{lem:KHoelder}
Let $b,h \in \mathcal{C}_b^\infty(\mathbb{R}) \cap \mathcal{B}_p^\beta$ \red{where $\beta-1/p>1/2-1/(2H)$}. Let $X$ be 
the strong solution to \eqref{eq:skew} with drift $b$. Let $\delta \in 
(0,1+H(\beta-1/p))$. Then there exists a constant $C>0$ which is 
independent of $X_0$, $b$ and $h$, and a nonnegative random variable $Z$ which satisfies 
$\EE[Z]\leqslant 
C \|h\|_{\mathcal{B}_p^\beta}(1+\|b\|^2_{\mathcal{B}_p^\beta})$ such that
\begin{equation}\label{eq:averagingX}
    \left|\int_s^t h(X_r) dr\right|\leqslant Z\, |t-s|^\delta.
\end{equation}
\end{lemma}

\begin{proof}
Let $\Lambda^h_t:=\int_0^t h(X_r) dr$. We  apply 
Lemma~\ref{lem:mainregularisation} to establish an upper bound for 
$\| \Lambda_t^h-\Lambda_s^h\|_{L^m}$, with the 
parameters $n=m$, $\gamma=\beta (<0)$ and $\alpha= 1+H(\beta-1/p)$. 
Hence, we get
\begin{align} \label{eq:kolm}
    \|\Lambda_t^h-\Lambda_s^h\|_{L^m}\leqslant C 
    \|h\|_{\mathcal{B}_p^\beta} 
    (1+[X-B]_{\mathcal{C}_{[0,T]}^{\alpha}(L^m)})(t-s)^{1+H(\beta-1/p)}.
\end{align}
By \eqref{3.4} and Lemma~\ref{regularity3.2} we know that
\begin{align*}
    [X-B]_{\mathcal{C}_{[0,T]}^{\alpha}(L^m)}\leqslant 
    [X-B]_{\COO{\alpha}{m}{\infty}{[0,T]}}\leqslant C\, (1 +  
    \|b\|^2_{\mathcal{B}_p^\beta}).
\end{align*}
Since for any $m$ there exists a constant $C$ such that \eqref{eq:kolm} holds, the result follows from Kolmogorov's continuity criterion.
\end{proof}

\subsection{Tightness and stability}\label{tightnessandstability}

In Proposition~\ref{prop:tightness} we show that if $(b^n)_{n\in \N}$ approximates $b$, the solution $X^n$ to \eqref{eq:skew} with drift $b^n$ converges weakly  (up to taking a subsequence). Proposition~\ref{prop:stability} states that 
the limit in probability of such a sequence $(X^n)_{n\in \N}$ is a solution to the original 
equation with drift $b$.

\begin{proposition} \label{prop:tightness}
Let \((b^n)_{n \in \mathbb{N}}\) 
be a sequence of 
smooth bounded functions converging to $b$ in $\mathcal{B}_p^{\beta-}$ \red{where $b \in \mathcal{B}_p^\beta$ for $\beta-1/p>1/2-1/(2H)$}. For $n 
\in \mathbb{N}$, let $X^{n}$ be the strong solution to  
\eqref{eq:skew} with initial condition $X_0$ and drift $b^n$. Then there exists a 
subsequence $(n_k)_{k\in \mathbb{N}}$ such that $(X^{n_k},B)_{k \in 
\mathbb{N}}$ converges weakly in the space $[\mathcal{C}_{[0,T]}]^2$.
\end{proposition}

\begin{proof}
Let $K^{n}_t:=\int_0^t b^n(X^{n}_r)dr$. 
\red{For $M>0$, let  
\begin{align*}
    A_M:=\{f \in \mathcal{C}_{[0,T]}: f(0)=0,\ |f(t)-f(s)|\leqslant M (t-s)^{1+H(\beta-1/p)},\ \forall (s,t) \in \Delta_{[0,T]}\}.
\end{align*}}
By Arzel\`a-Ascoli's theorem, $A_M$ is compact in $\mathcal{C}_{[0,T]}$. 
Applying Lemma~\ref{regularity3.2} and Markov's inequality, we get
\begin{align*}
    \mathbb{P}(K^{n} \notin A_M) &\leqslant \mathbb{P}(\exists (s,t) \in 
    \Delta_{[0,T]}:|K^{n}_{s,t}|> M (t-s)^{1+H(\beta-1/p)})\\
    &\leqslant C\,
    (1+\sup_{n \in \mathbb{N}}\|b^n\|^2_{\mathcal{B}_p^\beta}) \, M^{-1}.
\end{align*}
Hence, the sequence $(K^{n})_{n \in \mathbb{N}}$  is tight in $\mathcal{C}_{[0,T]}$. So 
$(K^{n},B)_{n \in \mathbb{N}}$ is tight in $[\mathcal{C}_{[0,T]}]^2$. Thus by Prokhorov's 
Theorem, there exists a subsequence $(n_k)_{k \in \mathbb{N}}$ such that $(K^{{n_k}},B)_{k 
\in \mathbb{N}}$ converges weakly in the space $[\mathcal{C}_{[0,T]}]^2$, and so does 
$(X^{{n_k}},B)_{k \in \mathbb{N}}$.

\end{proof}
\begin{remark}
	\label{rem:twoapprox}
	The previous proposition can be generalised to any pair of sequences of approximations 
	\((X^{1,n},X^{2,n})_{n\in\mathbb{N}}\) of strong solution to \eqref{eq:skew} with two 
	sequences \((b_1^n)_{n\in\mathbb{N}}\) 
	and \((b_2^n)_{n\in\mathbb{N}}\) of smooth approximations of \(b\).
\end{remark}

\begin{proposition} \label{prop:stability}
	Let $(\tilde{b}^k)_{k \in \mathbb{N}}$ be a sequence of smooth bounded functions 
	converging to $b$ in $\mathcal{B}_p^{\beta-}$ \red{where $b \in \mathcal{B}^\beta_p$ for $\beta-1/p>1/2-1/(2H)$}. Let $\tilde{B}^k$ have the same law as 
	$B$. 
	We consider $\tilde{X}^k$ the strong solution to \eqref{eq:skew} for $B=\tilde{B}^k$, initial 
	condition 
	$X_0$ and drift $\tilde{b}^k$. 
	We assume that  there exist stochastic processes $\tilde{X},\tilde{B}: [0,T]  
	\rightarrow \mathbb{R}$ such that $(\tilde{X}^k,\tilde{B}^k)_{k \in \mathbb{N}}$ converges 
	to 
	$(\tilde{X},\tilde{B})$ on $[\mathcal{C}_{[0,T]}]^2$ in probability. Then $\tilde{X}$ fulfills 
	\eqref{solution1} and \eqref{approximation2} from Definition~\ref{def:solution} and 
	for any $m\in [2,\infty)$, there exists $C>0$ such that
	\begin{align} \label{eq:regularity}
		[\tilde{X}-\tilde{B}]_{\COO{1+H(\beta-1/p)}{m}{\infty}{[0,T]}}\leqslant 
		C \, (1 
		+ 
			\sup_{k \in \mathbb{N}}\|\tilde{b}^k\|^2_{\mathcal{B}_p^\beta}) <\infty.
	\end{align}
\end{proposition}

\begin{proof}
	We again assume here w.l.o.g. that \(X_0 = 0\) and let
		\(\tilde{K}:=\tilde{X}-\tilde{B}\), 
	so that \eqref{solution1} is automatically verified. Let now $(b^n)_{n \in 
		\mathbb{N}}$ be any sequence of smooth bounded functions converging to $b$ in 
	$\mathcal{B}_p^{\beta-}$. 
	In order to verify that $\tilde{K}$ and $\tilde{X}$ fulfill \eqref{approximation2} 
	from Definition~\ref{def:solution}, we have to show that 
	\begin{align}
		\lim_{n \rightarrow \infty} \sup_{t \in [0,T]} \left|\int_0^t b^n(\tilde{X}_r) 
		dr-\tilde{K}_t\right|=0 
		\text{ in probability}. \label{convergence}
	\end{align}	
	By the triangle inequality we have that for $k,n \in \mathbb{N}$ and $t \in [0,T]$,
	\begin{align} \label{I1I2I3}
		\left|\int_0^t b^n(\tilde{X}_r) dr-\tilde{K}_t\right|\leqslant &\left|\int_0^t b^n(\tilde{X}_r) 
		dr-\int_0^t 
		b^n(\tilde{X}_r^k)dr\right|+\left|\int_0^t b^n(\tilde{X}_r^k)dr-\int_0^t 
		\tilde{b}^k(\tilde{X}_r^k)dr\right|\nonumber\\
		&+\left|\int_0^t \tilde{b}^k(\tilde{X}_r^k)dr-\tilde{K}_t\right|=:A_1+A_2+A_3.
	\end{align}
	Now we will show that all summands on the right hand side of \eqref{I1I2I3} 
	converge to $0$ uniformly on $[0,T]$ in probability as $n\to \infty$, 
	choosing $k=k(n)$ accordingly.
	
	First we bound $A_1$. Notice that
	\begin{align*}
		\left|\int_0^t b^n(\tilde{X}_r) dr-\int_0^t b^n(\tilde{X}_r^k)dr\right|&\leqslant 
		\|b^n\|_{\mathcal{C}^1} \int_0^t |\tilde{X}_r-\tilde{X}_r^k| dr\\
		&\leqslant \|b^n\|_{\mathcal{C}^1}\,  T \sup_{t \in [0,T]} 
		 |\tilde{X}_t-\tilde{X}_t^k|.
	\end{align*}
	For any $\varepsilon>0$, we can choose an increasing sequence $(k(n))_{n \in 
	\mathbb{N}}$ such that 
	\begin{align*}
		\mathbb{P}\Big(\|b^n\|_{\mathcal{C}^1}\, T\sup_{t \in [0,T]} 
		|\tilde{X}_t-\tilde{X}_t^{k(n)}| > \varepsilon\Big)< \frac{1}{n}, ~ \forall n \in 
		\mathbb{N}.
	\end{align*}
	Hence, we get that 
	\begin{align*}
		\lim_{n \rightarrow \infty} \sup_{t\in [0,T]} \left|\int_0^t b^n(\tilde{X}_r) dr-\int_0^t 
		b^n(\tilde{X}_r^{k(n)})dr\right|=0 \text{ in probability}.
	\end{align*}
		Now, let us bound $A_2$. Let $\beta^\prime<\beta$ with 
		$\beta^\prime-1/p>-1/(2H)+1/2$. 
	By 
	Lemma~\ref{lem:KHoelder} applied to $\tilde{X}^k$, $h=b^n-\tilde{b}^k$ and $\beta^\prime$ 
	instead of 
	$\beta$, we know that there exists a random variable 
	$Z_{n,k}$ such 
	that
	\begin{align} \label{eq:expectation}
		\EE[Z_{n,k}]&\leqslant C\, 
			\|b^n-\tilde{b}^k\|_{\mathcal{B}_p^{\beta^\prime}}(1+\|\tilde{b}^k\|^2_{\mathcal{B}_p^{\beta^\prime}})\nonumber\\
		& \leqslant C\, 
			(\|b^n-b\|_{\mathcal{B}_p^{\beta^\prime}}+\|\tilde{b}^k-b\|_{\mathcal{B}_p^{\beta^\prime}})
			 \, 
			(1+\sup_{m \in \mathbb{N}}\|\tilde{b}^m\|^2_{\mathcal{B}_p^{\beta^\prime}}),
	\end{align}
	for $C$ independent of $k,n$ and that we have
	\begin{align*}
		\sup_{t \in [0,T]}\left|\int_0^t b^n(\tilde{X}_r^k)dr-\int_0^t 
		\tilde{b}^k(\tilde{X}_r^k)dr\right|\leqslant 
		Z_{n,k}(1+T).
	\end{align*}
	Using Markov's inequality and \eqref{eq:expectation} we obtain that
	\begin{align*}
		\mathbb{P}&\left(\sup_{t \in [0,T]} \left|\int_0^t b^n(\tilde{X}_r^k)dr-\int_0^t 
		\tilde{b}^k(\tilde{X}_r^k)dr\right|>\varepsilon\right)\leqslant \varepsilon^{-1}\, \EE[Z_{n,k}] 
		\, (1+T)\\
		&\qquad \qquad \leqslant C\, \varepsilon^{-1}\, (1+T)\, 
			(\|b^n-b\|_{\mathcal{B}_p^{\beta^\prime}}+\|\tilde{b}^k-b\|_{\mathcal{B}_p^{\beta^\prime}})
			 \, 
			(1+\sup_{m \in \mathbb{N}}\|\tilde{b}^m\|^2_{\mathcal{B}_p^{\beta^\prime}}).
	\end{align*}
	Choosing $k=k(n)$ as before we get
	\begin{align*}
		\lim_{n\rightarrow \infty}\sup_{t \in [0,T]}\left|\int_0^t b^n(\tilde{X}^{k(n)}_r) dr-\int_0^t 
		\tilde{b}^{k(n)}(\tilde{X}^{k(n)}_r) dr\right|=0 \text{ in probability}.
	\end{align*}
		To bound the last summand $A_3$, recall that $\tilde{X}^k_t=\int_{0}^t 
		\tilde{b}^k(\tilde{X}^k_r) 
	dr+\tilde{B}^k_t$. Hence, we get that
	\begin{align*}
		\sup_{t \in [0,T]} \left|\int_0^t \tilde{b}^k(\tilde{X}^k_r) dr-\tilde{K}_t\right| \leqslant 
		\sup_{t \in 
		[0,T]}(|\tilde{X}^k_t-\tilde{X}_t|+|\tilde{B}_t^k-\tilde{B}_t|).
	\end{align*}
	Since by assumption $(\tilde{X}^k,\tilde{B}^k)_{k \in \mathbb{N}}$ converges to 
	$(\tilde{X},\tilde{B})$ on 
	$[\mathcal{C}_{[0,T]}]^2$ in probability, we get that
	\begin{align*}
		\lim_{n \rightarrow \infty} \sup_{t \in [0,T]} \left|\int_0^t 
		\tilde{b}^{k(n)}(\tilde{X}^{k(n)}_r)dr-\tilde{K}_t\right|=0 
		\text{ in probability}
	\end{align*}
	and therefore \eqref{convergence} holds true.
	
	It remains to show that \eqref{eq:regularity} holds true. By 
Lemma~\ref{regularity3.2}, we 
	know that there exists some $C>0$ such that for any $(s,t)\in \Delta_{[0,T]}$,
	\begin{align} \label{eq:regularity2}
		\|\EE^s[|(\tilde{X}_t^k-\tilde{B}_t^k)-(\tilde{X}_s^k-\tilde{B}_s^k)|^m]^{1/m}\|_{L^\infty}\leqslant
		 C 
		\, 
			(1 + \sup_{m \in \mathbb{N}}\|\tilde{b}^m\|^2_{\mathcal{B}_p^\beta}) \, 
			(t-s)^{1+H(\beta-1/p)}.
	\end{align}
	Using that $\int_0^t \tilde{b}^k(\tilde{X}^k_r) dr$ converges to $K_t$ on 
	$\mathcal{C}_{[0,T]}$ in 
	probability (by assumption) and that $\sup_{m \in 
			\mathbb{N}}\|\tilde{b}^m\|_{\mathcal{B}_p^\beta}$ is finite, we get \eqref{eq:regularity} 
			by 
	applying Fatou's Lemma to \eqref{eq:regularity2}.
\end{proof}

\subsection{Approximation by smooth drifts} \label{mainresults}

\begin{proof}[Proof of Theorem~\ref{th:approx}]
Instead of constructing a single weak solution, we will construct a couple $(Y^1,Y^2)$. This does not change the nature of the proof, but it will be extremely useful in the proof of Theorem~\ref{thm:uniqueness} when we will seek to construct a strong solution, \emph{via} a Gy\"ongy-Krylov  argument.

Let $(b_1^n)_{n \in \mathbb{N}}$, $(b_2^n)_{n \in \mathbb{N}}$ be sequences of 
smooth bounded functions converging to $b$ in $\mathcal{B}_p^{\beta-}$. 
By Proposition~\ref{prop:tightness} and Remark~\ref{rem:twoapprox}, 
there exists a subsequence $(n_k)_{k \in \mathbb{N}}$ such that $(X^{1,n_k},X^{2,n_k}, B)_{k 
\in \mathbb{N}}$ converges weakly in $[\mathcal{C}_{[0,T]}]^3$. Without loss of generality, 
we assume that $(X^{1,n},X^{2,n}, B)_{n \in \mathbb{N}}$ converges weakly. By the 
Skorokhod representation Theorem, there exists a sequence of random variables 
$(Y^{1,n},Y^{2,n},\hat{B}^n)_{n \in \mathbb{N}}$ defined on a common probability space 
$(\hat{\Omega},\hat{\mathcal{F}},\hat{P})$, such that
\begin{align} \label{samelaw}
    \text{Law}(Y^{1,n},Y^{2,n},\hat{B}^n)=\text{Law}(X^{1,n},X^{2,n}, B), \ \forall n \in \mathbb{N},
\end{align}
and  $(Y^{1,n},Y^{2,n},\hat{B}^n)$ converges a.s. to some $(Y^1,Y^2,\hat{B})$ in $[\mathcal{C}_{[0,T]}]^3$.
As $X^{i,n}$ solves the SDE \eqref{eq:skew} with drift $b_i^n$, we know by \eqref{samelaw} that $Y^{i,n}$ also solves 
\eqref{eq:skew} with drift $b_i^n$ and $\hat{B}^n$ 
instead of $B$. As $X^{i,n}$ is a strong solution, we have that $X^{i,n}$ 
is adapted to $\mathbb{F}^B$. Hence by \eqref{samelaw}, we know that $Y^{i,n}$ is adapted to $\mathbb{F}^{\hat{B}^n}$ as the conditional laws of 
$Y^{i,n}$ and $X^{i,n}$ agree and therefore they are strong solutions to 
\eqref{eq:skew} with $\hat{B}^n$ instead of $B$.

By Proposition~\ref{prop:stability}, we know that $Y^1$ and $Y^2$ fulfill (\ref{solution1}) and 
(\ref{approximation2}) from Definition~\ref{def:solution} with $\hat{B}$ instead of $B$ and 
they are both adapted with respect to the filtration $\hat{\mathbb{F}}$ defined by 
$\hat{\mathcal{F}}_t:= \sigma(Y^1_{s},Y^2_{s},\hat{B}_{s},s \in [0,t])$. 

\red{Following the same arguments as in the proof of Theorem~\ref{prop:existence}, after passing to the limit, we know that $\hat{B}$ is an $\hat{\mathbb{F}}$-fBm and $Y^1$ and $Y^1$ are weak solutions adapted to $\hat{\mathbb{F}}$.}

Lastly, \eqref{eq:regularity} gives that 
\begin{equation}\label{eq:Khatreg}
[Y^i-\hat{B}]_{\COO{1+H(\beta-1/p)}{m}{\infty}{[0,T]}}<\infty ,\quad i=1,2,
\end{equation}
and thus $Y^1-\hat{B}$ and $Y^2-\hat{B}$ belong to $\mathcal{C}_{[0,T]}^{1+H(\beta-1/p)}(L^m)$. 
\end{proof}

\subsection{Existence of strong solutions}

In order to prove strong existence,
 we follow a Yamada-Watanabe argument: this is done here by combining 
 the construction of weak solutions from the previous proof, the uniqueness result Proposition~\ref{unique} and Gy\"ongy-Krylov's result.
 
 \begin{lemma} \label{lem:approx}
Let $H<1/2$, $p \in [1,\infty]$ and $\beta$ such that $\beta>-1/(2H)+1$ and 
$\beta-1/p\geqslant -1/(2H)+1$. Let $b \in \mathcal{B}_p^\beta$,  
  and $(b^n)_{n \in 
\mathbb{N}}$ be a sequence of smooth bounded functions converging to $b$ in 
$\mathcal{B}_p^{\beta-}$. For $n \in \mathbb{N}$, consider $X^n$ the strong solution to 
\eqref{eq:skew} with drift $b^n$ and initial condition $X_0\in \R$. Then there exists an $\mathbb{F}^B$-adapted  
process $X:[0,T]\times \Omega \rightarrow \mathbb{R}$ such that
\begin{enumerate}[label=(\roman*)]
    \item $\sup_{t \in [0,T]} |X_t^n-X_t|
     \limbashaut{\longrightarrow}{n\rightarrow \infty}{\mathbb{P}} 0$;
    \item\label{2.10(3)} $[X-B]_{\COO{1/2+H}{2}{\infty}{[0,T]}}<\infty$ a.s.
\end{enumerate}
\end{lemma}

\begin{proof} 
Let $(X^{\phi(n)})_{n \in \mathbb{N}}$ and $(X^{\psi(n)})_{n \in 
\mathbb{N}}$ be two arbitrary subsequences of $(X^n)_{n \in \mathbb{N}}$. From the 
previous proof, we know that there exists a filtered probability space 
$(\hat{\Omega},\hat{\mathcal{F}},\hat{\mathbb{F}}, \hat{P})$, an 
$\hat{\mathbb{F}}$-fBm $\hat{B}$ and a pair $(Y^1,Y^2)$ of weak solutions to 
\eqref{eq:skew} adapted to $\hat{\mathbb{F}}$. By tightness 
(Proposition~\ref{prop:tightness} and Remark~\ref{rem:twoapprox}), 
we also have that there exist subsequences $\tilde{\phi}(n)$ and $\tilde{\psi}(n)$ such that 
$(X^{\tilde{\phi}(n)},X^{\tilde{\psi}(n)})$ converges weakly to $(Y^1,Y^2)$ on 
$[\mathcal{C}_{[0,T]}]^2$ for $n \rightarrow \infty$.

Combining \eqref{eq:Khatreg}, \eqref{3.4} and that by 
assumption $1+H(\beta-1/p)\geqslant 1/2+H$, we get that
\begin{align*}
    [Y^i-\hat{B}]_{\mathcal{C}^{1/2+H}_{[0,T]}(L^2)}\leqslant[Y^i-\hat{B}]_{\mathcal{C}^{1/2+H}_{[0,T]}(L^{2,\infty})}<\infty, \quad i=1,2.
\end{align*}
By Proposition~\ref{unique}, this gives  
$Y^1=Y^2$ a.s. Hence by Lemma 1.1 in \cite{Krylov}, we get that there exists $X$ 
such that $X^n$ converges in probability to $X$ on $\mathcal{C}_{[0,T]}$. Notice that $X^n$ is adapted to $\mathbb{F}^B$, for any $n \in 
\mathbb{N}$, as they are strong solutions to \eqref{eq:skew}. So $X$ is adapted 
to the same filtration. Lastly, \ref{2.10(3)} follows from \eqref{eq:regularity}.
\end{proof}

\begin{proof}[Proof of Theorem~\ref{thm:uniqueness}]
\ref{uniqueness(1)} Let $X$ be the process constructed in Lemma \ref{lem:approx}. Proposition~\ref{prop:stability} yields that $X$ is a strong solution to \eqref{eq:skew} fulfilling \eqref{eq:regularity} for any $m\geqslant 2$. Since $\beta-1/p\geqslant 1-1/(2H)$, we get that $[X-B]_{\mathcal{C}^{\frac{1}{2}+H}_{[0,T]}(L^{m,\infty})}<\infty$ for any $m\geqslant2$.

\ref{uniqueness(1-2)} Let $(Y^1,B)$ and $(Y^2,B)$ be weak solutions defined on the same probability space, with $Y^1-B$ and $Y^2-B$ in $\mathcal{C}^{1/2+H}_{[0,T]}(L^2)$. On this probability space, let $X$ be a strong solution which satisfies \eqref{eq:regXB} with the same fBm $B$. Since $X$ is also a weak solution adapted to $\mathbb{F}$, it follows from Proposition~\ref{unique} that $X=Y^i$ for $i=1,2$.

\ref{uniquelimit} \red{By Lemma~\ref{lem:approx} the sequence of strong solutions corresponding to an approximation sequence of the drift $b$ converges in probability to a strong solution $X$ to \eqref{eq:skew} such that
\begin{align} \label{eq:regularityXB}
[X-B]_{\COO{1/2+H}{2}{\infty}{[0,T]}}<\infty.
\end{align}
The limit $X$ is the same for any approximation sequence as uniqueness in the class of solutions fulfilling \eqref{eq:regularityXB} holds by \ref{uniqueness(1-2)}.}

\ref{uniqueness(2)} Observe that if $b$ is \red{a measure}, then we know by applying 
Proposition~\ref{prop:regularityALT} and using $\beta-1/p\geqslant 1-1/(2H)$ that any solution $X$ to \eqref{eq:skew} verifies $X-B \in \mathcal{C}^{1/2+H}_{[0,T]}(L^2)$. Hence by the previous uniqueness statement, $X$ is the pathwise unique solution to \eqref{eq:skew}.
\end{proof}

\begin{appendices}

\section{Elementary results on Besov spaces}\label{app:Besov}

\begin{definition}[Partition of unity] \label{def:POF}

Let $\chi, \rho \in C^\infty(\mathbb{R})$ be even functions and for \(j\geqslant 0\), 
$\rho_j(x)=\rho(2^{-j}x)$.
We assume that 
there exists $a,b,c>0$ with 
$\supp(\chi)\subset [-c,c]$ and $\supp(\rho)\subset [-b,-a]\cup [a,b]$. Moreover, we have 
\begin{align}
\chi + \sum_{j\geqslant 0} \rho_j &\equiv 1,\\
\supp(\chi)\cap \supp(\rho_j)&=\emptyset, \ \forall j\geqslant 1,\\
\supp(\rho_j)\cap \supp(\rho_i)&=\emptyset, \text{ if } |i-j|\geqslant 2.
\end{align}
Then we call the pair $(\chi,\rho)$ a partition of unity. 
\end{definition}
Existence of a partition of unity  is proven in \cite[Prop. 2.10]{BaDaCh}. 
\red{Let such a partition be fixed.} We denote the Schwartz space on 
$\mathbb{R}$ by $\mathcal{S}$ and the space of tempered distributions by 
$\mathcal{S}^\prime$.

\begin{definition}[Littlewood-Paley blocks] \label{def:LP}
Let $f \in \mathcal{S}^\prime$. We define its $j$-th Littlewood-Paley block by
\[
    \boldsymbol{\Delta}_j f=\begin{dcases}
        \mathcal{F}^{-1}(\rho_j \mathcal{F}(f))~ \text{ for } j\geqslant 0\, , \\
        \mathcal{F}^{-1}(\chi \mathcal{F}(f))~ \text{ for } j=-1\, ,\\
        0~ \text{ for } j\leqslant -2,
    \end{dcases}
\]
where \(\mathcal{F}\) and \(\mathcal{F}^{-1}\) denote the Fourier transform and its inverse.
\end{definition}

\begin{definition}
For $s \in \mathbb{R}$ and $1\leqslant p,q \leqslant \infty$, let the nonhomogeneous Besov 
space $\mathcal{B}_{p,q}^s$ be the space of tempered distributions $u\in\mathcal{S}^\prime$ such that
\begin{align*}
    \|u\|_{\mathcal{B}_{p,q}^s}:= \left(\sum_{j \in \mathbb{Z}}\left(2^{js}\|\boldsymbol{\Delta}_j u\|_{L^p(\mathbb{R})}\right)^q\right)^{\frac{1}{q}}<\infty.
\end{align*}
\end{definition}

Lemma~\ref{A.3} and \ref{A.2}, both taken from \cite{Atetal}, are used on a regular basis 
throughout the paper. For the readers convenience we state them again below. 

\begin{lemma} \label{A.3}
Let $\gamma \in \mathbb{R}$ 
and $p \in [1,\infty]$ \red{and the Gaussian semigroup $G_t$ defined
by \eqref{Gaussiansemigroup}}. Then there exists $C>0$ such that for any $f \in 
\mathcal{B}_p^\gamma$,
\begin{enumerate} [label=(\alph*)]
    \item $\|G_t f\|_{L^p(\mathbb{R})}\leqslant C \|f\|_{\mathcal{B}_p^\gamma} t^{\gamma/2}$ for any $t>0$, provided that $\gamma<0$; \label{A.3.1}
    \item $\lim_{t \rightarrow 0} G_tf=f$ in $\mathcal{B}^{\tilde{\gamma}}_p$ for every $\tilde{\gamma}<\gamma$; \label{A.3.2}
    \item $\sup_{t>0}\|G_tf\|_{\mathcal{B}_p^\gamma}\leqslant \|f\|_{\mathcal{B}_p^\gamma}$; \label{A.3.3}
    \item $\|G_tf\|_{\mathcal{C}^1}\leqslant C\|f\|_{\mathcal{B}_p^\gamma} t^{(\gamma-1/p-1)/2}$ for all $t>0$, provided that $\gamma-1/p<0$. \label{A.3.4}
\end{enumerate}
\end{lemma}

\begin{lemma} \label{A.2}
Let $f$ be a tempered distribution on $\mathbb{R}$, $\gamma \in \mathbb{R}$, $p \in [1,\infty]$. Then for any $a,a_1,a_2,a_3 \in \mathbb{R}$, $\alpha,\alpha_1, \alpha_2 \in [0,1]$, there exists a constant $C>0$ such that
\begin{enumerate}[label=(\alph*)]
    \item $\|f(a+\cdot)\|_{\mathcal{B}_{p}^\gamma}=\|f\|_{\mathcal{B}_{p}^\gamma}$; \label{A.4}
    \item $\|f(a_1+\cdot)-f(a_2+\cdot)\|_{\mathcal{B}_{p}^\gamma}\leqslant C |a_1-a_2|^\alpha \|f\|_{\mathcal{B}_{p}^{\gamma+\alpha}}$ \label{A.5};
   
    \item $\|f(a_1+\cdot)-f(a_2+\cdot)-f(a_3+\cdot)+ f(a_3+a_2-a_1+\cdot)\|_{\mathcal{B}_p^\gamma} 
\leqslant C |a_1-a_2|^{\alpha_1} |a_1-a_3|^{\alpha_2}\|f\|_{\mathcal{B}^{\gamma+\alpha_1+\alpha_2}_p}$.\label{A.6}
\end{enumerate}
\end{lemma}

\begin{lemma} \label{lem:heatkernel}
Let $\gamma \in \mathbb{R}$, $p \in [1,\infty]$ and $f \in \mathcal{B}_p^\gamma$ and \(t > 0\), the function $G_t f$ defined
by \eqref{Gaussiansemigroup} is smooth and bounded.
\end{lemma}

\begin{proof}
Smoothness comes from \cite[Th 3.13]{Stein}. 
We now prove boundedness: First, we consider the case  $f\in 
\mathcal{B}_p^\gamma$ for $\gamma \in \mathbb{R}$ and $p \in [1,\infty]$
such that $\gamma-1/p<0$. Using Lemma 
\ref{A.3}\ref{A.3.1} and the embedding $\mathcal{B}_p^\gamma\hookrightarrow 
\mathcal{B}_\infty^{\gamma-1/p}$, we have  
\begin{align*} 
    \|G_{t}f\|_{L^\infty}\leqslant C_\gamma \|f\|_{\mathcal{B}_\infty^{\gamma- 1/p}} t^{\frac{\gamma}{2}- \frac{1}{2p}} \leqslant C_\gamma \|f\|_{\mathcal{B}_p^{\gamma}} t^{\frac{\gamma}{2}- \frac{1}{2p}}.
\end{align*}
So, $G_t f$ is bounded. 
Now for arbitrary $\gamma \in \mathbb{R}$ and $p\in [1,\infty]$ such that $f \in 
\mathcal{B}_p^\gamma$, one can choose $\gamma'<\gamma$ such that 
$\gamma'-1/p<0$. We then have $f \in \mathcal{B}_p^{\gamma'}$ since 
$\mathcal{B}^{\gamma}_{p}\subset \mathcal{B}^{\gamma^\prime}_p$. 
By the first part of the proof  
$G_t f$ is bounded. 
\end{proof}

The following lemma provides an embedding between H\"older and Besov spaces for 
functions with compact support. It is useful when we deal with local times.
\begin{lemma} \label{embedding}
Let $M \subset \mathbb{R}$ be a compact set, $0 < \beta <\hat{\beta} <  1$ and $p,q \in 
[1,\infty]$. 
There exists constants $C, \tilde{C}>0$ only depending on \(M\) such that 
for any $f \in \mathcal{C}^{\hat{\beta}}$ with support included in $M$, we have
\begin{align*}
    \|f\|_{\mathcal{B}^\beta_{p,q}}\leqslant C \|f\|_{\mathcal{B}^{\hat{\beta}}_\infty}\leqslant 
    \tilde{C} \|f\|_{\mathcal{C}^{\hat{\beta}}}.
\end{align*}
\end{lemma}

The first inequality follows by Remark~\ref{embedding3} and the second inequality by the Proposition on page 14 in \cite{RunstSickel}.

\section{Properties of fBm}

\subsection{\red{Operator linking Bm and fBm}}

\red{In Lemma~\ref{lem:operatorcontinuity} we prove a continuity property of the operator $\mathcal{A}$ introduced in \eqref{operatortildeA}. This is crucial in order to show adaptedness in the proof of \ref{prop:existence}\ref{en:weak1} and Theorem~\ref{thm:existence}. The explicit form of the operator $\mathcal{A}$ is:}
\begin{align} \label{operator}
\mathcal{A}=\widetilde{\Pi}^{H-1/2} \mathcal{I}^{1/2-H} \widetilde{\Pi}^{1/2-H},
\end{align}
where the operators $\widetilde{\Pi}^h$ and 
$\mathcal{I}^h$ are defined as follows:
\begin{align*}
\forall h \in \mathbb{R}, \quad (\widetilde{\Pi}^h f)(t)&:= t^h f(t) - h \int_0^t s^{h-1} f(s) ds ,\\
\forall h > 0, \quad  (\mathcal{I}^h f)(t)&:=\red{\frac{1}{\Gamma(h)}} \int_0^t (t-s)^{h-1} f(s) ds.
\end{align*}

\begin{lemma} \label{lem:operatorcontinuity}
\red{Let $H<1/2$.} The operator $\mathcal{A}$ defined in \eqref{operator} continuously maps 
the space $(\mathcal{C}_{[0,T]},\|\cdot\|_\infty)$ to itself.
\end{lemma}

\begin{proof}
The operator $\mathcal{A}$ is linear, therefore it is sufficient to show that it is bounded. For $f \in \mathcal{C}_{[0,T]}$ and $t \in [0,T]$, we have by \eqref{operator} that
\begin{align*}
    \widetilde{\Pi}^{H-1/2}\mathcal{I}^{1/2-H} \widetilde{\Pi}^{1/2-H} f (t)= \sum_{i=1}^4 f_i (t),
\end{align*}
where
\begin{align*}
    f_1(t)&= t^{H-1/2} \frac{1}{\Gamma(1/2-H)} \int_0^t (t-y)^{-1/2-H} y^{1/2-H} f(y) dy\\
    f_2(t)&= t^{H-1/2} (1/2-H)  \frac{-1}{\Gamma(1/2-H)}  \int_0^t (t-y)^{-1/2-H}\int_0^y x^{-1/2-H} f(x) dx dy\\
    f_3(t)&= (1/2-H) \frac{1}{\Gamma(1/2-H)} \int_0^t s^{H-3/2} \int_0^s (s-y)^{-1/2-H} y^{1/2-H} f(y) dy ds\\
    f_4(t)&= (1/2-H)^2 \frac{-1}{\Gamma(1/2-H)} \int_0^t s^{H-3/2} \int_0^s (s-y)^{-1/2-H} \int_0^y r^{-1/2-H} f(r) dr dy ds.
\end{align*}
We have, after doing a changes of variables of the form $\xi=y/t$, uniformly on $[0,T]$, that
\begin{align*}
    |f_1(t)|&\leqslant C  t^{H-1/2} \|f\|_\infty \int_0^1 t^{-1/2-H} t^{1/2-H} t (1-z)^{-1/2-H} z^{1/2-H} dz \leqslant C T^{1/2-H} \|f\|_\infty;\\
    |f_2(t)| &\leqslant C t^{H-1/2} \|f\|_\infty \int_0^t (t-y)^{-1/2-H} y^{1/2-H} dy \leqslant C T^{1/2-H} \|f\|_\infty;\\
    |f_3(t)| &\leqslant C \|f\|_\infty \int_0^t s^{H-3/2} s^{1-2H} ds \leqslant C T^{1/2-H} \|f\|_\infty;\\
    |f_4(t)| &\leqslant C \|f\|_\infty \int_0^t s^{H-3/2} \int_0^s (s-y)^{-1/2-H} y^{1/2-H} dy ds \leqslant C T^{1/2-H} \|f\|_\infty.
\end{align*}
Hence, the operator is bounded. 
It remains to prove that $\mathcal{A}(f)$ is a bounded 
continuous function. For $f_{1}$, we perform the change of variables $z=y/t$ and it comes
\begin{align*}
f_{1}(t) = t^{1/2-H}  \frac{1}{\Gamma(1/2-H)} \int_0^1 (1-z)^{-1/2-H} z^{1/2-H} f(tz) dz .
\end{align*}
By the dominated convergence theorem, this is a continuous function. We proceed similarly for $f_{2}$. As for $i=3,4$, we have 
\begin{align*}
f_i(t)=\int_0^t g_i(s) ds
\end{align*}
for functions $g_i$ that are integrable on $[0,T]$, which implies the continuity of $f_{3}$ and $f_{4}$.
\end{proof}

\subsection{Local nondeterminism of fBm and proof of Lemma~\ref{lem:Cs}}\label{app:LND}

The goal of this section is to prove Lemma~\ref{lem:Cs}. As an intermediate step, we obtain first the following lemma, which is a local nondeterminism property for the fractional Brownian motion.
\begin{lemma} \label{condvar2}
Let $(B_t)_{t\geqslant 0 }$ be an $\mathbb{F}$-fractional Brownian motion with $H\leqslant 
1/2$. There exists $C>0$ such that for any $0 \leqslant s <u<t$,
\begin{align*}
    \EE\left[(\EE^u[B_t]-\EE^s[B_t])^2\right]\geqslant C(u-s)(t-s)^{2H-1}.
\end{align*}
\end{lemma}

Note that in the above the conditional expectation is taken as usual with respect to $\mathbb{F}$ and not with respect to the filtration generated by $B$.

\begin{proof}
The case $H=1/2$ is trivial, so we assume $H<1/2$.  
Recall that the process $W=\mathcal{A}(B)$, where $\mathcal{A}$ is the operator given in 
\eqref{operatortildeA}, is an $\mathbb{F}$-Brownian motion. Moreover, by Theorem 11 in 
\cite{Picard}, it satisfies
\begin{align} \label{integralrepresentation}
    B_t=\int_0^t K_H(t,r) dW_r,
\end{align}
where, for some $d_H>0$,
\begin{equation*} 
  K_H(t,r)=d_H\left[\left({\frac{t}{r}}\right)^{H-1/2}(t-r)^{H-1/2}+(1/2-H)r^{1/2-H}\int_r^t z^{H-3/2}(z-r)^{H-1/2} dz \right].
\end{equation*}
Therefore
\begin{align}
   \EE\left[(\EE^u[B_t]-\EE^s[B_t])^2\right]&=\EE\left[\left(\EE^u\left[\int_0^t K_H(t,r) dW_r\right]-\EE^s\left[\int_0^t K_H(t,r) dW_r\right]\right)^2\right]\nonumber\\
   &= \EE\left[\left(\int_s^u K_H(t,r) dW_r\right)^2\right] \nonumber\\
   &= \int_s^u K_H(t,r)^2 dr\nonumber\\
   &\geqslant (u-s) \min_{r \in [s,u]} K_H(t,r)^2. \label{kh}
\end{align}
Notice that by the change of variables $\tilde{z}=z/r$,
\begin{align*}
    \int_r^t z^{H-3/2}(z-r)^{H-1/2} dz=r^{2H-1}\beta_H(t/r),
\end{align*}
for $\beta_H(\tau)=\int_1^\tau \tilde{z}^{H-3/2}(\tilde{z}-1)^{H-1/2} d\tilde{z}$. Hence,
\begin{equation*}
     K_H(t,r)=d_H\left[\left({\frac{t}{r}}\right)^{H-1/2}(t-r)^{H-1/2}+(1/2-H)r^{H-1/2}\beta_H(t/r) \right].
\end{equation*}
Using that
\begin{align*}
    \beta_H\left(\tfrac{t}{r}\right)&=\int_1^{t/r} z^{H-3/2} (z-1)^{H-1/2} dz\\
    &\geqslant \int_1^{t/r} z^{2H-2} dz\\
    &=\frac{1}{1-2H}\left(1-\left(\frac{t}{r}\right)^{2H-1}\right)
\end{align*}
and subadditivity of $x \mapsto x^\alpha$ for $\alpha \in (0,1)$ we get, for $0<r<t$, that
\begin{align*}
 \frac{1}{d_H}K_H(t,r)&\geqslant \left(\frac{t}{r}\right)^{H-1/2}(t-r)^{H-1/2}+\frac{1/2-H}{1-2H}r^{H-1/2}\left(1-\left(\frac{t}{r}\right)^{2H-1}\right)\\
 &\geqslant (t-r)^{H-1/2}\left(\left(\frac{r}{t}\right)^{1/2-H}+\frac{1}{2}(t-r)^{1/2-H}r^{H-1/2}\left(1-\left(\frac{r}{t}\right)^{1-2H}\right)\right)\\
 &\geqslant (t-r)^{H-1/2} \left(\left(\frac{r}{t}\right)^{1/2-H}+\frac{1}{2}(t^{1/2-H}-r^{1/2-H})r^{H-1/2}\left(1-\left(\frac{r}{t}\right)^{1-2H}\right)\right)\\
 &= \frac{1}{2}(t-r)^{H-1/2}\left(\left(\frac{r}{t}\right)^{1/2-H} + \left(\frac{t}{r}\right)^{1/2-H} + \left(\frac{r}{t}\right)^{1-2H}-1\right)\\
 &\geqslant \frac{1}{2}(t-r)^{H-1/2},
\end{align*}
where the last line holds true as $x+x^{-1}\geqslant 2$ for $x>0$.
Plugging this into \eqref{kh} we get that
\begin{align*}
     \EE\left[(\EE^u[B_t]-\EE^s[B_t])^2\right] &\geqslant (u-s) \min_{r \in [s,u]} K_H(t,r)^2\\
     &\geqslant \frac{d_H^2}{4}(u-s)(t-s)^{2H-1}.
\end{align*}
\end{proof}

\begin{proof}[Proof of Lemma~\ref{lem:Cs}]
Proof of \ref{(C.8)}: Notice that $B_{t_2}-\EE^{t_1}[B_{t_2}]$ is Gaussian with zero mean and 
variance $\sigma_{{t_1},{t_2}}^2$. Furthermore, it is independent of $\mathcal{F}_{t_1}$, which 
can be seen using the integral representation (see \eqref{integralrepresentation}).
Hence, we have
\begin{align*}
	\EE^{t_1}[f(B_{t_2},\Xi)]&=\EE^{t_1}[f(B_{t_2}-\EE^{t_1}[B_{t_2}]+\EE^{t_1}[B_{t_2}],\Xi)]\\
	&=G_{\sigma_{{t_1},{t_2}}^2}f(\EE^{t_1}[B_{t_2}],\Xi).
\end{align*}
Proof of \ref{(C.10)}: By the local nondeterminism property of fBm (see 
Lemma 7.1 in \cite{Pitt}), there exists $C>0$ such that
\begin{align} \label{LND}
\sigma_{t_1,t_2}^2=C(t_2-t_1)^{2H}.
\end{align}
To see that \ref{(C.10)} holds true we use \ref{(C.8)}, 
Lemma~\ref{A.3}\ref{A.3.1}, the embedding $\mathcal{B}_p^\gamma \hookrightarrow 
\mathcal{B}_\infty^{\gamma-1/p}$ and \eqref{LND} to get that
\begin{align*}
    |\EE^{t_1}[f(B_{t_2},\Xi)]|\leqslant C \|f(\cdot,\Xi)\|_{\mathcal{B}_\infty^{\gamma-1/p}}\sigma_{{t_1},{t_2}}^{\gamma-1/p}\leqslant C \|f(\cdot,\Xi)\|_{\mathcal{B}_p^\gamma}({t_2}-{t_1})^{H(\gamma-1/p)}.
\end{align*}

Proof of \ref{(C.11)}: First notice that for fixed $x \in \mathbb{R}$, $t>0$ and $\phi:\mathbb{R} \rightarrow \mathbb{R}$ with $\|\phi\|_{\mathcal{C}^1}<\infty$,
\begin{align} \label{heatkernelbound}
    |\phi(x)-G_t \phi(x)|\leqslant C \|\phi\|_{\mathcal{C}^1} \sqrt{t}.
\end{align}
To see that \ref{(C.11)} holds true note that due to \ref{(C.8)}, 
\eqref{heatkernelbound} and Cauchy-Schwarz's inequality we get
\begin{align*}
    \EE[|f(B_{t_2},\Xi)&-\EE^{t_1}[f(B_{t_2},\Xi)]|]=\EE[|f(B_{t_2},\Xi)-G_{\sigma^2_{{t_1},{t_2}}}f(\EE^{t_1}[B_{t_2}],\Xi)|]\\
    &\leqslant \EE[|f(B_{t_2},\Xi)-f(\EE^{t_1}[B_{t_2}],\Xi)|]+\EE[|f(\EE^{t_1}[B_{t_2}],\Xi)-G_{\sigma^2_{{t_1},{t_2}}}f(\EE^{t_1}[B_{t_2}],\Xi)|]\\
    &\leqslant \EE[\|f(\cdot,\Xi)\|_{\mathcal{C}^1}|B_{t_2}-\EE^{t_1}[B_{t_2}]|]+C\, \EE[\|f(\cdot,\Xi)\|_{\mathcal{C}^1}]  \, \sigma_{{t_1},{t_2}}\\
    &\leqslant C\, \|\|f(\cdot,\Xi)\|_{\mathcal{C}^1}\|_{L^2} ({t_2}-{t_1})^H,
\end{align*}
where the last line holds true by \eqref{LND} and as
\begin{align*}
\|(B_{t_2}-\EE^{t_1}[B_{t_2}])\|_{L^2}\leqslant \|B_{t_2}-B_{t_1}\|_{L^2}+\|B_{t_1}-\EE^{t_1}[B_{t_2}]\|_{L^2}\leqslant 2({t_2}-{t_1})^H
\end{align*}
by Jensen's inequality for conditional expectations.

Proof of \ref{(C.9)}: In order to prove \ref{(C.9)} we first state the following inequality, for 
some $C>0$: for any $t_1 < \tilde{t} < t_2$,
 \begin{align} 
    \hat{\sigma}_{{t_1},\tilde{t},{t_2}}^2:=\var(\EE^{\tilde{t}}[B_{t_2}]-\EE^{t_1}[B_{t_2}])&\geqslant C ({\tilde{t}}-{t_1})({t_2}-{t_1})^{-1+2H}. 
    \label{LND2}
\end{align}
The above inequality holds true by Lemma~\ref{condvar2}.
Notice that $\EE^{\tilde{t}}[B_{t_2}]-\EE^{t_1}[B_{t_2}]$ is Gaussian with mean zero and variance $\hat{\sigma}_{{t_1},\tilde{t},{t_2}}^2$ and it is independent of $\mathcal{F}_{t_1}$. Using this, \ref{(C.8)} and  H\"older's inequality for $q=p/n\geqslant 1$ and $q^\prime=q/(q-1)$, we get 
\begin{align*}
    \EE^{t_1}\left[|\EE^{\tilde{t}}[f(B_{t_2},\Xi)]|^n\right]&=\EE^{t_1}[|G_{{\sigma}_{{\tilde{t}},{t_2}}^2}f(\EE^{\tilde{t}}[B_{t_2}],\Xi)|^n]\\
    &=\int g_{\hat{\sigma}_{{t_1},{\tilde{t}},{t_2}}^2}(z)|G_{\sigma_{{\tilde{t}},{t_2}}^2}f(\EE^{t_1}[B_{t_2}]+z,\Xi)|^n dz\\
    &\leqslant \|g_{\hat{\sigma}_{{t_1},{\tilde{t}},{t_2}}^2}\|_{L^{q^\prime}(\mathbb{R})}\Big\|\left(G_{\sigma_{{\tilde{t}},{t_2}}^2}f(\EE^{t_1}[B_{t_2}]+\cdot,\Xi)\right)^n\Big\|_{L^q(\mathbb{R})}\\
    &= \|G_{\hat{\sigma}_{{t_1},{\tilde{t}},{t_2}}^2} \delta_0\|_{L^{q^\prime}(\mathbb{R})}\|G_{\sigma_{{\tilde{t}},{t_2}}^2}f(\cdot,\Xi)\|_{L^p(\mathbb{R})}^n.
\end{align*}
Using Lemma~\ref{A.3}\ref{A.3.1}, \eqref{LND}, \eqref{LND2} and that 
$\|\delta_0\|_{\mathcal{B}_x^{-1+1/x}}<\infty$ for $x \geqslant 1$, we get 
\begin{align*}
    \|G_{\hat{\sigma}^2_{{t_1},{\tilde{t}},{t_2}}} \delta_0\|_{L^{q^\prime}(\mathbb{R})}&\|G_{\sigma^2_{{\tilde{t}},{t_2}}}f(\cdot,\Xi)\|_{L^p(\mathbb{R})}^n \leqslant C \|f(\cdot,\Xi)\|_{\mathcal{B}^\gamma_p}^n\|\delta_0\|_{\mathcal{B}_{q^\prime}^{-n/p}}\, \sigma_{{\tilde{t}},{t_2}}^{\gamma n}\, \hat{\sigma}_{{t_1},{\tilde{t}},{t_2}}^{-n/p}\\
    &\leqslant C \|f(\cdot,\Xi)\|_{\mathcal{B}^\gamma_p}^n\|\delta_0\|_{\mathcal{B}_{q^\prime}^{-n/p}} ({t_2}-{\tilde{t}})^{H\gamma n} ({\tilde{t}}-{t_1})^{-n/(2p)} ({t_2}-{t_1})^{n(1-2H)/(2p)}.
\end{align*}
Hence,
\begin{align} \label{ineq:C.9}
     \EE^{t_1}\Big[|\EE^{\tilde{t}}[f(B_{t_2},&\Xi)]|^n\Big]\nonumber\\
     &\leqslant C \|f(\cdot,\Xi)\|_{\mathcal{B}^\gamma_p}^n\|\delta_0\|_{\mathcal{B}_{q^\prime}^{-n/p}} ({t_2}-\tilde{t})^{H\gamma n} ({\tilde{t}}-t_1)^{-n/(2p)} ({t_2}-t_1)^{n(1-2H)/(2p)}.
\end{align}
After taking expectations in \eqref{ineq:C.9} and raising both sides to the power 
$1/n$, we obtain \ref{(C.9)}.
\end{proof}

\section{Stochastic sewing Lemma and regularising properties of the fBm}\label{app:sewing}

In this section we will state and prove some results that are crucial throughout the paper. The statements and their proofs are close extensions to fBm of Lemma 6.1, its corollaries and Lemma 5.2 in \cite{Atetal}.

\red{First, in the following two lemmas we recall two recent extensions of the stochastic 
sewing Lemma (see \cite[Th.~4.1 and Th.~4.7]{Le}). In both statements, let $0<S<T$, let 
$(\Omega,\mathcal{F},\mathbb{F},\mathbb{P})$ be a filtered probability space and let $A: 
\Delta_{[S,T]} \rightarrow L^m$ such that $A_{s,t}$ is $\mathcal{F}_t$-measurable for $(s,t) 
\in \Delta_{[S,T]}$.}
\begin{lemma}[Stochastic sewing Lemma] \label{sts}
Let $n \in [m,\infty]$. Assume that there exist constants $\Gamma_1,\Gamma_2\geqslant 0,\, \varepsilon_1,\varepsilon_2>0, \,\alpha_2 \in [0,1/2)$ such that for every $(s,t) \in \Delta_{[S,T]}$ and $u:=(s+t)/2$,
\begin{align}
    \|\EE^s[\delta A_{s,u,t}]\|_{L^n}&\leqslant \Gamma_1 (t-s)^{1+\varepsilon_1},\label{sts1}\\
    \|\EE^S[|\delta A_{s,u,t}|^m]^{1/m}\|_{L^n} &\leqslant \Gamma_2 (u-S)^{-\alpha_2}(t-s)^{1/2+\varepsilon_2}. \label{sts2}
\end{align}
\red{Suppose} there exists a process $(\mathcal{A}_t)_{t\in [S,T]}$ such that, for any $t \in [S,T]$ 
and any sequence of partitions $\Pi_k=\{t_i^k\}_{i=0}^{N_k}$ of $[S,t]$ with mesh size going 
to zero, we have
\begin{align} \label{sts3}
    \mathcal{A}_t=\lim_{k\rightarrow \infty}\sum_{i=0}^{N_k}A_{t_i^k,t_{i+1}^k} \text{ in probability.} 
\end{align}

Moreover, there exists a constant $C=C(\varepsilon_1,\varepsilon_2,\red{\alpha_2},m)$ independent of $S,T$ such that for every $(s,t) \in \Delta_{[S,T]}$ we have
\begin{align*}
    \|\EE^S[|\mathcal{A}_t-\mathcal{A}_s-A_{s,t}|^m]^{1/m}\|_{L^n}\leqslant C \Gamma_2 \red{(t-s)^{1/2-\alpha_2+\varepsilon_2}}+C\Gamma_1\red{(t-s)^{1+\varepsilon_1}}
\end{align*}
and
\begin{align*}
    \|\EE^S[\mathcal{A}_t-\mathcal{A}_s-A_{s,t}]\|_{L^n}\leqslant C \Gamma_1 \red{(t-s)^{1+\varepsilon_1}}.
\end{align*}
\end{lemma}

\begin{lemma} [Stochastic sewing lemma with random controls] \label{lem:stsrandomcontrols}
\red{Let $\lambda$ be a random control. Assume that there exist constants $\Gamma_1,\alpha_1,\beta_1\geqslant 0$ with $\alpha_1+\beta_1>1$ such that
\begin{align}\label{sts:randomcontrols}
|\EE^u \delta A_{s,u,t}| \leqslant \Gamma_1 |t-s|^{\alpha_1} \lambda(s,t)^{\beta_1} \text{ a.s. }
\end{align}
for all $(s,t) \in \Delta_{[S,T]}$ and $u\coloneqq (s+t)/2$. Additionally, assume that \eqref{sts2} holds for $m=n$ and $\alpha_2=0$ \red{and that there exists a process $(\mathcal{A}_t)_{t \in [S,T]}$ such that \eqref{sts3} holds}. Then there is a map $B\colon \Delta_{[S,T]} \rightarrow L^m$ and a constant $C>0$ such that for all $(s,t) \in \Delta_{[S,T]}$,
\begin{align}
|\mathcal{A}_t-\mathcal{A}_s-A_{s,t}|&\leqslant C \Gamma_1 |t-s|^{\alpha_1} \lambda(s,t)^{\beta_1} + B_{s,t} \text{ a.s. and}\\
\|B_{s,t}\|_{L^m} &\leqslant C \Gamma_2 |t-s|^{1/2+\varepsilon_2}.
\end{align}}
\end{lemma}

\begin{lemma} \label{regulINT}
Let  $\gamma \in (-1/(2H),0)$, $m \in [2,\infty)$, $n \in [m,\infty]$, $p \in [n,\infty]$ and 
$d\in \N$. Then there exists a constant $C>0$ such that for any $0\leqslant S\leqslant T$, 
any $\mathcal{F}_{S}$-measurable random variable $\Xi$ in $\R^d$ and any bounded 
measurable function $f:\mathbb{R}\times\R^d \rightarrow \mathbb{R}$ fulfilling
\begin{enumerate}[label=(\roman*)]
    \item $\EE\left[ \|f(\cdot,\Xi)\|_{\mathcal{C}^1}^2\right]<\infty$;
    \item $\EE\left[ \|f(\cdot,\Xi)\|_{\mathcal{B}_p^\gamma}^n\right]<\infty$,
\end{enumerate}
we have  for any $t\in[S,T]$ that
\begin{equation}\label{eq:regulINT}
    \left\|\EE^S\left[\left|\int_S^t f(B_r,\Xi) \, dr\right|^m\right]^{1/m}\right\|_{L^n} \leqslant C \, \| \|f(\cdot,\Xi)\|_{\mathcal{B}_p^\gamma}\|_{L^n}\, (t-S)^{1+H(\gamma-1/p)}.
\end{equation}
\end{lemma}

\begin{proof} %
In order to show \eqref{eq:regulINT}, we will apply Lemma~\ref{sts}. For $S\leqslant s 
\leqslant t \leqslant T$, let
\begin{align*}
    \mathcal{A}_t:=\int_S^t f(B_r,\Xi) dr \quad \text{and}\quad A_{s,t}:=\EE^s\left[\int_s^t f(B_r,\Xi) dr\right].
\end{align*}
Notice that we have $\EE^s[\delta A_{s,u,t}]=0$, so \eqref{sts1} trivially holds.

In order to establish inequality \eqref{sts2}, we show that 
\begin{align} \label{(4.8)}
    \|\delta A_{s,u,t}\|_{L^n}\leqslant \Gamma_2 (u-S)^{-\alpha_2}(t-s)^{1/2+\varepsilon_2}
\end{align}
holds true for some $\alpha_{2}\in [0,1/2)$ and $\varepsilon_{2}>0$, which is sufficient by the tower property and conditional Jensen's inequality. 
For $u=(s+t)/2$ we have by Minkowski's integral inequality, Jensen's inequality for 
conditional expectation and Lemma~\ref{lem:Cs}\ref{(C.9)} that
\begin{align*}
    \|\delta A_{s,u,t}\|_{L^n}&\leqslant \left\|\EE^s\left[\int_u^t f(B_r,\Xi) dr\right]\right\|_{L^n} + \left\|\EE^u\left[\int_u^t  f(B_r,\Xi) dr\right]\right\|_{L^n}\\
    &\leqslant \int_u^t\left(\|\EE^s f(B_r,\Xi) \|_{L^n}+\|\EE^u f(B_r,\Xi) \|_{L^n}\right) dr\\
    &\leqslant 2 \int_u^t \|\EE^u f(B_r,\Xi) \|_{L^n} dr \\
    &\leqslant C\int_u^t \|\| f(\cdot,\Xi) \|_{\mathcal{B}_p^\gamma}\|_{{L^n}}(r-u)^{H\gamma}(u-S)^{-\frac{1}{2p}}(r-S)^{\tfrac{1-2H}{2p}}dr\\
    &\leqslant C \int_u^t \|\| f(\cdot,\Xi) \|_{\mathcal{B}_p^\gamma}\|_{{L^n}} (r-u)^{H\gamma} (u-S)^{-\frac{H}{p}} dr\\
    &\leqslant C \, \| \| f(\cdot,\Xi) \|_{\mathcal{B}_p^\gamma}\|_{L^n}\,  (t-u)^{1+H\gamma}(u-S)^{-\frac{H}{p}},
\end{align*}
where the penultimate inequality holds true as $r-S\leqslant 2(u-S)$. Hence, we 
have \eqref{(4.8)} for $\varepsilon_2=1/2+H\gamma>0$ and $\alpha_2=H/p<1/2$.

Let $t\in [S,T]$. Let $(\Pi_k)_{k \in \mathbb{N}}$ be a 
sequence of partitions of $[S,t]$ with mesh size converging to zero. For each $k$, denote $\Pi_k=\{t_i^k\}_{i=1}^{N_k}$. By 
Lemma~\ref{lem:Cs}\ref{(C.11)} we have that
\begin{align*}
    \|\mathcal{A}_t-\sum_i A_{t^k_i,t^k_{i+1}}\|_{L^1}&\leqslant\sum_i \int_{t^k_i}^{t^k_{i+1}}\|f(B_r,\Xi)-\EE^{t_i^k} f(B_r,\Xi) \|_{L^1} dr\\
    &\leqslant C \, \|\|f(\cdot,\Xi) \|_{\mathcal{C}^1}\|_{L^2}\, (t-S)\, |\Pi_k|^H \longrightarrow 0.
\end{align*}
Hence \eqref{sts3} holds true.

Applying Lemma~\ref{sts}, we get
\begin{align*}
    \|\EE^S[|\mathcal{A}_t-\mathcal{A}_S|&^m]^{1/m}\|_{L^n}\\
    &\leqslant \|A_{S,t}\|_{L^n} + C \, \| \| f(\cdot,\Xi) \|_{\mathcal{B}_p^\gamma}\|_{L^n}\, \left(\int_S^t(r-S)^{-\frac{2H}{p}} dr\right)^{\frac{1}{2}} (t-S)^{\frac{1}{2}+H\gamma}\\
    &\leqslant \|A_{S,t}\|_{L^n} + C \, \| \| f(\cdot,\Xi) \|_{\mathcal{B}_p^\gamma}\|_{L^n}\, (t-S)^{1+H(\gamma-1/p)}.
\end{align*}
Applying Minkowski's integral inequality and Lemma~\ref{lem:Cs}\ref{(C.10)}, we get that
\begin{align*}
   \|A_{S,t}\|_{L^n}&=\Big\|\EE^S \int_S^t f(B_r,\Xi) dr \Big\|_{L^n}\\
    &\leqslant \int_S^t \|\EE^S f(B_r,\Xi) \|_{L^n} dr\\
    &\leqslant C  \int_S^t \| \| f(\cdot,\Xi) \|_{\mathcal{B}_p^\gamma}\|_{L^n}\, (r-S)^{H(\gamma-1/p)} dr \\
    &\leqslant C \, \| \| f(\cdot,\Xi) \|_{\mathcal{B}_p^\gamma}\|_{L^n}\,  (t-S)^{1+H(\gamma-1/p)}. 
\end{align*}
Hence
\begin{align*}
    \|\EE^S[|\mathcal{A}_t-\mathcal{A}_S|^m]^{1/m}\|_{L^n} \leqslant C\, \| \|f(\cdot,\Xi) \|_{\mathcal{B}_p^\gamma}\|_{L^n}\, (t-S)^{1+H(\gamma-1/p)}.
\end{align*}
\end{proof}

\begin{corollary} \label{cor:6.2.2}
Let $\gamma \in (-1/(2H),0)$, $m \in [2,\infty)$ and $p \in [m,\infty]$. Let $\lambda \in [0,1]$ and assume that $\gamma>-1/(2H)+\lambda$. There exists $C>0$
 such that for any $f \in \mathcal{C}_b^\infty(\mathbb{R}) \cap \mathcal{B}_p^\gamma$, any $0\leqslant s \leqslant t \leqslant T$ and any $\mathcal{F}_s$-measurable random variables $\kappa_1,\kappa_2 \in L^m$, one has
\begin{align}
\Big\|\int_s^t (f(B_r+&\kappa_1)-f(B_r+\kappa_2)) \, dr \Big\|_{L^m} \nonumber\\
&\leqslant C \|f\|_{\mathcal{B}_p^\gamma}\|\kappa_1-\kappa_2\|_{L^m}^\lambda (t-s)^{1+H(\gamma-\lambda-1/p)}. \label{eq:6.2.2}
\end{align}
\end{corollary}

\begin{proof}
We aim to apply Lemma~\ref{regulINT} for the function ${(z,(x_{1},x_{2})) \mapsto f(z+x_1)-f(z+x_2)}$ with $m=n$ and $\Xi = (\kappa_{1},\kappa_{2})$.
By Lemma~\ref{A.2}\ref{A.5} and Jensen's inequality, we have that 
\begin{align*}
    \|\|f(\cdot + \kappa_1)-f(\cdot + \kappa_2)\|_{\mathcal{B}_p^{\gamma-\lambda}}\|_{L^m}&\leqslant C\|f\|_{\mathcal{B}_p^\gamma}\||\kappa_1-\kappa_2|^\lambda\|_{L^m}.\\
    &\leqslant C\|f\|_{\mathcal{B}_p^\gamma}\|\kappa_1-\kappa_2\|_{L^m}^\lambda. 
\end{align*}
Moreover 
\begin{align*}
    \|f(\cdot +\kappa_1)-f(\cdot+\kappa_2)\|_{\mathcal{C}^1}\leqslant 2 \|f\|_{\mathcal{C}^1}<\infty.
\end{align*}
Therefore all assumptions of Lemma~\ref{regulINT} are fulfilled and the result follows.
\end{proof}

\begin{corollary} \label{cor:6.2.3}
Let $\gamma \in (-1/(2H),0)$, $m \in [2,\infty)$ and $p \in [m,\infty]$. Let $\lambda, \lambda_1, \lambda_2 \in (0,1]$ and assume that $\gamma>-1/(2H)+\lambda$ and $\gamma>-1/(2H)+\lambda_1+\lambda_2$. There exists a constant $C>0$
 such that for any $f \in \mathcal{C}_b^\infty(\mathbb{R}) \cap \mathcal{B}_p^\gamma$, any $0\leqslant s \leqslant u \leqslant t \leqslant T$, any $\mathcal{F}_s$-measurable random variables $\kappa_1,\kappa_2 \in L^m$ and any $\mathcal{F}_u$-measurable random variables $\kappa_3, \kappa_4 \in L^m$, we have that
\begin{align}
    \Big\|\int_u^t &\left(f(B_r+\kappa_1)-f(B_r+\kappa_2)-f(B_r+\kappa_3)+f(B_r+\kappa_4) \right)dr \Big\|_{L^m} \nonumber\\
    &\leqslant C \|f\|_{\mathcal{B}_p^\gamma}\|\EE^s[|\kappa_1-\kappa_3|^m]^{1/m}\|_{L^\infty}^{\lambda_2}\|\kappa_1-\kappa_2\|_{L^m}^{\lambda_1}(t-u)^{1+H(\gamma-\lambda_1-\lambda_2-1/p)} \nonumber\\
    &\quad + C\|f\|_{\mathcal{B}_p^\gamma} \|\kappa_1-\kappa_2 - \kappa_3 +\kappa_4\|_{L^m}^\lambda (t-u)^{1+H(\gamma-\lambda-1/p)}.\label{eq:cor}
\end{align}
\end{corollary}

\begin{proof}
Let $h: \mathbb{R} \times \R^4 \rightarrow \mathbb{R}$ be defined by 
\begin{align*}
    h: (z,(x_{1},x_{2},x_{3},x_{4})) \mapsto f(z+x_1)-f(z+x_2)-f(z+x_3)+ f(z+x_3+x_2-x_1).
\end{align*}
and $g: \mathbb{R} \times \R^4 \rightarrow \mathbb{R}$ by 
\begin{align*}
   g : (z,(x_{1},x_{2},x_{3},x_{4})) \mapsto f(z+x_4)-f(z+x_3+x_2-x_1).
\end{align*}
Let $\Xi = (\kappa_{1},\kappa_{2},\kappa_{3},\kappa_{4})$. Hence the integrand on the left hand side of \eqref{eq:cor} is 
$h(B_r,\Xi)+g(B_r ,\Xi)$. The proof will consist in applying Lemma 
\ref{regulINT} to the integral of $h$, and Corollary~\ref{cor:6.2.2} to the integral of $g$.
By Lemma~\ref{A.2}\ref{A.6}, we have
\begin{align} \label{normh}
    \|h(\cdot,\Xi) \|_{\mathcal{B}_p^{\gamma-\lambda_1-\lambda_2}}\leqslant \|f\|_{\mathcal{B}_p^\gamma}|\kappa_1-\kappa_2|^{\lambda_1}|\kappa_1-\kappa_3|^{\lambda_2}.
\end{align}
Using that $\kappa_1-\kappa_2$ is $\mathcal{F}_s$-measurable, by the tower property, 
Jensen's inequality for conditional expectation and \eqref{normh}, we get that
\begin{align*}
    \EE\|h(\cdot,\Xi)\|_{\mathcal{B}_p^{\gamma-\lambda_1-\lambda_2}}^m &\leqslant \|f\|_{\mathcal{B}_p^\gamma}\EE\left[|\kappa_1-\kappa_2|^{m\lambda_1}\EE^s[|\kappa_1-\kappa_3|^{m\lambda_2}]\right]\\
    &\leqslant \|f\|_{\mathcal{B}_p^\gamma}\EE\left[|\kappa_1-\kappa_2|^{m\lambda_1}\right]\left\|\EE^s[|\kappa_1-\kappa_3|^{m\lambda_2}]\right\|_{L^\infty}\\
    &\leqslant \|f\|_{\mathcal{B}_p^\gamma} \|\kappa_1-\kappa_2\|_{L^m}^{m\lambda_1}\|\EE^s[|\kappa_1-\kappa_3|^m]^{1/m}\|_{L^\infty}^{m\lambda_2}.
\end{align*}
Furthermore
\begin{align*}
\|h(\cdot,\Xi)\|_{\mathcal{C}^1}\leqslant 4 \|f\|_{\mathcal{C}^1}.
\end{align*}
Hence, we get the result by applying Lemma~\ref{regulINT} for $S=u$ to the integral of $h$
 and Corollary~\ref{cor:6.2.2} to the integral of $g$.
\end{proof}

As another consequence of the stochastic sewing Lemma, we finally prove 
Lemma~\ref{lem:mainregularisation}.
\begin{proof}[Proof of Lemma \ref{lem:mainregularisation}]
We assume that $[\psi]_{\COO{\alpha}{m}{n}{[s,t]}}<\infty$, otherwise 
\eqref{eq:lem5.2} trivially holds. For $(\tilde{s},{\tilde{t}}) \in \Delta_{[s,t]}$, let
\begin{align} \label{eq:A}
    A_{\tilde{s},{\tilde{t}}}:=\int_{\tilde{s}}^{\tilde{t}} f(B_r+\psi_{\tilde{s}}) dr \ \text{and } \mathcal{A}_t:=\int_s^{\tilde{t}} f(B_r+\psi_r) dr.
\end{align}
In the following, we check the necessary conditions in order to apply Lemma~\ref{sts}. In order to show that \eqref{sts1} and \eqref{sts2} hold true with $\alpha_2=0$, $\varepsilon_1=H(\gamma-1/p-1)+\alpha>0$ and $\varepsilon_2=1/2+H(\gamma-1/p)>0$, we show that there exists a constant $C>0$ independent of $s,t,\tilde{s}$ and $\tilde{t}$ such that
\begin{enumerate}[label=(\roman*)]
\item \label{en:(1)}$\|\EE^{\tilde{s}} \delta A_{{\tilde{s}},u,{\tilde{t}}}\|_{L^n}\leqslant C \|f\|_{\mathcal{B}_p^\gamma} [\psi]_{\COO{\alpha}{m}{n}{[s,t]}} ({\tilde{t}}-{\tilde{s}})^{1+H(\gamma-1-1/p)+\alpha}.$
\item \label{en:(2)}$\|\EE^{s}[|\delta A_{{\tilde{s}},u,{\tilde{t}}}|^m]^{1/m}\|_{L^n}\leqslant C \|f\|_{\mathcal{B}_p^\gamma} ({\tilde{t}}-{\tilde{s}})^{1+H(\gamma-1/p)}.$
\item \label{en:(3)}If \ref{en:(1)} and \ref{en:(2)} are verified, then \eqref{sts3} gives the 
convergence in probability of $\sum_{i=0}^{N_n-1} A_{t^k_i,t^k_{i+1}}$ along any sequence of 
partitions $\Pi_k=\{t_i^k\}_{i=0}^{N_k}$ of $[s,\tilde{t}]$ with mesh converging to $0$. We will show 
that the limit is the process $\mathcal{A}$ given in \eqref{eq:A}.
\end{enumerate}

Assume for now that \ref{en:(1)}, \ref{en:(2)} and \ref{en:(3)} hold. Applying Lemma~\ref{sts}, we obtain that
\begin{align*}
    \bigg\|\EE^s \bigg[\bigg| \int_{\tilde{s}}^{\tilde{t}} f(B_r+\psi_r) dr\bigg|^m\bigg]^{1/m}\bigg\|_{L^n} \leqslant &C \|f\|_{\mathcal{B}_p^\gamma}({\tilde{t}}-{\tilde{s}})^{1+H(\gamma-1/p)}\\
    &\quad +C \|f\|_{\mathcal{B}_p^\gamma}[\psi]_{\COO{\alpha}{m}{n}{[s,t]}}({\tilde{t}}-{\tilde{s}})^{1+H(\gamma-1-1/p)+\alpha}\\
    & \quad +\|\EE^s[|A_{{\tilde{s}},{\tilde{t}}}|^m]^{1/m}\|_{L^n}.
\end{align*}
In \eqref{eq:Ast} we will see that $\|\EE^{s}[|A_{{\tilde{s}},{\tilde{t}}}|^m]^{1/m}\|_{L^n}\leqslant C \|f\|_{\mathcal{B}_p^\gamma}({\tilde{t}}-{\tilde{s}})^{1+H(\gamma-1/p)}$. Then, choosing $(\tilde{s},\tilde{t})=(s,t)$ we get \eqref{eq:lem5.2}.

We now check that the conditions \ref{en:(1)}, \ref{en:(2)} and \ref{en:(3)} actually hold.

Proof of \ref{en:(1)}: For 
$s\leqslant {\tilde{s}} \leqslant u \leqslant {\tilde{t}} \leqslant t$, 
we have
\begin{equation*} %
    \delta A_{{\tilde{s}},u,{\tilde{t}}}= \int_u^{\tilde{t}} f(B_{r}+\psi_{\tilde{s}})-f(B_{r}+\psi_u) dr.
\end{equation*}
Hence, by the tower property of conditional expectation and Fubini's 
Theorem, we get 
\begin{align*}
    |\EE^{\tilde{s}} \delta A_{{\tilde{s}},u,{\tilde{t}}}|&=\left|\EE^{\tilde{s}} \int_u^{\tilde{t}} \EE^u 
    [f(B_{r}+\psi_{\tilde{s}})-f(B_{r}+\psi_u)]
    dr \right|.
\end{align*}
Now using Lemma~\ref{lem:Cs}\ref{(C.10)} with the $\mathcal{F}_{u}$-measurable variable $\Xi=(\psi_{\tilde{s}},\psi_{u})$, Lemma~\ref{A.2}\ref{A.5} for 
$\alpha=1$ and again Fubini's Lemma, we obtain that
\begin{align*}
    \Big|\EE^{\tilde{s}} \int_u^{\tilde{t}} \EE^u 
    [f(B_{r}+\psi_{\tilde{s}})-f(B_{r}&+\psi_u)]
    dr \Big|\\
    &\leqslant \EE^{\tilde{s}} \int_u^{\tilde{t}} \|
    f(\cdot+\psi_{\tilde{s}})-f(\cdot+\psi_u)
    \|_{\mathcal{B}_p^{\gamma-1}}(r-u)^{H(\gamma-1-1/p)} dr\\
    &\leqslant C \|f\|_{\mathcal{B}_p^\gamma} \int_u^{\tilde{t}} \EE^{\tilde{s}} [|\psi_u-\psi_{\tilde{s}}|] (r-u)^{H(\gamma-1-1/p)} dr.
\end{align*}
Hence we get
\begin{equation} \label{Mink}
    \|\EE^{\tilde{s}} \delta A_{{\tilde{s}},u,{\tilde{t}}}\|_{L^n}\leqslant C \|f\|_{\mathcal{B}_p^\gamma} \int_u^{\tilde{t}} \|\EE^{\tilde{s}} |\psi_u-\psi_{\tilde{s}}|\|_{L^n} (r-u)^{H(\gamma-1-1/p)} dr.
\end{equation}
By the conditional Jensen's inequality, we have  
\begin{equation*}
    \|\EE^{\tilde{s}}[|\psi_u-\psi_{\tilde{s}}|]\|_{L^n}\leqslant [\psi]_{\COO{\alpha}{m}{n}{[s,t]}}(u-{\tilde{s}})^\alpha.
\end{equation*}
Combining this with equation \eqref{Mink}, we get 
\begin{equation*}
    \|\EE^{\tilde{s}} \delta A_{{\tilde{s}},u,{\tilde{t}}}\|_{L^n}\leqslant C \|f\|_{\mathcal{B}_p^\gamma} [\psi]_{\COO{\alpha}{m}{n}{[s,t]}} ({\tilde{t}}-{\tilde{s}})^{1+H(\gamma-1-1/p)+\alpha}.
\end{equation*}
Proof of \ref{en:(2)}: Note that by Jensen's inequality for conditional expectation, tower property and Lemma~\ref{regulINT} we have that
\begin{align} \label{eq:Ast}
\| \EE^{s}[|A_{\tilde{s},\tilde{t}}|^m]^{1/m}\|_{L^n}=\Big(\EE\big[(\EE^{s}\EE^{\tilde{s}}|A_{\tilde{s},\tilde{t}}|^m)^{n/m}\big]\Big)^{1/n}
&\leqslant \Big(\EE\, \EE^{s}\big[\EE^{\tilde{s}}[|A_{\tilde{s},\tilde{t}}|^m]^{n/m}\big]\Big)^{1/n}\nonumber\\
&=\Big(\EE\big[ \EE^{\tilde{s}}[|A_{\tilde{s},\tilde{t}}|^m]^{n/m}\big]\Big)^{1/n}\nonumber\\
&\leqslant C\|f\|_{\mathcal{B}^{\gamma}_p}(\tilde{t}-\tilde{s})^{1+H(\gamma-1/p)}.
\end{align}
\red{After similarly controlling $\|\EE^{s}[|A_{{\tilde{s}},u}|^m]^{1/m}\|_{L^n}$ and $\|\EE^{s}[|A_{u,\tilde{t}}|^m]^{1/m}\|_{L^n}$, we get}
\begin{equation*}
    \|\EE^{s}[|\delta A_{{\tilde{s}},u,{\tilde{t}}}|^m]^{1/m}\|_{L^n}\leqslant C \|f\|_{\mathcal{B}_p^\gamma} ({\tilde{t}}-{\tilde{s}})^{1+H(\gamma-1/p)}.
\end{equation*}
\red{The proof of \ref{en:(3)} can be done by similar arguments as for \ref{enum:2ALT} in the proof of Proposition~\ref{prop:regularityALT}.}
\end{proof}

\end{appendices}

\providecommand{\bysame}{\leavevmode\hbox to3em{\hrulefill}\thinspace}
\providecommand{\MR}{\relax\ifhmode\unskip\space\fi MR }
\providecommand{\MRhref}[2]{%
	\href{http://www.ams.org/mathscinet-getitem?mr=#1}{#2}
}
\providecommand{\href}[2]{#2}

\begin{acks}
	 L.A. acknowledges the support of the Labex de Math\'ematique Hadamard. L.A. and A.R. acknowledge the support of the SIMALIN project ANR-19-CE40-0016 from the French National Research Agency.\\
	We would like to thank the anonymous referees for the careful reading and numerous 
	helpful suggestions to improve the manuscript. We also thank Lucio Galeati for discussions that led to correct an error in an estimate of Section 7.
	\end{acks}

\end{document}